\documentclass[11pt, oneside]{article}

\usepackage[T1]{fontenc}
\usepackage[utf8]{inputenc}
\usepackage{amstext,amsmath, amssymb, amsthm, mathrsfs, mathtools, thm-restate}
\usepackage[a4paper, total={6.2in, 9.1in}, left={1.1in}, right={1.1in}, top={1.1in},bottom={1.1in}, footskip={.5in}]{geometry}
\usepackage{titlesec}
\usepackage{authblk}
\usepackage{enumerate}
\usepackage[dvipsnames]{xcolor}
\usepackage{comment}
\usepackage[colorlinks=true,
    linkcolor=black,
    citecolor=RoyalBlue,
    urlcolor=cyan]{hyperref}
\usepackage[capitalise]{cleveref}
\usepackage{etoolbox}
\usepackage{stmaryrd}
\usepackage{thmtools}

\usepackage{setspace}
\allowdisplaybreaks

\makeatletter
\newcommand\@authorinfo{}
\renewcommand{\author}[2][]{%
  \g@addto@macro\@author{#2\and}%
  \g@addto@macro\@authorinfo{%
    \par\bigskip
    \noindent{\scshape #2}\par
  }%
}
\newcommand{\address}[1]{%
  \g@addto@macro\@authorinfo{#1\par}%
}
\newcommand{\email}[1]{%
  \g@addto@macro\@authorinfo{\texttt{#1}\par}%
}
\AtEndDocument{%
  \par\bigskip
  \@authorinfo
}
\makeatother

\makeatletter
\renewcommand{\maketitle}{%
    \vspace*{1em}
    \begin{center}
        {\normalfont\bfseries\MakeUppercase{\@title}\par}%
        \vskip 1em
        {\normalfont\small \@date \par}%
    \end{center}
}
\makeatother

\renewenvironment{abstract}{%
    \noindent
    \begin{list}{}{%
        \setlength{\leftmargin}{2.5em}%
        \setlength{\rightmargin}{2.5em}%
        }
        \item\relax
        \textsc{Abstract. }\ignorespaces
        }{%
  \end{list}
}

\titleformat{\section}
  {\normalfont\scshape\centering}
  {\thesection.}{0.5em}{}

\titleformat{\subsection}[runin]
  {\normalfont\bfseries}
  {\thesubsection.}{0.5em}{}[.]

\usepackage[backend=bibtex]{biblatex}
\renewbibmacro{in:}{}
\DeclareFieldFormat{pages}{#1}

\usepackage{amsthm}

\newtheorem{lettertheorem}{Theorem}
\newtheorem{question}{}

\newtheorem{lettertheoremprime}{Theorem}

\newtheorem{theorem}{Theorem}[section]
\newtheorem{lemma}[theorem]{Lemma}
\newtheorem{proposition}[theorem]{Proposition}
\newtheorem{corollary}[theorem]{Corollary}
\newtheorem{definition}[theorem]{Definition}

\theoremstyle{definition} 
\newtheorem{remark}[theorem]{Remark}

\numberwithin{equation}{section}

\DeclareMathOperator{\nR}{\mathbb{R}}

\DeclareMathOperator{\nN}{\mathbb{N}}

\DeclareMathOperator{\nS}{\mathbb{S}}

\DeclareMathOperator{\GL}{GL}

\DeclareMathOperator{\tr}{tr}

\DeclareMathOperator{\Ker}{Ker}

\DeclareMathOperator{\supp}{supp}
\DeclareMathOperator{\Id}{Id}

\DeclareMathOperator{\Ric}{Ric}
\DeclareMathOperator{\R}{R}

\DeclareMathOperator{\vol}{vol}

\DeclareMathOperator{\BO}{\mathrm{O}}

\DeclareMathOperator{\G}{G}
\DeclareMathOperator{\Lg}{\mathscr L_{\textit{g}}} 
\DeclareMathOperator{\LL}{\mathscr L} 
\DeclareMathOperator{\p}{\mathrm{p}}
\DeclareMathOperator{\q}{\mathrm{q}}
\DeclareMathOperator{\h}{\mathrm{h}}
\DeclareMathOperator{\m}{\mathrm{m}}

\newcommand{\repeatedletter}{}

\newtheoremstyle{repeatletterstyle}
  {\topsep}{\topsep}
  {\itshape}{}
  {\bfseries}{.}{0.5em}
  {Theorem~\repeatedletter}

\theoremstyle{repeatletterstyle}
\newtheorem*{repeatedlettertheorem}{}
\theoremstyle{plain}

\newenvironment{replettertheorem}[1]{\renewcommand{\repeatedletter}{\ref{#1}}
\begin{repeatedlettertheorem}
}{\end{repeatedlettertheorem}}

\makeatletter
\patchcmd{\tableofcontents}
  {\@starttoc{toc}}
  {\vspace{-0.7cm}\@starttoc{toc}}
  {}{}
\makeatother

\title{Conformal Green functions and Yamabe metrics \\ of Sobolev regularity}

\date{}

\author{Rodrigo Avalos} 
\address{\noindent Fachbereich Mathematik, Eberhard Karls Universit\"at T\"ubingen, Auf der Morgenstelle 10, 72076 T\"{u}bingen, Germany
\\
\textit{Current address}: Institut für Mathematik, Universität Rostock, Ulmenstraße 69, Haus 3, 18057 Rostock, Germany}
\email{rodrigo.avalos@uni-rostock.de}

\author{Albachiara Cogo}
\address{\noindent Fachbereich Mathematik, Eberhard Karls Universit\"at T\"ubingen, Auf der Morgenstelle 10, 72076 T\"{u}bingen, Germany
\\
\textit{Current address}: Centro di Ricerca Matematica Ennio De Giorgi, Scuola Normale Superiore di Pisa, Piazza dei Cavalieri 3, 56126, Pisa, Italy}
\email{albachiara.cogo@sns.it}

\author{Andoni Royo Abrego}
\address{\noindent Fachbereich Mathematik, Eberhard Karls Universit\"at T\"ubingen, Auf der Morgenstelle 10, 72076 T\"{u}bingen, Germany
\\
\textit{Current address}: Institut Élie Cartan de Lorraine, Université de Lorraine, Boulevard des Aiguillettes
54506, Nancy }
\email{andoni.royo-abrego@univ-lorraine.fr}

\begin{document}

\maketitle

\begin{center}
    {\footnotesize \MakeUppercase{Rodrigo Avalos, Albachiara Cogo, and Andoni Royo Abrego}}
\end{center}
\begin{abstract}
    We provide a full resolution of the Yamabe problem on closed 3-manifolds for Riemannian metrics of Sobolev class $W^{2,q}$ with $q > 3$. This requires developing an elliptic theory for the conformal Laplacian for rough metrics and establishing existence, regularity and a delicate blow-up analysis for its Green function. Most of the analytical work is carried out in dimensions $n \geq 3$ and for $W^{2,q}$ Riemannian metrics with $q>\tfrac{n}{2}$ and should be of independent interest.
\end{abstract}

\vspace{0.5cm}

\tableofcontents

\hfill

\pagestyle{plain}

\section[{\textbf{Introduction}}]{Introduction}
\label{Introduction}

The study of scalar curvature for rough Riemannian metrics has attracted considerable attention in recent years, partly motivated by Gromov's program concerning singular spaces with positive scalar curvature \cite{GromovC0Scalar,GromovLectures,GromovConvexPolytopes}, Huisken's isoperimetric notion of mass \cite{HuiskenMass1,HuiskenMass3,ChodoshEichmairShiYu,JauregiLee}, and Ricci flow smoothing \cite{BamlerGromovv,Burkhardt-GuimADM,ChuLee}. A common goal of all these approaches is to understand $C^0$ aspects of scalar curvature. As one of the central problems in scalar curvature geometry, it is thus natural to study the Yamabe problem for rough Riemannian metrics. In the present work, we are concerned with the following:
\begin{question} 
    \label{main question}
    Given a $W^{2,q}$-Riemannian metric $g$ for some $q > \tfrac{n}{2}$ on a smooth closed manifold $M$ of dimension $n \geq 3$, is there a metric in the conformal class
\begin{align*}
    [\,g\,]_{W^{2,q}} \doteq \Bigl\{u^{\frac{4}{n-2}}g \: : \: u\in W^{2,q}(M)\,, \: u>0 \Bigr\}
\end{align*}
of constant scalar curvature?
\end{question}
\noindent
Note that the space $W^{2,q}(M)$ is an algebra under multiplication, satisfies the embedding $W^{2,q}(M) \hookrightarrow C^{0,2-\frac{n}{q}}(M)$ for $q > \tfrac{n}{2}$ and provides a well-defined conformal class. The above problem is equivalent to finding positive $W^{2,q}(M)$ solutions to the semilinear elliptic equation with critical exponent
\begin{equation}
    \label{eq: Yamabe equation}
    \Lg u \doteq -a_n \Delta_g u + \R_g u = \lambda \, u^{2^*-1} \,, \qquad a_n \doteq 4 \, \frac{n-1}{n-2} \,,
\end{equation}
where $\Delta_g$ and $\R_g$ are the Laplace--Beltrami operator and the scalar curvature of $g$, $\lambda$ is a real number and $2^* \doteq \tfrac{2n}{n-2}$ is the 
limiting Sobolev exponent for the embedding $W^{1,2}(M) \hookrightarrow L^{2^*}(M)$. Solutions to \eqref{eq: Yamabe equation} are critical points of the functional 
\begin{equation}
    \label{eq: Q energy}
    Q_g(u) \doteq \frac{\int_M \bigl(a_n|\nabla u|^2_g + \R_g u^2\bigr) \, d\mu_g}{\left(\int_M|u|^{2^*} d\mu_g\right)^{\frac{2}{2^*}}}
\end{equation}
and the associated conformally invariant infimum
\begin{equation*}
    \lambda(M, g) \doteq \inf \Big\{Q_g(u) \; : \; u \in W^{1,2}(M) \,, \; u \neq 0\Big\} 
\end{equation*}
is the so-called \textit{Yamabe invariant}\footnote{Beware of the convention of some authors to use the term \textit{Yamabe invariant} to refer to the topological invariant obtained from $\lambda(M, g)$ by taking the supremum over all conformal classes.}. If $g$ is smooth, the combination of the celebrated work of H. Yamabe \cite{Yamabe}, N. Trudinger \cite{Trudinger}, T. Aubin \cite{Aubin1} and R. Schoen \cite{SchoenYamabe} asserts that such a conformal metric always exists and is smooth --see also \cite{Lee-Parker,SchoenYauBook,AubinBook} for a panoramic treatise of the Yamabe problem--. 
In particular, such a metric is given by 
$u^{\frac{4}{n-2}}g$
for $u$ solving \eqref{eq: Yamabe equation} and its  constant scalar curvature is equal to 
$ \lambda = u^{1-2^*} \Lg u $ 
and has the same sign of $\lambda(M, g)$.
When $\lambda(M,g) \leq 0$ (Yamabe non-positive), solving the above problem follows from 
rather standard 
variational techniques, whereas the $\lambda(M,g) > 0$ (Yamabe positive) case, in addition, requires understanding the blow-up behavior of the Green function of the operator $\Lg$, which finds deep connections with conformal geometry and general relativity. An alternative proof was established by Brendle \cite{BrendleYamabeFlow1,BrendleYamabeFlow2} using a heat flow approach, and Bahri \cite{Bahri93,Brezis-Bahri} using critical points at infinity.

Beyond the classical setting, equation \eqref{eq: Yamabe equation} has extensively been studied in the context of the \textit{singular Yamabe problem}, where one seeks for a complete conformal metric of constant scalar curvature on $M \setminus \Lambda$, where $\Lambda \subset M$ is a closed subset such as a collection of points or a hypersurface --see \cite{MazzeoPollackUhlenbeck, CaffarelliGidasSpruck, HanLiLi, MarquesSingular, MazzeoPacard, XiongZhang,Finn} and references therein--.

In a different direction, the Yamabe problem has also been studied on singular spaces, such as smooth manifolds with conical (or cone-edge) singularities \cite{AkutagawaBotvinnikCylindrical,AkutagawaMondello,FregugliaMalchiodi,FregugliaMalchiodiMalizia}, orbifolds \cite{AkutagawaBotvinnikOrbifold, AkutagawaOrbifold,ViaclovskyOrbifoldYamabe,ViaclovskyJu} and more general stratified spaces \cite{AkutagawaCarronMazzeo,Mondello2017}. In contrast with the smooth setting, there are examples of conformal classes on these singular spaces that do not admit a minimizer of the Yamabe quotient, see \cite{ViaclovskyOrbifoldYamabe,AkutagawaMondello} for instance.

Complementary to the above results, we are interested in the Yamabe problem for metrics that can be rough everywhere, although with less severe singularities. The relevance of Sobolev-type metrics to model physical phenomena has been highlighted in \cite{Maxwell0,Maxwell1,MaxwellFarCMC,Holst1,Holst-LichCompactWithBoundary} in the context of the \textit{conformal method} of solving the Einstein constraint equations. This problem  
determines the admissible initial data sets for the evolution problem in general relativity, where realistic matter distributions are not smooth. Analytically, it consists in finding solutions to the generalized Yamabe-type equation called \emph{Lichnerowicz equation} for rough metrics, which requires some a priory understanding of the Yamabe problem.

In the aforementioned work \cite[Proposition 3.3 and Proposition 4.3]{Maxwell1}, D. Maxwell solves the Yamabe problem in the non-positive case for $H^{s}$-metrics with $s>\frac{n}{2}$ via a monotone iteration scheme using barriers. This threshold has recently been lowered to $W^{1,p}$-metrics with $p>n$ by H. Zhang \cite{ZhangW1pAubin}, via extending to this class \emph{Aubin's theorem} \cite{Aubin1}, which asserts the existence of a constant scalar curvature metric provided that $\lambda(M,g) < \lambda(\nS^n)$ --see \cref{The Aubin--Trudinger--Yamabe Theorem} for details--. In contrast, the positive case appears as a much harder problem, fully open even for higher regularity classes, such as $C^{1,\alpha}$-metrics with extra Sobolev control. In this setting, the case of $W^{2,q}$-metrics with $q>n$ has been studied by Madani \cite{Madani} under the strong assumption that $g$ is smooth in a neighbourhood of a point. This last condition allows a large amount of the standard analysis for smooth metrics, in particular the behaviour of the conformal Green function, to remain valid without any alterations.

\addtocontents{toc}{\protect\setcounter{tocdepth}{1}}
\subsection{Main results}
\addtocontents{toc}{\protect\setcounter{tocdepth}{2}}
In this work we provide an existence, regularity and blow-up analysis of the conformal Green function for metrics of Sobolev regularity (\cref{theorem C}) and use it, among other tools, to give a positive answer to the above \cref{main question} in the case of $3$-dimensional manifolds, provided that $q > 3$. Moreover, if the original metric is more regular, so is the conformally deformed one:
\begin{lettertheorem}
    \label{theorem A}
    Let $M$ be an orientable, smooth, closed 3-manifold and consider a $W^{k,q}$-Riemannian metric $g$ with $k \geq 2$ and $q > 3$. Then, there exists a $W^{k,q}$-Riemannian metric conformal to $g$ of constant scalar curvature.
\end{lettertheorem}
\noindent
We remark that the orientability assumption is due to the usage of a version of the positive mass theorem (PMT) for rough metrics due to D. Lee and P. LeFloch \cite{LeeLeF}, which works with asymptotically Euclidean (AE) spin manifolds with one end --see \cref{The Positive Mass Theorem} for details--. Should their result hold true for either non-spin manifolds or manifolds with several ends\footnote{This is by virtue of a simple argument working with the orientable double cover.}, one could immediately drop the orientability assumption in \cref{theorem A}. It is also worth mentioning that orientability, as well as the dimension restriction and the $q > 3$ hypothesis, are exclusive to the positive Yamabe case, due to its much more intricate nature. In fact, in the non-positive Yamabe case we establish: 
\begin{lettertheoremprime}
    \label{theorem B}
     Let $M$ be a smooth, closed manifold of dimension $n \geq 3$ and consider a $W^{k,q}$-Riemannian metric $g$ with $k \geq 2$ and $q > \tfrac{n}{2}$. Suppose that $\lambda(M,g) \leq 0$. Then, there exists a $W^{k,q}$-Riemannian metric conformal to $g$ of constant scalar curvature.
\end{lettertheoremprime}
\noindent
Despite the existence of a $W^{1,q}$-metric with $q > n$ of constant scalar curvature in a non-positive Yamabe class is implied by \cite{ZhangW1pAubin}, we are able to bootstrap the regularity of the conformal metric to that of the the original one. Moreover, our proof of Aubin's theorem becomes remarkably simpler by considering a slightly more regular class of metrics (which avoids the need to treat the scalar curvature distributionally) and by exploiting the conformal nature of the problem in recognizing a good \textit{conformal gauge}. More precisely, we move to a conformal metric of continuous scalar curvature by solving the eigenvalue problem for $\Lg$. For details, see the Yamabe classification of \cref{cor: Yamabe classification} based on \cite[Proposition 3.3]{Maxwell1} and \cref{rmk: conformal gauge}.

The strategy to prove \cref{theorem A} roughly follows the approach in \cite{Lee-Parker} for smooth metrics: Aubin's theorem reduces the problem to constructing a positive test function $\psi \in W^{1,2}(M)$ such that $Q_g(\psi) < \lambda(\nS^3)$, which is achieved by decompactifying the original manifold into an asymptotically Euclidean one via the Green function associated with $\Lg$ and using the so-called \textit{Aubin bubbles}. The PMT then plays a key role in order to obtain a desirable sign on the leading non-singular term of the conformal Green function whenever $(M,g)$ is not conformally diffeomorphic to $\nS^3$. While we establish existence of minimizers of \eqref{eq: Q energy}, it is worth noting that looking for higher energy solutions via the alternative strategy of Bahri-Brezis \cite{Brezis-Bahri} could bypass the need of a low-regularity PMT. Such strategy would regardless require our analysis of the conformal Green function in the lower regularity setting --see \cref{theorem C}--.

Let us briefly highlight the main difficulties one encounters in the case of rough metrics. First, the operator $\Lg$ has rough coefficients and consequently, we need to develop an elliptic theory suited to our problem.  
Second, standard versions of the PMT are not suited for our setting and we will instead have to work with a finer version of D. Lee and P. LeFloch \cite{LeeLeF}.
We remark that we need to track down the regularity of the isometry between the AE manifold and $\mathbb{E}^3$ involved in the
rigidity claim of this PMT more carefully than what directly comes from \cite[Theorem 1.1]{LeeLeF} --see \cref{The Positive Mass Theorem} for more details--.
Third, regarding the actual applicability of such PMT, it is crucial to understand the blow-up behavior of the conformal Green function near the pole and the consequent control on the decay of the metric for the decompactified manifold, performed with respect to an appropriate choice of coordinates. 
Among other 
complications, 
it is not possible to employ the powerful properties of  \textit{conformal normal coordinates} to get a refined control of the blow-up behavior of the conformal Green function, as in the smooth setting \cite[Section 6]{Lee-Parker}. The construction of such coordinates, as in \cite[Section 5]{Lee-Parker}, consists in considering normal coordinates for some conformal metric selected to cancel successive higher-order covariant derivatives of the Ricci tensor
at the selected point, which are pointwise not well-defined in our rough setting. 
In addition, even standard normal coordinates constructed via the exponential map require $C^{1,1}$ regularity of the metric, from which one should anticipate that the analysis of the conformal Green function will face important conceptual and technical difficulties.

Relying, among others, on the elliptic theory for $\Lg$ with rough coefficients and on the construction of harmonic and normal coordinates for metrics of low-regularity, 
the following theorem addresses the existence, regularity and blow-up analysis of a conformal Green function, which may be of interest elsewhere, as it extends previous results on the Green function of Schrödinger-type operators.
\begin{lettertheorem}
    \label{theorem C}
    Let $M$ be a smooth, closed manifold of dimension $n \geq 3$ and consider a $W^{2,q}$-Riemannian metric $g$ with $q > \tfrac{n}{2}$ and $\lambda(M, g) > 0$. Then, for each point $p \in M$, the conformal Laplacian $\Lg$ admits a unique, positive Green function $G_{p} \in W_{loc}^{2, q}\bigl(M\setminus\{p\}\bigr)$ centered at $p$. Moreover, there exists a conformal harmonic chart $(U,x^i)$ around $p$ such that
    \begin{equation}
        \label{CGF_Expansion}
         G_{p}(x) = \frac{B}{|x|^{n-2}}  + h(x) 
    \end{equation}
    for some $B>0$ and a function $h \in W^{2,r}(U)$ satisfying
   \begin{equation}
        \label{h_decayIntro}
        h(x) = A + \BO\bigl(|x|^{2-\frac{n}{r}}\bigr) \quad \text{as} \quad |x| \to 0
    \end{equation}
    for each $1 \leq r < \tfrac{nq}{q(n-2) + n}$ and some $A \in \nR$. If $q > n$, there additionally holds
    \begin{equation}
        \label{h_decay_qnIntro}
        h(x) = A + \BO_1\bigl(|x|^{2-\frac{n}{r}}\bigr) \quad \text{as} \quad |x| \to 0
    \end{equation}
    for each $1 \leq  r < \tfrac{nq}{q(n-2) + n}$ and harmonic coordinates can be substituted by normal coordinates.
\end{lettertheorem}

\noindent
The \textit{conformal harmonic chart} of \cref{theorem C} refers to harmonic coordinates of the aforementioned good conformal gauge, that is, a metric conformal to $g$ chosen so that the scalar curvature is continuous and positive. Normal coordinates refer, on the other hand, to coordinates explicitly constructed such that $g_{ij}(p) = 0$ and $\partial_kg_{ij}(p) = 0$, not via the exponential map --see \cref{subsec: normal coordinates} for details--.

We remark that, while in the smooth setting usual conformal normal coordinates provide an expansion $h(x) = A + \BO_2(| x |)$, our control 
in \eqref{h_decayIntro} and \eqref{h_decay_qnIntro} becomes much more subtle. 
Note that it is only when $r>\frac{n}{2}$ the function $h$ becomes continuous, $\BO(|x|^{2-\frac{n}{r}})=o(1)$ and $A=h(p)$. In all the other cases, there is no actual meaning for the constant $A$ and one may directly write, $h(x)=\BO\bigl(|x|^{2-\frac{n}{r}}\bigr)$ and $h(x)=\BO_1\bigl(|x|^{2-\frac{n}{r}}\bigr)$, respectively. One may check that $r>\frac{n}{2}$ can only occur when $n=3$ and $q>3$; in this case one also gains the $C^{1,2-\frac{n}{r}}$-control of \eqref{h_decay_qnIntro}, which ensures suitable regularity and decay of the metric of the AE decompactified manifold for the applicability of the PMT of \cite{LeeLeF}. Moreover, such first-order control allows us to easily read off the ADM mass of the AE manifold, being it proportional to the constant $A$. The restriction to dimension $n=3$ and $q > 3$ in the hypotheses of \cref{theorem A} stems exactly from here.

Comparing \cref{theorem C} to other results in the literature, let us comment that in \cite[Appendix A]{DruetHebeyRobert} the authors deal with the Laplace–Beltrami operator of a smooth metric with a Hölder continuous potential, whereas in \cite[Lemma 4.2]{DruetLaurain} with the Euclidean Laplacian on a domain of $\nR^3$ and a potential function in $L^q$ with $q > 3$. Above we are considering an analogous problem to \cite[Lemma 4.2]{DruetLaurain} on a closed manifold, but for the geometric Laplacian with rough coeﬀicients and the scalar curvature as potential. Moreover, we deal with any dimension and with a larger interval of $q$, although It should be noted that we strongly exploit the conformal properties of $\Lg$. Finally, comparing with asymptotic expansions of \cite[Chapter 3, Section 5]{TaylorToolsForPDEs}, in \eqref{h_decay_qnIntro} we are incorporating an expansion valid for metrics below the $C^1$-threshold.

\addtocontents{toc}{\protect\setcounter{tocdepth}{1}}
\subsection{Outline of the paper}
\addtocontents{toc}{\protect\setcounter{tocdepth}{2}}
In \cref{Preliminaries} we set our conventions for Sobolev spaces on closed and asymptotically Euclidean manifolds and show some fundamental properties for later use. We also include a brief introduction to the Yamabe problem, focusing on the key aspects for rough metrics and, most importantly, we present a Yamabe classification (\cref{cor: Yamabe classification}). In \cref{Analytical Results} we develop the core analytical tools of the paper. We study the conformal Laplacian in detail for rough Riemannian metrics, including its Fredholm, invertibility, and regularity properties. We also introduce normal and harmonic coordinates in this setting. Some of the results of this section might be of interest elsewhere. \cref{The Conformal Green's function} is devoted to the study of the conformal Green function, where a delicate blow-up analysis and the proof of \cref{theorem C} are carried out. Finally, the Yamabe problem is treated in \cref{The Yamabe Problem}. Here, \cref{theorem A} and \cref{theorem B} are proven following the classical steps and using the results developed in the previous sections.

\subsection*{Acknowledgements}
The authors are grateful to Gerhard Huisken and Bruno Premoselli for their interest in this work and for providing useful references. Rodrigo Avalos is thankful to the \emph{Alexander von Humboldt
Foundation} for partial financial support during the writing of this paper.
This project began at the \textit{Simons Center for Geometry and Physics} in Stony Brook on the occasion of the "Mass, the Einstein Constraint Equations, and the Penrose Inequality Conjecture" workshop. The authors thank the institute for the hospitality.

\section{Preliminaries}
\label{Preliminaries}

In this section we will collect some analytic definitions and results relevant for the core of the paper. These are mainly related to Sobolev spaces and low regularity geometric structures.
Moreover, in this section we will review some key properties from the Yamabe problem for low regularity metrics. 

\subsection{Sobolev spaces on closed manifolds}\label{SectionSobolevSpaces}

Given an arbitrary domain $\Omega \subset \nR^n$ we denote by $L^p(\Omega)$, $W^{k,p}(\Omega)$ and $C^{k,\alpha}(\Omega)$ the usual Lebesgue, Sobolev and Hölder spaces and use the notation $W^{k,p}_0(\Omega)$ and $C^{k,\alpha}_0(\Omega)$ for the closures of $C^{\infty}_0(\Omega)$ in $W^{k,p}(\Omega)$ and $C^{k,\alpha}(\Omega)$, respectively. We also say that $u \in W^{k,p}_{loc}(\Omega)$ if $u \in W^{k,p}(K)$ for all smooth\footnote{Throughout the paper, we say that a domain is \emph{smooth} if it has smooth boundary.} $K \subset\subset \Omega$. If $E \to \Omega$ is a vector bundle, we write $W^{k,p}(E)$ for the space of $W^{k,p}$ sections of $E$, and likewise for the other spaces. For the precise definitions and properties about these Banach spaces we refer the reader to \cite{GT,AdamsFournierBook} and references therein.

There are many distinct but equivalent ways of defining Sobolev spaces on smooth, closed manifolds. One common approach --see e.g. \cite{HebeyBook}-- is to use a smooth background Riemannian metric $g$, which induces a smooth measure $\mu_{g}$ and a covariant derivative $\nabla$, which in turn can be used to define Sobolev norms
\begin{equation}
    \label{eq: SobolevNorm backgroundmetric}
    \|u\|_{W^{k,p}(M,d\mu_{g})} = \sum_{i=0}^k\left(\int_M|\nabla^iu|^p \, d\mu_g\right)^{\frac{1}{p}} \,.
\end{equation}
In this work, however, it will be convenient to adopt a \textit{local} approach. We denote by $\Gamma(T_sM)$ the space of measurable sections of the tensor bundle $T_sM$.
\begin{definition}
    \label{SobSpacesManifoldDefnTensors.1}
    Let $M$ be a smooth, closed manifold and $s\in \mathbb{N}_0$. For $k \in \mathbb{N}_0$ and $1\leq p< \infty$, we define the Sobolev spaces
    \begin{align*}
        W^{k,p}(T_sM) \doteq \bigl\{u \in \Gamma(T_sM) \; : \; \varphi_* u \in W^{k,p}_{loc}\bigl(T_s(\varphi(U)\bigr) \; \text{for all charts} \; (U, \varphi)\bigr\} \,.
    \end{align*}
    Given a finite atlas $\mathcal A = \{(U_\alpha, \varphi_\alpha)\}_{\alpha=1}^N$ and a subordinate partition of unity $\{\eta_\alpha\}_{\alpha=1}^N$, we consider the norms\footnote{The Sobolev norm of a tensor field $T = T_{i_1, \ldots, i_s}(x) dx^{i_1} \otimes \ldots \otimes dx^{i_s}$ in Euclidean space is just $\|T\|_{W^{k,p}} = \sum_{i_1=1}^n \cdots \sum_{i_s=1}^n  \|T_{i_1 \ldots i_s}\|_{W^{k,p}}$.}
    \begin{align}
        \label{SobolevNormTensorFields}
        \|u\|_{W^{k,p}(M,\mathcal A)} \doteq \sum_{\alpha=1}^N \|(\varphi_\alpha)_*(\eta_\alpha u)\|_{W^{k,p}(\varphi_\alpha(U_\alpha))} \,,
    \end{align}
    which make the above space Banach.
\end{definition}
\begin{remark}
    Notice that by Lemma \ref{prop: tensor_regularity} applied to the case of smooth diffeomorphisms, distinct choices of atlases and subordinate partitions of unity yield equivalent norms for (\ref{SobolevNormTensorFields}). In fact, we will not make reference to the chosen atlas and simply denote the norms by $\|\cdot\|_{W^{k,p}(M)}$. One can check, for instance, from the proof of the second half of \cite[Theorem 2.20]{AubinBook}, that the norms \eqref{eq: SobolevNorm backgroundmetric} and \eqref{SobolevNormTensorFields} are equivalent. 
\end{remark}
Using that the map
\begin{align*}
    W^{l,p}(M)&\to W^{l,p}(\varphi_{1}(U_{1}))\times \cdots\times W^{l,p}(\varphi_{N}(U_{N})),
    \\
    u&\mapsto \left((\varphi_1)_{*}(\eta_1 u),\cdots,(\varphi_N)_{*}(\eta_N u))\right) \,.
\end{align*}
is an isometric isomorphism, it is immediate to see that the usual properties of local Sobolev spaces (such as embeddings, density theorems or multiplication properties) transfer to the global ones. We will also need to work with Sobolev spaces of \textit{negative} regularity, which locally are defined simply by duality
\begin{align}
    \label{NegativeSobolevLocal}
    W^{-k,p'}(\Omega) \doteq \bigl(W_0^{k,p}(\Omega)\bigr)'.
\end{align}
Properties of these spaces can be found, for instance, in \cite[Chapter 3]{AdamsFournierBook}. Given a smooth closed manifold $M$, one can extend this local definition to that of $W^{-k,p'}(M)$, similarly to the case of $W^{k,p}(M)$ with $k\in \mathbb{N}$.  
To accurately introduce the associated definitions, let us denote by $\mathscr{D}(M)$ the space of test functions on $M$, which is $C^{\infty}(M)$ equipped with the usual inductive limit topology. Since $M$ is compact, this space is Fréchet. Then, the space of distributions $\mathscr{D}'(M)$ on $M $ is defined 
as the topological dual to $\mathscr{D}(M)$. In order to have a local characterization of elements in $\mathscr{D}'(M)$, consider a local chart $(U_{\alpha},\varphi_{\alpha})$ in $M$ and notice that by, for instance, 
\cite[Theorem 46]{HolstBehzadanSobSpaces} the map
\begin{align}
    \label{CoordRepD}
    \begin{split}
    T_{\alpha}:\mathscr{D}(U)&\to \mathscr{D}(\varphi_{\alpha}(U)),
    \\
    v&\mapsto v\circ\varphi_{\alpha}^{-1} 
    \end{split}
\end{align} 
is a topological isomorphism. 
The dual map $T_{\alpha}^{*}:\mathscr{D}'(\varphi_{\alpha}(U_{\alpha}))\to \mathscr{D}'(U_{\alpha})$ then reads 
\begin{align*}
    (T^{*}_{\alpha}u)(v)=u(T_{\alpha}v)=u(v\circ\varphi_{\alpha}^{-1}) \,, 
\end{align*} 
for $u\in \mathscr{D}'(\varphi_{\alpha}(U_{\alpha}))$ and $v\in \mathscr{D}(U_{\alpha})$ is again a linear topological isomorphism. We denote $\hat{T}_{\alpha}\doteq (T_{\alpha}^{*})^{-1}$ and notice that this map assigns to each distribution on $U_{\alpha}\subset M$ a distribution on $\varphi_{\alpha}(U_{\alpha})\subset\mathbb{R}^n$. Then, given $u\in \mathscr{D}'(U_{\alpha})$ one regards $\hat{T}_{\alpha}u\in \mathscr{D}'(\varphi_{\alpha}(U_{\alpha}))$ as its \emph{local} coordinate
representation, in the same way that  
$T_{\alpha}v$ 
is the local coordinate representation of $v$ in (\ref{CoordRepD}). Observe that whenever $M$ is equipped with a continuous Riemannian metric $g$ and $u \in L^1(M,d\mu_g)$ is regarded as a regular distribution, then $\hat{T}_{\alpha}u = u \circ \varphi_\alpha^{-1}$: letting $v\in \mathscr{D}(\varphi_{\alpha}(U_{\alpha}))$ and denoting by $\{x^i\}_{i=1}^n$ the associated coordinates, we have
\begin{align}
    \label{RegDistribM}
    \begin{split}
        (\hat{T}_{\alpha}u)(v)&=u(T_{\alpha}^{-1}v)=\int_{M}u(v\circ\varphi_{\alpha}) d\mu_g = \int_{\varphi_{\alpha}(U_{\alpha})}(u\circ\varphi_{\alpha}^{-1})(x)v(x)d\mu_g(x)
        \\
        &=(u\circ\varphi_{\alpha}^{-1})(v),
    \end{split}
\end{align}
where each identity in the first line follows by definition, and the meaning of the second line is the usual action of $u\circ\varphi_{\alpha}^{-1}\in L^1_{loc}(\varphi_{\alpha}(U_{\alpha}))$ as an element of $\mathscr{D}'(\varphi_{\alpha}(U_{\alpha}))$ on $v\in \mathscr{D}(\varphi_{\alpha}(U_{\alpha}))$. The above construction extends this to the general case of elements of $\mathscr{D}'(M)$ and allows for the following definition.
\begin{definition}
    \label{def: NegativeSobolevGLobal}
    Let $M$ be a smooth, closed manifold. For $k\in \nN$ and $1<p<\infty$ we define 
    \begin{align*}
        W^{-k,p'}(M) \doteq \bigl\{ u \in \mathscr D'(M) \; : \; \hat{T}_{\alpha}u \in W^{-k,p}(\varphi_{\alpha}(U_{\alpha})) \; \text{for all charts} \; (U_{\alpha}, \varphi_{\alpha})\bigr\} \,.
    \end{align*}
    Given a finite atlas $\mathcal{A}=\{(U_{\alpha},\varphi_{\alpha})\}_{\alpha=1}^N$ and a subordinate partition of unity $\{\eta_{\alpha}\}_{\alpha}^N$, we consider the norms
    \begin{align*}
        \|u\|_{W^{-k,p}(M,\mathcal A)} \doteq \sum_{\alpha=1}^N \|\hat{T}_{\alpha}(\eta_\alpha u)\|_{W^{-k,p'}(\varphi_\alpha(U_\alpha))} \,,
    \end{align*}
    which make the above space Banach.
\end{definition}
It is easy to see by the local duality properties and \cite[Theorem 3.41]{AdamsFournierBook} for smooth diffeomorphisms that distinct choices of atlas or partition of unity yield equivalent norms (an explicit proof of the local statement can be found in \cite[Theorem 75]{HolstBehzadanSobSpaces}). By the local nature of \cref{def: NegativeSobolevGLobal}, these spaces are easily seen to satisfy most of the properties of their local counterparts $W^{-k,p'}(\Omega)$; in particular, they are reflexive and smooth functions are dense. This allows to relate $W^{-k,p'}(M)$ and $(W^{k,p}(M))'$:
\begin{theorem}
    \label{SobolevDualIso}
    Let $M$ be a smooth, closed manifold and fix $k\in \mathbb{N}$ and $1<p<\infty$. Given a continuous Riemannian metric $g$ on $M$, the bilinear form
    \begin{align*}
        C^{\infty}(M) \times C^{\infty}(M) \to \nR \,, \qquad (u,v) \mapsto \langle u,v\rangle_{L^2(M,d\mu_g)}
    \end{align*}
    extends uniquely by continuity to a continuous bilinear form
    \begin{align}
        \label{SobPairing.2}
        W^{-k,p'}(M) \times W^{k,p}(M) \to \mathbb{R} \,, \qquad (u,v) \mapsto \langle u,v\rangle_{(M,g)} \,.
    \end{align}
    Moreover, the $S_{k,p}:W^{-k,p'}(M)\to (W^{k,p}(M))'$ mapping defined by $(S_{k,p}u)(v)=\langle u,v\rangle_{(M,g)}$ for all $v\in W^{k,p}(M)$ is a linear topological isomorphism.
\end{theorem}
\begin{proof}
    The first assertion follows directly from the straightforward inequality
    \begin{align*}
        |\langle u,v\rangle_{L^2(M,d\mu_g)}|\leq C(M,g)\Vert u\Vert_{W^{-k,p'}(M)}\Vert v\Vert_{W^{k,p}(M)} \,,
    \end{align*}
    so let us show the isomorphism. First, notice that the map $S_{k,p}$ is linear and well-defined and that for any continuous Riemannian metric $g$ one has $L^{p'}(M,d\mu_g) \cong (L^p(M,d\mu_g))'$. One may then parallel the discussion of \cite[Chapter 3, Section 3.13]{AdamsFournierBook} closely observing that, for each $u\in L^{p'}(M,d\mu_g)$, the functionals $L_u:W^{k,p}(M)\to \mathbb{R}$ given by $L_u(v)\doteq \langle u,v\rangle_{L^2(M,d\mu_g)}$, define elements in $(W^{k,p}(M))'$ and, moreover, that the set $\{L_u \: :\: u\in L^{p'}(M,d\mu_g)\} \subset (W^{k,p}(M))'$ is dense. Since $C^{\infty}(M)$ is dense in $L^{p'}(M,d\mu_g)$ as well, we see that for any $u\in (W^{k,p}(M))'$ there is a sequence $\{u_j\}_{j=1}^{\infty}\subset C^{\infty}(M)$ such that
    \begin{align}
        \label{DualPairing.1}
        u(v)=\lim_{j\to \infty}\langle u_j,v\rangle_{L^2(M,d\mu_g)}
    \end{align}
    for all $v \in W^{k,p}(M)$. In other words, $\Vert u-L_{u_j}\Vert_{(W^{k,p}(M))'} \to 0$ as $j\to\infty$, where each $L_{u_j}\in W^{-k,p'}(M)$ by (\ref{RegDistribM}). On the other hand, simply using the definition of the global norms,
    \begin{align*}
        \Vert L_{u_j}-L_{u_k}\Vert_{W^{-k,p'}(M)}&=\sum_{\alpha}\Vert \hat{T}_{\alpha}(\eta_{\alpha}(L_{u_j}-L_{u_k}))\Vert_{W^{-k,p'}(\varphi_{\alpha}(U_{\alpha}))} \,,
    \end{align*}
    and noticing that since $\eta_{\alpha}(L_{u_j}-L_{u_k})\in L^{p'}(M,d\mu_g)$ are regular distributions, we have
    \begin{align*}
        \Vert \hat{T}_{\alpha}(\eta_{\alpha}(L_{u_j}&-L_{u_k}))\Vert_{W^{-k,p'}(\varphi_{\alpha}(U_{\alpha}))}
        \\
        &=\sup_{v\in W_0^{k,p}(\varphi_{\alpha}(U_{\alpha}))}\frac{|\langle (\eta_{\alpha}(u_j-u_k)\circ\varphi^{-1}_{\alpha}),v\rangle_{L^2(\varphi_{\alpha}(U_{\alpha}),d\mu_g)}|}{\Vert v\Vert_{W^{k,p}(\varphi_{\alpha}(U_{\alpha}))}}
        \\
        &\leq C_{\alpha}\sup_{v\in W_0^{k,p}(\varphi_{\alpha}(U_{\alpha}))}\frac{|\langle u_j-u_k,v\circ\varphi_{\alpha})\rangle_{L^2(U_{\alpha},d\mu_g)}|}{\Vert v\circ\varphi_{\alpha}\Vert_{W^{k,p}(U_{\alpha})}}
        \\
        &\leq C_{\alpha}\sup_{v\in W^{k,p}(M)}\frac{|\langle u_j-u_k,v\rangle_{L^2(M,d\mu_g)}|}{\Vert v\Vert_{W^{k,p}(M)}} = C_\alpha\Vert L_{u_j} - L_{u_k}\Vert_{(W^{k,p}(M))'} \to 0 
    \end{align*}
    as $j,k\to\infty$, where the second inequality follows from $\varphi^{*}(W_0^{k,p}(\varphi_{\alpha}(U_{\alpha})))\subset W^{k,p}(M)$ and the last limit holds as by hypothesis $\{L_{u_j}\}_{j=1}^{\infty}$ is a Cauchy sequence in $(W^{k,p}(M))'$. This implies $\{L_{u_j}\}_{j=1}^{\infty}$ is a Cauchy sequence in $W^{-k,p'}(M)$ as well and, by completeness, there is some $\hat{u}\in W^{-k,p'}(M)$ such that $L_{u_j}\to\hat{u}$ in $W^{-k,p'}(M)$. But then, by the definition of the bilinear form \eqref{SobPairing.2} we have that
    \begin{align*}
        u(v)=\lim_{j\to \infty}\langle u_j,v\rangle_{L^2(M,dV_g)}=\langle \hat{u},v\rangle_{(M,g)}
    \end{align*}
    for all $v\in W^{k,p}(M)$. Therefore, $S_{k,p}$ is surjective. 
    
    Concerning injectivity, assume that $(S_{k,p}\hat{u})(v)=0$ for all $v\in W^{k,p}(M)$. In particular,
    \begin{align*}
        0 = \langle \hat{u},v\rangle_{(M,g)}=\lim_{j\to\infty}L_{u_j}(v) \,,
    \end{align*}
    where $L_{u_j}\to \hat{u}$ in $W^{-k,p'}(M)$. We claim that the action of $\hat{u}$ on $\mathscr{D}(M)$ can be computed in terms of $\langle \hat{u},v\rangle_{(M,g)}$ restricted to $v\in \mathscr{D}(M)$. For this, letting $(U_{\alpha},\varphi_{\alpha})$ be an arbitrary chart, we see that
    \begin{align*}
        L_{u_j}(v)\to \hat{u}(v) \: \forall\: v\in \mathscr{D}(M) &\Longleftrightarrow L_{u_j}(v)\to \hat{u}(v) \: \forall\: v\in \mathscr{D}(U_{\alpha}) \; \forall \, U_\alpha
        \\
        &\Longleftrightarrow (\hat{T}_{\alpha}L_{u_j})(T_{\alpha}v)\to (\hat{T}_{\alpha}\hat{u})(T_{\alpha}v) \: \forall\: v\in \mathscr{D}(U_{\alpha}) \; \forall \, U_\alpha
        \\
        &\Longleftrightarrow (\hat{T}_{\alpha}L_{u_j})(\tilde{v})\to (\hat{T}_{\alpha}\hat{u})(\tilde{v}) \: \forall\: \tilde{v}\in \mathscr{D}(\varphi_{\alpha}(U_{\alpha})) \; \forall \, U_\alpha
    \end{align*}
    where we have used that $T_\alpha$ and $\hat T_\alpha$ are topological isomorphisms. But then, since $\hat{T}_{\alpha}L_{u_j}\to\hat{T}_{\alpha}\hat{u}$ in $W^{-k,p'}(\varphi_{\alpha}(U_{\alpha}))$ and $W^{-k,p'}(\varphi_{\alpha}(U_{\alpha}))\hookrightarrow \mathscr{D}'(\varphi_{\alpha}(U_{\alpha}))$,\footnote{Since $\varphi_{\alpha}(U_{\alpha})$ is a domain in $\mathbb{R}^n$, then $\mathscr{D}(\varphi_{\alpha}(U_{\alpha}))\overset{\iota}{\hookrightarrow} W_{0}^{k,p}(\varphi_{\alpha}(U_{\alpha}))$ and thus $W^{-k,p'}(\varphi_{\alpha}(U_{\alpha}))\overset{\iota^{*}}{\hookrightarrow} \mathscr{D}'(\varphi_{\alpha}(U_{\alpha}))$.} we have that $\hat{T}_{\alpha}L_{u_j}\to\hat{T}_{\alpha}\hat{u}$ in $\mathscr{D}'(\varphi_{\alpha}(U_{\alpha}))$ as well, implying that indeed 
    \begin{align*}
        (\hat{T}_{\alpha}L_{u_j})(\tilde{v})\to (\hat{T}_{\alpha}\hat{u})(\tilde{v})
    \end{align*}
    for all $\tilde{v}\in \mathscr{D}(\varphi_{\alpha}(U_{\alpha}))$. Therefore, we conclude that $L_{u_j}(v)\to \hat{u}(v)$ and consequently
    \begin{align*}
        \langle \hat{u},v\rangle_{(M,g)}=\hat{u}(v)
    \end{align*}
    holds for all $v\in \mathscr{D}(M)$, showing that $\hat{u}=0$ as an element of $\mathscr{D}'(M)$. This establishes injectivity. We have therefore shown that $S_{k,p}$ is a bounded isomorphism between Banach spaces and thus it is topological isomorphism since the inverse must then be bounded by the open mapping theorem.
\end{proof}
Let us finally briefly discuss \emph{interpolating Sobolev spaces}, which will be useful on a few occasions, although
in this paper we will work mostly with Sobolev spaces with integer degree of regularity.
A family of Banach spaces interpolating between two selected Banach spaces can be obtained via \emph{complex interpolation theory} --for a detailed description 
we refer the reader to \cite{Calderon1,Triebel1}, despite most of the properties we shall use being contained more concisely in \cite[Chapter 4, Section 2]{Taylor1}--. 
In particular, following \cite[Chapter VII, Section 7.66]{Adams1st}, given a smooth bounded domain $\Omega \subset \nR^n$, $1<p<\infty$, $\theta\in (0,1)$ and $k\in \mathbb{N}_0$, we define
\begin{align}
    \label{BesselPotentials.3}
    H^{s,p}(\Omega)\doteq [L^p(\Omega),W^{k,p}(\Omega)]_{\theta}, \text{ for } s=\theta k,
\end{align}
denoting the family of spaces interpolating between $L^p(\Omega)$ and $W^{k,p}(\Omega)$, parametrized by $\theta\in (0,1)$.
We then denote by $H^{s,p}_0(\Omega)$ the closure of $C^{\infty}_0(\Omega)$ in $H^{s,p}(\Omega)$, and for $1<p<\infty$ and $s<0$ a real number,
\begin{align}
    \label{BesselPotentials.4}
    H^{s,p'}(\Omega)\doteq [H_0^{-s,p}(\Omega)]'.
\end{align}
If $\Omega$ is substituted by $\nR^n$, these spaces coincide with the usual \emph{Bessel potentials} defined via the Fourier transform --see \cite[Proposition 6.2]{Taylor3}--.

\subsection{Weighted Sobolev spaces and AE manifolds}
\label{AESection}
The classical strategy to solve the Yamabe problem on closed manifolds in low dimensions relies on a tight relation with the positive mass theorem of general relativity. We then review some standard terminology and results associated with asymptotically Euclidean manifolds. Following the conventions in \cite{Bartnik86}, we use the norms 
\begin{align}\label{WeightedSobolevNormsDefs}
\begin{split}
    \|u\|_{W^{k,p}_{\delta}(\mathbb{R}^n)} &\doteq \sum_{|\beta|\leq k} \|\sigma^{-\delta-\frac{n}{p}+|\beta|}\partial^{\beta}u\|_{L^p(\nR^n)}
    \\
    \|u\|_{C^{k,\alpha}_{\delta}(\nR^n)} &\doteq \sum_{|\beta|<k}\|\sigma^{-\delta+|\beta|}\partial^\beta u\|_{C^0(\nR^n)} + \sup_{|\beta|=k} \|\sigma^{-\delta+k+\alpha}\partial^\beta u\|_{C^{0,\alpha}(\nR^n)}
\end{split}
\end{align}
with weight $\sigma(x)\doteq (1+|x|^2)^{\frac{1}{2}}$ in Euclidean space, and define the following.
\begin{definition}
    \label{WeightedSobSpacesManifold}
    Let $M$ be a smooth $n$-manifold and $\Phi$ a Euclidean structure at infinity, that is, a smooth diffeomorphism
    \begin{equation*}
        \Phi : M \setminus K \to \nR^n \setminus \overline{B_R}
    \end{equation*}
    for some compact set $K \subset M$ and radius $R > 0$. For each $s \in \nN_0$, $k \in \nN_0$, $1\leq p< \infty$ and $\delta \in \nR$, we define the weighted Sobolev spaces 
    \begin{equation*}
        W^{k,p}_{\delta}(T_sM,\Phi) \doteq \Bigl\{u \in W^{k,p}_{loc}(T_sM) \;:\; \|\Phi_*u\|_{W^{k,p}_{\delta}(\nR^n\setminus\overline{B_R})} < \infty\Bigr\} \,.
    \end{equation*}
    Given a finite coordinate cover $\mathcal A = \{(U_\alpha,\varphi_\alpha)\}_{\alpha=1}^{N-1}$ of $K$ and a partition of unity $\{\eta_\alpha\}_{\alpha=1}^{N}$ subordinate to $\mathcal A \cup (M \setminus K,\Phi)$, we consider the norms
    \begin{equation*}
        \|u\|_{W^{k,p}_\delta(T_sM,\Phi)} = \sum_{\alpha=1}^{N-1} \|(\varphi_\alpha)_*(\eta_\alpha u)\|_{W^{k,p}(\varphi_\alpha(U_\alpha))} + \|\Phi_*(\eta_{N} u)\|_{W^{k,p}_\delta(\nR^n \setminus \overline{B_R})}
    \end{equation*}
    which make the above spaces Banach. We will omit, if no confusion arises, the reference to the structure at infinity $\Phi$.
\end{definition}
Analogously, one can define weighted Hölder spaces on manifolds admitting a Euclidean structure at infinity. Most properties of the associated weighted Sobolev spaces, such as embeddings and multiplication properties, can be found in \cite{Bartnik86,NirenbergWalker,McOwen,Lockhart,CB-C}. We refer the reader to \cite[Theorem 2.3]{avalos2024sobolev} for a compiled version.

 Using the above notion of weighted Sobolev spaces, we now introduce 
 AE manifolds, namely, non-compact manifolds endowed with a metric decaying  
to the Euclidean one along an end.
\begin{definition}
    \label{AE-Manifolds}
	Let $(M,g)$ be a Riemannian $n$-manifold and $\Phi$ a Euclidean structure at infinity. Then, $(M,g)$ is called $W^{k,p}_{\delta}$-asymptotically Euclidean with respect to $\Phi$ if
    \begin{enumerate}
        \item[$(i)$] $g \in W^{k,p}_{loc}(T_2M)$ for some $k>\frac{n}{p}$,
        \item[$(ii)$] $g-\Phi^{*}g_{\nR^n}\in W^{k,p}_{\delta}(T_2M)$ for some $\delta<0$.
    \end{enumerate}
\end{definition}
AE manifolds arise as models for isolated gravitational systems and have been intensively studied in the last decades. Conservation principles give rise to associated geometric invariants, such as the so called ADM energy, which captures the total energy of the system. For our discussion, we restrict ourselves to 3 dimensions:
\begin{definition}
    Let $(M,g)$ be a $C^{1}_{-\tau}$-AE 3-manifold with respect to a structure at infinity $\Phi_x$ with asymptotic coordinates $\{x^i\}_{i=1}^3$ for some $\tau>0$. Then, its ADM energy is defined by
    \begin{align}
        \label{ADMEnergy}
        E_{ADM}(g)\doteq \frac{1}{16\pi}\lim_{r\to\infty}\int_{S_r}\left( \partial_{i}g_{ij}-\partial_{j}g_{ii} \right)\frac{x^i}{|x|}d\omega_r,
    \end{align}
    whenever the limit exists, and where $S_r$ denotes a topological sphere of radius $r$ contained in the end of $M^3$ and $d\omega_r$ its induced Euclidean volume form.
\end{definition}
The above definition has been presented in a rather broad context due to its use in this paper,
but it leaves open the question of when it is indeed well-defined. It is known since the work of R. Bartnik (\cite[Section 4]{Bartnik86}) that the ADM energy is finite and independent of the asymptotic coordinates or sequence of topological spheres,
provided that $(M,g)$ is $W^{2,q}_{-\tau}$-AE with $\tau > \tfrac{1}{2}$, $q > 3$ and $\R_g \in L^1(M,d\mu_g)$. Its computation becomes particularly simple when the manifold is asymptotic to the spacial Schwarzschild manifold, which will be exploited in the resolution of the Yamabe problem in \cref{The Yamabe Problem}.
\begin{definition}
    Let $(M,g)$ be a $C^k_{-\tau}$-AE 3-manifold with respect to a structure at infinity $\Phi_x$ with coordinates $\{x^i\}_{i=1}^3$ for some $k\in \mathbb{N}_0$ and $\tau>0$. We say that $(M,g)$ is asymptotically Schwarzschildian of order $k$ with respect to $\Phi_x$ if 
    \begin{align}
        \label{ASDefn}
        g_{ij}(x) = \left(1+\frac{4m}{|x|}\right)\delta_{ij} + O_k(|x|^{-1-\varepsilon}),
    \end{align}
    for some constant $m$ and some $\varepsilon>0$. In such case, we say that $g\in AS_{\varepsilon}(k)$.
\end{definition} 
Notice that for any $AS_{\varepsilon}(1)$ metric $g$, direct computations give that $E_{ADM}=m$.

\subsection{A Yamabe classification}\label{The Yamabe classification}

Let $g \in W^{2,q}(T_2M)$ be a Riemannian metric on a smooth, closed manifold $M$ of dimension $n \geq 3$ with $q > \tfrac{n}{2}$. It is in particular Hölder continuous. Given a positive function $u \in W^{2,q}(M)$, a standard computation shows that the conformal metric $\tilde g \doteq u^{\frac{4}{n-2}} g$ has scalar curvature
\begin{equation}
    \label{eq: ConformalScalar}
    \R_{\tilde g} = u^{-\frac{n+2}{n-2}}\Big(-a_n\Delta_gu + \R_g u\Big) \,, \qquad a_n \doteq 4 \, \frac{n-1}{n-2} \,.
\end{equation}
Notice that $W^{2,q}(M) \otimes W^{2,q} \hookrightarrow W^{2,q}(M)$ by the multiplication properties of \cite[Theorem 2.2]{avalos2024sobolev}, implying that $\tilde g \in W^{2,q}(T_2M)$ and consequently $\R_{\tilde g} \in L^q(M)$. In view of \eqref{eq: ConformalScalar}, the metric $\tilde g$ has constant scalar curvature equal to $\lambda$ if and only if $u$ solves the semi-linear elliptic PDE
\begin{equation}
    \label{eq: Yamabe equation 1}
    \Lg u \doteq -a_n \Delta_g u + \R_g u = \lambda \, u^{\frac{n+2}{n-2}} \,,
\end{equation}
which is the Euler--Lagrange equation of
\begin{equation*}
    Q_g(u) \doteq \frac{E_g(u)}{\|u\|_{L^{2^*}(M,d\mu_g)}^2} \,, \qquad E_g(u) \doteq \int_M \bigl(a_n|\nabla u|^2_g + \R_g u^2\bigr) \, d\mu_g \,.
\end{equation*}
Here $2^* \doteq \tfrac{2n}{n-2}$ denotes the optimal Sobolev embedding power of $W^{1,2}(M)$. More generally, one may consider for any $s \in [2, 2^*]$ the Rayleigh quotient
\begin{equation}
    \label{eq: Qs functionals}
    Q_g^s(u) \doteq \frac{E_g(u)}{\|u\|_{L^{s}(M,d\mu_g)}^2}
\end{equation}
with critical points satisfying the Euler--Lagrange equation
\begin{equation}
    \label{eq: sPDE}
    \Lg u = q_s u^{s-1} \,, \qquad q_s \doteq Q_g^s(u) \,.
\end{equation}
It follows from the $W^{1,2}(M) \hookrightarrow L^{2^*}(M)$ embedding that the bilinear functional $A_g : W^{1,2}(M) \times W^{1,2}(M) \to \nR$ defined by
\begin{align}
    \label{YamabeBilinearForm}
    A_g(u, v) \doteq \int_M \Big(a_n \langle\nabla u, \nabla v\rangle_g + \R_g u \, v \Big) \, d\mu_g
\end{align}
is continuous and therefore it is convenient to use $W^{1,2}(M)$ as the space of admissible functions for the variational problem. In the present context, it is worth noticing that since the metric $g$ is continuous, it is easy to check by a standard computation that the $\|\cdot\|_{W^{1,2}(M,d\mu_g)}$ norm defined in \eqref{eq: SobolevNorm backgroundmetric} is equivalent to the one introduced in \cref{SobSpacesManifoldDefnTensors.1} for scalar fields\footnote{For scalar fields $\nabla u$ does not involve Christoffel symbols and, thereby, the norm $\|\cdot\|_{W^{1,2}(M,d\mu_g)}$ is well-defined for $C^0$-Riemannian metrics.}. Below, and whenever necessary, we will make use of this equivalence.
Using multiplication properties on the interpolation Sobolev spaces introduced in Section \ref{SectionSobolevSpaces}, one can show the following:
\begin{proposition}\label{lemma: Ru2 interpolation}
    Let $M$ be a smooth, closed manifold of dimension $n \geq 3$ and consider a $W^{2,q}$-Riemannian metric $g$ with $q > \tfrac{n}{2}$. Then, the $W^{1,2}(M) \to \nR$ map
    \begin{align*}
        u \mapsto \int_M \R_g u^2 d\mu_g
    \end{align*}
    is weakly continuous and for every $\varepsilon > 0$ there exists a constant $C_\varepsilon > 0$ such that
    \begin{equation}
        \int_M \R_g u^2 \, d\mu_g \leq \|\R_g\|_{L^q(M,d\mu_g)} \Big(\varepsilon \|u\|^2_{W^{1,2}(M,d\mu_g)} + C_\varepsilon \|u\|^2_{L^2(M,d\mu_g)}\Big) \,.
    \end{equation}
    Moreover, for each $s \in [2, 2^*]$ the functional $Q_g^s$ defined in \eqref{eq: Qs functionals} is bounded from below in $W^{1,2}(M)$.
\end{proposition}
\begin{proof}
    The proof follows, \emph{mutatis mutandis}, as in \cite[Lemma 3.1]{Maxwell1}). A similar result also appears in \cite[Proposition 2.3 and Lemma 4.2]{ZhangW1pAubin} via different techniques.
\end{proof}
\cref{lemma: Ru2 interpolation} in particular justifies the definition
\begin{equation}
    \lambda_s(M, g) \doteq \inf \Big\{Q_g^s(u) \; : \; u \in W^{1,2}(M) \,, \; u \neq 0\Big\} \,.
\end{equation}
The case $s=2$ of \eqref{eq: sPDE} corresponds to the eigenvalue problem of the operator $\Lg$ and $\lambda_2(M,g)$ is thus its first eigenvalue. The special case $\lambda_{2^*}(M,g)$ is the \textit{Yamabe invariant} of $(M, g)$ and it follows immediately from the fact that
\begin{align*}
    Q_{v^{2^*-2}g}(u) = Q_{g}(vu)
\end{align*}
for any positive $v\in W^{2,q}(M)$, that $\lambda(M,g) \doteq \lambda_{2^*}(M,g)$ is an invariant of the conformal class 
\begin{align*}
    [\,g\,]_{W^{2,q}} = \Bigl\{u^{\frac{4}{n-2}}g \: : \: u\in W^{2,q}(M)\,, \: u>0 \Bigr\} \,.
\end{align*}
As a consequence, the Yamabe invariant of the round sphere $\nS^n$ equals to $\tfrac{a_n}{C_S}$, where $C_S$ is the optimal constant in the $W^{1,2}(\nR^n) \hookrightarrow L^{2^*}(\nR^n)$ Sobolev inequality --see \cite[Theorem 3.3]{Lee-Parker}--. Another consequence of \cref{lemma: Ru2 interpolation}, combined with the Rellich--Kondrachov compactness, is that direct methods allow for minimization of $Q_g^s$ whenever $s$ is subcritical:
\begin{proposition}\label{prop: subcritical existence}
    Let $M$ be a smooth, closed manifold of dimension $n \geq 3$ and consider a $W^{2,q}$-Riemannian metric $g$ with $q > \tfrac{n}{2}$. For each $s \in [2, 2^*)$ there exists a non-negative function $u \in W^{1,2}(M)$ with $\|u\|_{L^s(M,d\mu_g)} = 1$ such that $Q_g^s(u) = \lambda_s(M, g)$.
\end{proposition}
\begin{proof}
    Since $Q_g^s$ is bounded from below by \cref{lemma: Ru2 interpolation}, we consider a sequence $\{u_k\}_{k=1}^{\infty} \subset W^{1,2}(M)$ with $\|u_k\|_{L^s(M,d\mu_g)} = 1$ such that $\lim_{k\to \infty} Q_g^s(u_k) = \lambda_s(M, g)$. Since $Q_g(|u_k|) = Q_g(u_k)$, we may as well assume that $u_k \geq 0$. In addition, using \cref{lemma: Ru2 interpolation}, we estimate
    \begin{align*}
    &\Vert u_k\Vert^2_{W^{1,2}(M,d\mu_g)} \leq \frac{1}{a_n}Q^s_g(u_k) - \frac{1}{a_n}\langle \R_g,u_k^2\rangle_{L^2(M,d\mu_g)} + \Vert u_k\Vert^2_{L^{2}(M,d\mu_g)}\\
        &\leq \frac{1}{a_n} Q^s_g(u_k) + \frac{1}{a_n}\Vert \R_g\Vert_{L^q(M,d\mu_g)}\Big(\varepsilon\Vert u_k\Vert^2_{W^{1,2}(M,d\mu_g)} +C_{\varepsilon}\Vert u_k\Vert^2_{L^2(M,d\mu_g)}\Big) + \Vert u_k\Vert^2_{L^{2}(M,d\mu_g)}.
    \end{align*}
    Picking $\varepsilon$ small enough allows us to absorb the second term in the left-hand side and, together the Hölder inequality $\Vert u_k\Vert_{L^2(M,d\mu_g)}\leq C(M,g)\Vert u_k\Vert_{L^s(M,d\mu_g)}=C(M,g)$, conclude that
    \begin{align*}
        \|u_k\|^2_{W^{1,2}(M,d\mu_g)} \leq C_1(M,g)(Q_g^s(u_k)+1) \,.
    \end{align*}
    In particular, $\{u_k\}_{k=1}^{\infty}$ is a bounded sequence in $W^{1,2}(M)$. Once this is established standard function analytical arguments grant the existence of a minimiser $u\in W^{1,2}(M)$ with $\|u\|_{L^s(M,d\mu_g)} = 1$ which, since $Q_g^s(u)=Q_g^s(|u|)$, we may take to be non-negative. See e.g. \cite[Section 3.3.2]{RoyoThesis} for details.
\end{proof}
Proposition \ref{prop: subcritical existence} above gives existence of non-negative weak solutions to \eqref{eq: sPDE}. The regularity theory developed in \cref{The Conformal Laplacian with Sobolev coefficients} together with a Harnack inequality due to Trudinger will allow us to promote them to positive solutions in $W^{2,q}(M)$ --see Corollary \ref{prop: subcritical regularity}--.
An important consequence is the following classification.
\begin{corollary}[Yamabe classification]
    \label{cor: Yamabe classification}
    Let $M$ be a smooth, closed manifold of dimension $n \geq 3$ and consider a $W^{2,q}$-Riemannian metric $g$ with $q > \tfrac{n}{2}$. 
    \begin{enumerate}
        \item[$(i)$] $\lambda(M, g) > 0$ if and only if there exists some element in $[\,g\,]_{W^{2,q}}$ with $W^{2, q}$ positive scalar curvature.
        \item[$(ii)$] $\lambda(M, g) = 0$ if and only if there exists some element in $[\,g\,]_{W^{2,q}}$ with zero scalar curvature.
        \item[$(iii)$] $\lambda(M, g) < 0$ if and only if there exists some element in $[\,g\,]_{W^{2,q}}$ with $W^{2, q}$ negative scalar curvature.
    \end{enumerate}
\end{corollary}
\begin{proof}
This result, in the stated level of regularity, follows by paralleling the proof of \cite[Proposition 3.3]{Maxwell1} as explained in \cite[Theorem 11]{Holst1}, but appealing to \cite[Theorem 3.3]{avalos2024sobolev} in order to fix the regularity claim of \cite[Theorem 11]{Holst1}. We refer the reader to \cite[Section 3.1]{avalos2024sobolev}.
\end{proof}

\section{Analytical Results}
\label{Analytical Results}

In this section, we investigate the main analytical properties of the conformal Laplacian $\Lg$ of $W^{2,q}$-Riemannian metrics on closed manifolds of dimension $n \geq 3$ and $q > \tfrac{n}{2}$. We develop a local and global regularity theory and establish isomorphism ranges for positive Yamabe metrics. Although the associated local theory for $W^{1,2}$-weak solutions for operators like $\Lg$ is well-known from standard references such as \cite[Chapter 8]{GT}, we will work with weaker a-priori solutions and therefore outside of the scope of the results therein. Some results closely related to our thresholds can be found in \cite[Chapter 3, Proposition 1.13]{TaylorToolsForPDEs}, which one may for instance compare to \cref{thm: main global elliptic regularity}. It is also worth noting that similar mapping properties of $\mathscr{L}_g$ for rough metrics have been studied in the context of the Einstein constraint equations of general relativity \cite{Maxwell1}. Once again, our results are slightly outside of the scope of these references and hence we need to develop a theory tailored for our applications. 

Along this section we will provide some regularity statements which apply to a set of \emph{weak} solutions to elliptic equations associated with operators with rough coefficients --see \cref{thm: 12 local elliptic regularity} and \cref{thm: conf Laplace elliptic regularity}--. In particular, we will often work with a priori solutions below $W^{1,2}$, where the De Giorgi--Nash theory or related results (such as those developed in \cite{Morrey}) do not apply. We refer the reader to \cite{Brezis1,Serrin} and the introduction to \cite{avalos2024sobolev} for further discussions on this topic.

\subsection{Laplace--Beltrami operator with Sobolev coefficients}\label{SectionLapceOp}
We begin by studying the Laplace--Beltrami operator
\begin{equation*}
    \Delta_g \doteq \tr_g\nabla^2
\end{equation*}
of metrics of Sobolev regularity. For a Riemannian metric $g \in W^{2,q}(T_2M)$, the operator $\Delta_g : C^{\infty}(M) \to W^{1,q}(M)$ can locally be written in divergence form
\begin{equation}\label{eq: Laplace-Beltrami local}
    \Delta_gu = \frac{1}{\sqrt{\mathrm{det}(g)}}\partial_{i}(g^{ij}\sqrt{\mathrm{det}(g)}\partial_{j}u)
\end{equation}
and as a consequence of \cite[Theorem 2.2]{avalos2024sobolev}, it extends to a bounded map $\Delta_g:W^{2,p}(M)\to L^{p}(M)$ for all $1< p\leq q$. If $u\in C^{\infty}(M)$ is a regular distribution in $\mathscr{D}'(M)$ and $v\in \mathscr{D}(U_\alpha)$ for a coordinate chart $(U_{\alpha},\varphi_{\alpha})$, then $\Delta_gu\in \mathscr{D}'(M)$ is defined by
\begin{align}
    \label{DistributionalLapGlboal.1}
    (\Delta_gu)(v)=\langle u,\Delta_gv\rangle_{L^2(M,d\mu_g)}=\langle u,\Delta_gv\rangle_{(M,g)} \,.
\end{align}
Thus, using the notation of \cref{SectionSobolevSpaces}, we can directly compute the local representative $\hat{T}_{\alpha}\Delta_gu$ for $u\in C^{\infty}(M)$ on any coordinate chart, given by
\begin{align*}
    (\hat{T}_{\alpha}\Delta_gu)(\tilde{v})&=(\Delta_gu)(\tilde{v}\circ\varphi_{\alpha})=\langle u,\Delta_g\tilde{v}\circ\varphi_{\alpha}\rangle_{L^2(U_{\alpha},d\mu_g)}
    \\
    &=\langle \Delta_g(u\circ\varphi^{-1}_{\alpha}),(\mathrm{det}(g)\circ\varphi_{\alpha}^{-1})^{\frac{1}{2}}\tilde{v}\rangle_{L^2(\varphi_{\alpha}(U_{\alpha}),dx)}
\end{align*}
for all $\tilde v \in \mathscr D(\varphi_\alpha(U_\alpha))$, that is, 
\begin{align}
    \label{DistributionalLapLocal}
    \hat{T}_{\alpha}\Delta_gu=(\mathrm{det}(g)\circ\varphi_{\alpha}^{-1})^{\frac{1}{2}}\Delta_g(u\circ\varphi^{-1}_{\alpha}) \,.
\end{align}
Appealing again to \cite[Theorem 2.2]{avalos2024sobolev}, it is not hard to see that one may extend $\Delta_g$ by continuity to a bounded map $L^p(M) \to W^{-2,p}(M)$ for all $p \geq q'$ and it is locally represented by \eqref{DistributionalLapLocal}. Moreover, this extension grants that the dual pairing \eqref{DistributionalLapGlboal.1} extends by continuity as well, that is, for all $u\in W^{-2,p}(M)$ and all $v\in W^{2,p'}(M)$ with $p\geq q'$ there holds
\begin{align*}
    \langle \Delta_gu,v\rangle_{(M,g)}&=\lim_{j\to \infty}\langle \Delta_gu_j,v\rangle_{(M,g)}=\lim_{j\to \infty}\langle u_j,\Delta_gv\rangle_{L^2(M,d\mu_g)}=\langle u,\Delta_gv\rangle_{L^2(M,d\mu_g)} \,.
\end{align*}
Below we present the local regularity properties for \eqref{eq: Laplace-Beltrami local} which are of interest along this paper:  
\begin{theorem}
    \label{thm: 12 local elliptic regularity}
    Let $\Omega \subset \nR^n$ be an open, bounded domain with smooth boundary and consider a Riemannian metric $g \in W^{2,q}(T_2\Omega)$ with $q > \tfrac{n}{2}$.
    \begin{enumerate}
        \item[$(i)$] If $u \in L^{q'}(\Omega)$ and $\Delta_gu \in W^{-1,p}(\Omega)$ for some $\tfrac{1}{p} \leq \tfrac{1}{q'} + \tfrac{1}{n}$, then it follows that $u \in W^{1,p}_{loc}(\Omega)$.
        \item[$(ii)$] If $u \in L^{q'}(\Omega)$ and $\Delta_gu \in L^{p}(\Omega)$ for some $1 < p \leq q$, then it follows that $u \in W^{2,p}_{loc}(\Omega)$.
    \end{enumerate}
\end{theorem}
\begin{proof}
    Both items are special cases of the finer assertions in \cite[Theorem 2.21]{MaxHolTso2}. Also, a finer version of item $(i)$ follows from \cite[Proposition 5.4]{TaylorBochner} and a bootstrap argument applied to a localisation of $u$ with a cut-off function. One may then also use \cite[Proposition 5.4]{TaylorBochner} via an additional routine bootstrap argument.
\end{proof}
The above result directly implies the following global version of it on a compact manifold $M$, by analysing the local regularity of $u\in L^{q'}(M)$ inside any small coordinate ball:
\begin{theorem}
    \label{thm: Rodri Global W1p W2p regularity}
    Let $M$ be a smooth, closed manifold of dimension $n \geq 3$ and consider a $W^{2,q}$-Riemannian metric $g$ on $M$ with $q > \tfrac{n}{2}$. 
   \begin{enumerate}
       \item[$(i)$] If $u \in L^{q'}(M)$ and $\Delta_g u \in W^{-1,p}(M)$ for some  $\tfrac{1}{p} \leq \tfrac{1}{q'} + \tfrac{1}{n}$, then it follows that $u \in W^{1,p}(M)$.
        \item[$(ii)$] If $u \in L^{q'}(M)$ and $\Delta_g u \in L^p(M)$ for some $1 < p \leq q$, then it follows that $u \in W^{2,p}(M)$.
    \end{enumerate}
\end{theorem}
By an induction argument, we further obtain:
\begin{theorem}
    \label{thm: main local elliptic regularity}
    Let $\Omega \subset \nR^n$ be an open, bounded domain with smooth boundary and consider a Riemannian metric $g \in W^{k,q}(T_2\Omega)$ with $q > \tfrac{n}{2}$ and $k \geq 2$. If $u \in L^{q'}(\Omega)$ and $\Delta_g u \in W^{k-2, p}(\Omega)$ for some $1 < p \leq q$, then it follows that $u \in W^{k,p}_{loc}(\Omega)$.
\end{theorem}
\begin{proof}
    We will proceed by induction on $k \geq 2$. Notice that the base case $k = 2$ has been proven in part $(ii)$ of \cref{thm: 12 local elliptic regularity}. Let us then assume that the theorem holds for $k-1 \geq 2$ and show that it holds for $k$. By hypothesis $\Delta_g u = f$ with $f \in W^{k-2, p}(\Omega)$, and the $(k-1)$-case implies that $u\in W^{k-1,p}_{loc}(\Omega)$. Hence, differentiating it in direction $x^a$, we get 
    \begin{align*}
        \Delta_g (\partial_a u) = \partial_a f - (\partial_a g^{ij}) \left(\partial_i \partial_j u  - \Gamma_{ij}^l \,\partial_l u \right) + g^{ij} \partial_a \Gamma_{ij}^l \,\partial_l u\,.
    \end{align*}
    Notice that since $k \geq 3$, then $\partial_a g^{ij}$ and  $\Gamma_{ij}^l$ are $W^{k-1, q}(\Omega)$ with $k - 1 \geq 2$ and $q > \tfrac{n}{2}$; in turn, in combination with the hypothesis $u \in W^{k-1, p}(\Omega)$, the second term of the right hand side is in $W^{k-3, p}(\Omega)$ by \cite[Theorem 2.2]{avalos2024sobolev}. On the other hand, also by the multiplication properties of \cite[Theorem 2.2]{avalos2024sobolev}, 
    \begin{equation*}
        \partial_a \Gamma_{ij}^l \,\partial_l u \in W^{k-2, q}(\Omega) \otimes W^{k-2, p}(\Omega) \hookrightarrow W^{k-3, p}(\Omega) \,.
    \end{equation*}
    Hence, by combining these observation with the hypotheses on $f$ and $u$, the function $\partial_a u \in  W^{k-3, q'}(\Omega)\, \cap \,W^{k-2, p}(\Omega)$ satisfies $\Delta_g (\partial_a u) \in W^{k-3, p}(\Omega)$. By the inductive hypothesis, we can immediately conclude that $\partial_a u \in W^{k-1, p}_{loc}(\Omega)$, from which the desired result follows. 
\end{proof}

We can then again extend \cref{thm: main local elliptic regularity} to compact manifolds:
\begin{theorem}\label{thm: main global elliptic regularity}
    Let $M$ be a smooth, closed manifold of dimension $n \geq 3$ and consider a $W^{k,q}$-Riemannian metric $g$ on $M$ with $k \geq 2$ and $q > \tfrac{n}{2}$. If $u \in L^{q'}(M)$ and $\Delta_g u \in W^{k-2, p}(M)$ for some $1 < p \leq q$, then $u \in W^{k,p}(M)$.
\end{theorem}

\subsection{The conformal Laplacian with Sobolev coefficients}
\label{The Conformal Laplacian with Sobolev coefficients}

We now move on to study the \textit{conformal Laplacian}
\begin{align*}
    \Lg = -a_n\Delta_g + \R_g \,.
\end{align*}
As observed in \cref{The Yamabe classification}, it is a very special Schrödinger-type operator that naturally shows up in conformal deformations of the scalar curvature. As such, it is not only invariant under diffeomorphisms, but also covariant under conformal deformations: for a conformal metric $\tilde g = u^{\frac{4}{n-2}} g$, the operator transforms like
\begin{equation}
    \label{eq: conformal covariance Lg}
    \mathscr L_{\bar g} \,v = u^{-\frac{n+2}{n-2}} \Lg(uv) \,.
\end{equation}
This symmetry will be conveniently exploited later on.
\begin{lemma}\label{lemma: Lg is Fredholm}
    Let $M$ be a smooth, closed manifold of dimension $n \geq 3$ and consider a $W^{k,q}$-Riemannian metric $g$ on $M$ with $k \geq 2$ and $q > \tfrac{n}{2}$. For each $1 < p \leq q$ the operator 
    \begin{equation*}
        \Lg : W^{k,p}(M) \to W^{k-2,p}(M)
    \end{equation*}
    is Fredholm of index zero and there exists a constant $C = C(M,g,n,p,q)$ such that
    \begin{equation}\label{estimate: semi-Fredholm}
        \|u\|_{W^{k,p}(M)} \leq C \Big(\|\Lg u \|_{W^{k-2,p}(M)} + \|u\|_{W^{k-2,p}(M)}\Big)
    \end{equation}
    holds for all $u \in W^{k,p}(M)$.
\end{lemma}
\begin{proof}  
    First, observe that if a sequence of metrics $\{g_j\}_{j=1}^{\infty} \subset C^{\infty}(T_2M)$ converges to $g$ in $W^{2,q}(T_2M)$, it is not hard to check that $\mathscr{L}_{g_j} \to \Lg$ in operator norm between the Banach spaces $W^{k,p}(M)$ and $W^{k-2,p}(M)$, that is, for any $u \in W^{k,p}(M)$ with $\|u\|_{W^{k,p}(M)} = 1$ there holds
    \begin{equation*}
        \| (\Lg - \mathscr{L}_{g_j})u  \|_{W^{k-2,p}(M)}  \leq C(M, p, q) \, \|g - g_j\|_{W^{k, q}(M)}\,.
    \end{equation*}
    Now, since the conformal Laplacian of a smooth metric is formally self-adjoint, it follows from standard results on linear operators that the operators $\mathscr{L}_{g_j}$ are Fredholm with index zero for all $j$. In addition, it is well-known --see e.g \cite[Theorem 5.16]{Kato}-- that the Fredholm index is locally constant in the space of semi-Fredholm operators acting between two given Banach spaces. Hence, once we show that $\Lg$ is semi-Fredholm, we can immediately conclude that it is Fredholm of index zero. Therefore, it suffices to prove the elliptic estimate \eqref{estimate: semi-Fredholm}.

    Concerning the estimate \eqref{estimate: semi-Fredholm}, a local version of it can be extracted, for instance, \cite[Theorem 2.21]{MaxHolTso2} plus interpolation of Sobolev norms. The general case then follows by a standard gluing argument using a finite partition of unity and Sobolev interpolation inequalities.
 \end{proof}

The above Fredholm properties of $\Lg$ are accompanied by the corresponding regularity results, as described in the following theorem:
\begin{theorem}
    \label{thm: conf Laplace elliptic regularity}
    Let $M$ be a smooth, closed manifold of dimension $n \geq 3$ and consider a  $W^{k,q}$-Riemannian metric $g$ on $M$ with $k \geq 2$ and $q > \tfrac{n}{2}$. If $u \in L^{q'}(M)$ and $\Lg u \in W^{k-2,p}(M)$ for some $1<p\leq q$, then $u \in W^{k,p}(M)$.
\end{theorem}
\begin{proof}
    The statement follows by the local regularity claim of \cite[Theorem 2.21]{MaxHolTso2}. Alternatively, rewriting 
    \begin{equation}
        \label{eq: Conformal to Beltrami Laplacian}
        \Delta_g u = -a_n^{-1} \bigl(\Lg u + \R_g u \bigr)\,.
    \end{equation}
    one can appeal to \cref{thm: Rodri Global W1p W2p regularity} and a standard bootstrap argument to prove the theorem.
\end{proof}
Next, we prove the main existence result for $\Lg$. 
Again, the conformal nature of the operator becomes apparent, as the positivity of the Yamabe invariant suffices in order to show invertibility. For general Schrödinger-type operators, purely analytic arguments suggest imposing positivity of the \textit{potential term} for a maximum principle type argument, or at least a smallness condition on its negative part to grant coercitivity --see, for instance \cite{DruetLaurain} or \cite{DruetHebeyRobert}--. 
\begin{theorem}
    \label{lemma: positive yamabe isomorphism _ final}
    Let $M$ be a smooth, closed manifold of dimension $n \geq 3$ and consider a $W^{k,q}$-Riemannian metric $g$ on $M$ with $k \geq 2$ and $q > \tfrac{n}{2}$. If $\lambda(M, g) > 0$, then the operator 
    \begin{equation*}
        \Lg : W^{k,p}(M) \to W^{k-2, p}(M)
    \end{equation*}
    is an isomorphism for every $1 < p \leq q$ and $k \geq 2$.
\end{theorem}
\begin{proof}
    We first observe that $\Lg : W^{k,p}(M) \to W^{k-2,p}(M)$ is Fredholm of index zero for $1 < p \leq q$ and $k \geq 2$ by \cref{lemma: Lg is Fredholm}, so it suffices to show injectivity. Since $W^{k,p}(M) \hookrightarrow W^{2, p}(M)$ for all $k \geq 2$, it is enough to prove injectivity of $\Lg : W^{2,p}(M) \to L^{p}(M)$ to conclude, and due to Theorem \ref{thm: conf Laplace elliptic regularity}, we are done once we show that $\Ker\bigl(\Lg |_{W^{2,q}(M)}\bigr) = 0$. To do so, suppose $u \in \Ker(\Lg|_{W^{2,q}(M)})$ is non-trivial and thus can be normalized so that $\|u\|_{L^{2^*}(M,d\mu_g)} = 1$. But then,
    \begin{equation*}
        0 = u\Lg u = a_n \int_M |\nabla u|^2 \, d\mu_g + \int_M \R_g u^2 \, d\mu_g = Q_g(u) \,,
    \end{equation*}
    and since $W^{2,q}(M) \hookrightarrow W^{1,2}(M)$ for $q > \tfrac{n}{2}$, it contradicts the positivity of the Yamabe invariant.
\end{proof}

A consequence of the previous results is a regularity statement for solutions to all the subcritical equations related to \eqref{eq: sPDE}.
\begin{corollary}
    \label{prop: subcritical regularity}
    Let $M$ be a smooth, closed manifold of dimension $n \geq 3$ and consider a $W^{k,p}$-Riemannian metric $g$ on $M$ with $k \geq 2$ and $q > \tfrac{n}{2}$. Suppose that $u \in W^{1,2}(M)$ is a non-negative weak solution of
    \begin{equation*}
        \Lg u = \lambda u^{s-1}
    \end{equation*}
    for some $s \in [2, 2^*)$ and $\lambda \in \nR$. Then, $u$ is either identically zero or is a strictly positive $W^{2,q}(M)$ function.
\end{corollary}
\begin{proof}
    We begin by showing that $u \in W^{2,q}(M)$. The Sobolev embedding $W^{1, 2}(M) \hookrightarrow L^{2^*}(M)$ ensures that
    \begin{equation*}
        \Lg u = \lambda u^{s-1} \, \in L^{\frac{2^*}{s-1}}(M) \,.
    \end{equation*}
    The hypothesis $u \in L^{q'}(M)$ of  \cref{thm: conf Laplace elliptic regularity} is satisfied since $L^{2^*}(M) \hookrightarrow L^{q'}(M)$ as $q' \leq 2^*$ or equivalently $q \geq \tfrac{2n}{n+2}$, which is guaranteed by $q > \tfrac{n}{2}$ and $n \geq 3$. If $\tfrac{2^*}{s-1} \geq q$, then $\Lg u \in L^q(M)$ and  the claim follows directly by applying \cref{thm: conf Laplace elliptic regularity}. Otherwise, for $\tfrac{2^*}{s-1} < q$,
    we obtain $u \in W^{2, \frac{2^*}{s-1}}(M)$. Let us then consider the embedding
    \begin{equation*}
        W^{2, \tfrac{2^*}{s-1}}(M) \hookrightarrow L^{r_1}(M) \,.
    \end{equation*}
    If $\tfrac{2^*}{s-1} \geq \frac{n}{2}$ or, equivalently, $s \leq \tfrac{n+2}{n-2}$, then the embedding works in particular for $r_1 = (s-1)q$, so that $\Lg u \in L^q(M)$ and the assertion follows from \cref{thm: conf Laplace elliptic regularity}. If $\tfrac{2^*}{s-1} < \frac{n}{2}$ instead, then the embedding holds for
    $\tfrac{1}{r_1} \doteq \tfrac{s-1}{2^*} - \tfrac{2}{n} \,>\, 0.$
    Observe that
    \begin{equation*}
        \label{ineq_s_2*}
        \frac{s-1}{2^*}- \frac{2}{n} < \frac{1}{2^*} \, \iff  \, \frac{n-2}{2n} (s-2) < \frac{2}{n}  \, \iff  \, s-2 < \frac{4}{n-2}  \, \iff  \, s < 2^*,
    \end{equation*}
    so it follows that $r_1 > 2^*$. Now, if $\tfrac{r_1}{s-1} \geq q$ we conclude by \cref{thm: conf Laplace elliptic regularity} since $\Lg u \in L^q(M)$. Otherwise,
    $\Lg u \in L^{\tfrac{r_1}{s-1}}(M)$ and \cref{thm: conf Laplace elliptic regularity} implies that $u \in W^{2, \tfrac{r_1}{s-1}}(M)$. As before, the claim follows if $\tfrac{r_1}{s-1} \geq \tfrac{n}{2}$. Otherwise, we iterate the procedure for a sequence 
    \begin{equation*}
        \frac{1}{r_k} \doteq \frac{s-1}{r_{k-1}} - \frac{2}{n}
    \end{equation*}
    until $\tfrac{r_k}{s-1} \geq \tfrac{n}{2}$, which is possible due to the following reasoning. First, by induction we can show that $r_k > 2^*$ for all $k \geq 1$. Since we have already computed $r_1 > 2^*$, let us assume $r_k > 2^*$ so that
     \begin{equation*}
         \frac{1}{r_{k+1}} = \frac{s-1}{r_{k}} - \frac{2}{n} < \frac{s-1}{2^*} - \frac{2}{n} = \frac{1}{r_1} < \frac{1}{2^*}
     \end{equation*}
     and the claim is proven. Consequently, the difference between $r_{k-1}$ and $r_{k}$ is a positive number independent of $k$, as it can be seen from
     \begin{align*}
         \frac{1}{r_{k-1}} - \frac{1}{r_k} &= -\frac{(s-2)}{r_{k-1}} + \frac{2}{n}
         \\
         &\geq -\frac{(s-2)}{2^*} + \frac{2}{n} = \frac{4 - (s-2)(n-2)}{2n} > 0 \,.
     \end{align*} 
     We remark that the strict positivity is ensured exactly because $s < 2^*$. By the reasoning above, we have established that $u \in W^{2,q}(M)$.
    
     Regarding positivity, Trudinger's Harnack inequality \cite[Theorem 5.2]{TrudingerMeasurableCoef} applies to

     \noindent
     our case and implies that if $u(\p) = 0$ at some $\p \in M$, then $u \equiv 0$ in any open coordinate chart around $\p$. In particular, $u^{-1}(0)$ is open. On the other hand, $u \in W^{2,q}(M)$ and hence continuous, implying that $u^{-1}(0)$ is closed. Therefore, it must be that either $u > 0$ or $u \equiv 0$ on $M$.
\end{proof}

\begin{remark}
    \label{rmk: conformal gauge}
    As explained in \cref{The Yamabe classification}, the Yamabe classification in Corollary \ref{cor: Yamabe classification} follows from the eigenvalue case  $(s = 2)$, which is a special case of \cref{prop: subcritical regularity}. In simple terms, this tells us that each $W^{2,q}$-conformal class contains a metric of continuous (in fact, $W^{2,q}$) scalar curvature. Such a representative could be regarded as a good conformal gauge and it will make the analysis in the following sections considerably simpler.
\end{remark}

\subsection{Normal and harmonic coordinates}
\label{subsec: normal coordinates}

The standard definition of normal coordinates via the exponential map requires uniqueness of solutions to the geodesic equation, which is only guaranteed for $C^{1,1}$-metrics. Below that threshold, there is a way around provided that $g \in C^1(T_2M)$. Following \cite[Definition 1.24]{AubinBook}, we say that a local coordinate system $\{x^i\}_{i=1}^n$ is a \emph{normal coordinate system} centered at $p \in M$, if $x^i(p) = 0$ and the components of the metric satisfy
\begin{equation}
    \label{conditions: normal_coords}
    g_{ij}(0) = \delta_{ij}, \qquad \partial_{k} g_{ij}(0) = 0 \qquad \text{ for all } \; i, j, k = 1, \dots, n \,.
\end{equation}
In particular, the Christoffel symbols vanish at $p$. The existence of such coordinate systems follows rather explicitly from an inverse function theorem argument, without even mentioning geodesics, see \cite[Proposition 1.25]{AubinBook}. Analyzing the proof, one notices that it only requires the metric to be $C^1$ and that the obtained normal coordinates are smoothly related to the original differential structure of $M$. Hence, we have 
\begin{proposition}[\cite{AubinBook}, Proposition 1.25]
    \label{NormalCoordsC1}
    Let $M$ be a smooth Riemannian manifold of dimension $n \geq 2$ and consider a $C^1$-Riemannian metric $g$ on $M$. Then, for each point $p \in M$ there exists a chart $(U,\varphi)$ with $\varphi(p)=0$ satisfying the following properties:
    \begin{enumerate}
        \item[$(i)$] the induced coordinates $\{x^i\}_{i=1}^n$ are normal centered at $p$ $($satisfy \eqref{conditions: normal_coords}$)$,
        \item[$(ii)$] $\varphi \in C^\infty(U,\nR^n)$. 
    \end{enumerate}
\end{proposition}

In particular, by Sobolev embeddings, \cref{NormalCoordsC1} allows us to work with normal coordinates around any point, whenever $g \in W^{2,q}(T_2M)$ with $q > n$. To go below this threshold, and in particular to allow for $q > \tfrac{n}{2}$, we will make use of harmonic coordinates. We say that a local coordinate system $\{x^i\}_{i=1}^n$ is a \textit{harmonic} if
\begin{equation}
    \label{conditions: harmonic_coords}
    \Delta_g x^i = 0 \qquad \text{for all } \; i = 1, \ldots, n \,.
\end{equation}
A simple computation shows that this implies that $g^{ij}\Gamma^k_{ij} = 0$. In contrast with \eqref{conditions: normal_coords}, the condition \eqref{conditions: harmonic_coords} is not at a given point, but rather in the whole domain of definition of $\{z^i\}_{i=1}^n$. Despite the construction of harmonic coordinates being much more subtle than that of normal coordinates, it carries over to rougher metrics. In the following proposition, we summarize existence and regularity properties of harmonic coordinates around arbitrary points in $M$. The existence part is a standard construction --see e.g. \cite[Chapter 3, \S 9]{TaylorToolsForPDEs} and references therein--. The regularity part is suited to our hypotheses and makes use of the elliptic regularity theory of \cref{SectionLapceOp}.
\begin{proposition}
    \label{prop: Existence and Regularity of harmonic coordinates}
    Let $M$ be a smooth manifold of dimension $n\geq 3$ and consider a $W^{k,q}$-Riemannia metric $g$ on $M$ with $k \geq 2$ and $q>\frac{n}{2}$. Then, for each $p \in M$ there exists a chart $(V,\psi)$ with $\psi(p)=0$ satisfying the following properties:
    \begin{enumerate}
        \item [$(i)$] the induced coordinates $\{x^i\}_{i=1}^n$ are harmonic $($satisfy \eqref{conditions: harmonic_coords}$)$,
        \item[$(ii)$] $g_{ij}(0) = \delta_{ij}$,
        \item[$(iii)$] $\psi \in W^{k+1,q}(V,\nR^n)$.
    \end{enumerate}
\end{proposition}
\begin{proof}
    Starting from an arbitrary chart $(U,\varphi)$ with $\varphi(p) = 0$, with associated coordinates $\{y^i\}_{i=1}^n$ chosen so that $g_{ij}(0) = \delta_{ij}$, the existence of a new chart $(V,\psi)$, $\psi: V \subset U \to B_{2\rho} \doteq B_{2\rho}(0)$, satisfying the harmonic condition $\Delta_{g}\psi^l=0$ for every $l=1,\ldots,n$ is a consequence, for instance, of \cite[Proposition 9.1]{TaylorToolsForPDEs} as long as $g \in C^{0,\alpha}(T_2M)$ for some $\alpha > 0$, where $\rho = \rho(n,\alpha,g)$. Hence, for our case of $g\in W^{2,q}(T_2M) \hookrightarrow C^{0,2-\frac{n}{q}}(T_2M)$ the existence of such a $\psi$ is granted with $\rho=\rho(n,q,g)$. Moreover, since the harmonic condition \eqref{conditions: harmonic_coords} is linear, we assume without loss of generality that 
    in the new coordinates $x^l \doteq \psi^l$ it holds $x^i(0) = 0$ and $g_{ij}(0) = \delta_{ij}$. This shows items ($i$) and ($ii$).

    Following the construction in \cite[Proposition 9.1]{TaylorToolsForPDEs} (see also \cite[Theorem 1.1]{Salo-pharm}), the harmonic coordinates $\{x^l\}_{l=1}^n$ are $W^{1,2}(B_{2\rho}) \, \cap \, C^{1,\alpha}(B_{2\rho})$ weak solutions of the elliptic boundary value problem
    \begin{equation}
        \label{harmonic eq in Brho}
        \begin{cases}
            \partial_{i}\bigl(a^{ij}(y)\partial_{j}x^l(y)\bigr) = 0 \qquad &\text{in} \quad B_{2\rho} \\
            x^l(y) = y^l &\text{on} \quad \partial B_{2\rho}
        \end{cases},
    \end{equation}
    where $a^{ij}(y) \doteq \sqrt{\det g(y)}\,g^{ij}(y) \in W^{k,q}(B_{2\rho}) \subset C^{0,\alpha}(B_{2\rho})$ by hypothesis and $\partial_i$ stands for derivatives with respect to $y^i$. In order to prove ($iii$) we need to bootstrap them to $W^{k+1,q}_{loc}(B_{2\rho})$. If we apply \cref{thm: main local elliptic regularity} with a vanishing right-hand-side we deduce that $z^l \in W^{k,q}_{loc}(B_{2\rho})$; thereby, $\{x^l\}_{l=1}^n$ admit second weak derivatives. Differentiating the equation in direction $y^m$ we obtain
    \begin{align}
        \label{HarmCoordBootstrap}
        \partial_i \bigl(a^{ij}\partial_j (\partial_m x^l)\bigr) = -\partial_i\partial_m a^{ij}\partial_jx^l - \partial_ma^{ij}\partial_i\partial_jx^l \doteq f \,,
    \end{align}
    or equivalently, $\Delta_g (\partial_m x^l) = (\sqrt{\det g})^{-1}\, f$.
    Concerning the regularity of the right-hand side, $(\sqrt{\det g})^{-1} \in W^{k, q}_{loc}(B_{2\rho})$, so we simply need to study the function $f$. Regarding the first term, using that $x^l \in C^{1, \alpha}(\overline{B_{2\rho}}) \cap  W_{loc}^{k, q}(B_{2\rho})$ and $\partial_i\partial_m a^{ij} \in W^{k-2, q}(B_{2\rho})$, together with the multiplication properties 
    \begin{align*}
        L^q(B_{2\rho}) \otimes C^{0, \alpha}(B_{2\rho}) &\hookrightarrow L^q(B_{2\rho}) \\
        W^{k-2, q}(B_{2\rho}) \otimes W_{loc}^{k-1, q}(B_{2\rho}) &\hookrightarrow W^{k-2, q}_{loc}(B_{2\rho}) \qquad \text{ for } k \geq 3 \,,
    \end{align*}
    we conclude that it belongs to $W^{k-2,q}_{loc}(B_{2\rho})$ for any $k \geq 2$. For the second term of the right-hand-side of \eqref{HarmCoordBootstrap}, we have $\partial_ma^{ij} \in W^{k-1,q}(B_{2 \rho})$ and $\partial_i\partial_j x^l \in W^{k-2,q}_{loc}(B_{2\rho})$. In the case $k \geq 3$, we get that  the product is in $W^{k-2, q}_{loc}(B_{2\rho})$ just as before. If $k = 2$ instead, we may use in the case $q > n$ the multiplication
    \begin{equation*}
        W^{1, q}(B_{2\rho}) \otimes L_{loc}^q(B_{2\rho}) \hookrightarrow L_{loc}^q(B_{2\rho})
    \end{equation*}
    to see that the right-hand side is $L^q_{loc}(B_{2\rho}) = W^{k-2, q}_{loc}(B_{2\rho})$. In the case $q = n$, then we can in particular consider the embedding $W^{1,n}(B_{2\rho}) \otimes L^n_{loc}(B_{2\rho}) \hookrightarrow L^{\frac{3n}{4}}_{loc}(B_{2\rho})$ 
    so that the right-hand side of \eqref{HarmCoordBootstrap} is in $ L^{\frac{3n}{4}}_{loc}(B_{2\rho})$ and \cref{thm: main local elliptic regularity} gives $x^l \in W^{3,\frac{3n}{4}}_{loc}(B_{2\rho})$. This implies, in turn, that the second term of $f$ is now in 
    \begin{equation*}
        W^{1, n}(B_{2\rho}) \otimes W^{1, \frac{3n}{4}}_{loc}(B_{2\rho}) \hookrightarrow L^n_{loc}(B_{2\rho})
    \end{equation*}
    and the right-hand side of \eqref{HarmCoordBootstrap} is $L^n_{loc}(B_{2\rho}) =L^q_{loc}(B_{2\rho}) = W^{k-2, q}_{loc}(B_{2\rho})$. Last, in the case $q < n$, we use the multiplication property
    \begin{equation*}
        W^{1,q}(B_{2\rho}) \otimes L_{loc}^q(B_{2\rho}) \hookrightarrow L_{loc}^{p_0}(B_{2\rho}) \quad \text{with} \quad \frac{1}{p_0} \doteq \frac{2}{q} - \frac{1}{n} \,,
    \end{equation*}
    so that the right-hand side of \eqref{HarmCoordBootstrap} is in $ L^{p_0}_{loc}(B_{2\rho})$ and \cref{thm: main local elliptic regularity} implies that $x^l \in W^{3,p_0}_{loc}(B_{2\rho})$. This gives, in turn, that the second term of $f$ is now in
    \begin{equation*}
        W^{1,q}(B_{2\rho}) \otimes W^{1,p_0}_{loc}(B_{2\rho}) \hookrightarrow L^{p_1}_{loc}(B_{2\rho}) \quad \text{with} \quad \frac{1}{p_{1}} \doteq \frac{1}{q} - \frac{2}{n} + \frac{1}{p_0} \,.
    \end{equation*}
    Observe that thanks to our hypothesis $q > \tfrac{n}{2}$,
    \begin{equation*}
        \frac{1}{p_0} - \frac{1}{p_1} = \frac{2}{n} - \frac{1}{q} > 0 \,.
    \end{equation*}
    Repeating the process iteratively, we find a non-accumulating sequence given by $\tfrac{1}{p_{i+1}} \doteq \tfrac{1}{q} - \tfrac{2}{n} + \tfrac{1}{p_i}$ and thus find a large enough $i \in \nN$ such that $p_i \geq q$. 
    
    We conclude that $f \in W^{k-2, q}(B_{3\rho/2})$ in all cases, so we can use \cref{thm: main local elliptic regularity} to obtain that $x^l \in W^{k+1,q}(B_\rho)$, as desired.
\end{proof}
We remark that, unlike normal coordinates, harmonic coordinates are not smoothly related to the original differential structure of $M$, as they are only $W^{k+1,q}$-compatible. This suggests that if we were to consider a harmonic chart as part of the differentiable structure of $M$, this would only be a differentiable structure of limited Sobolev regularity. On the other hand, \cref{lemma: Adams diffeo lemma} immediately implies that if a function $u$ is $W^{l,p}$-regular in harmonic coordinates, then it is also $W^{l,p}$-regular in the original differentiable structure provided that $l \leq k+1$ and $1 \leq p \leq q$. Likewise for any tensor field due to \cref{prop: tensor_regularity}. More precisely, we have
\begin{corollary}
    \label{cor: reg change harmonic to smooth}
    Let $M$ be a smooth manifold of dimension $n \geq 3$ and consider a $W^{k,q}$-Riemannian metric $g$ on $M$ with $k \geq 2$ and $q > \tfrac{n}{2}$. Let $(V,\psi)$ be a harmonic chart on $M$ like that constructed in \cref{prop: Existence and Regularity of harmonic coordinates}.
    \begin{enumerate}
        \item[$(i)$] For $l \leq k+1$ and $1 \leq p \leq q$, there holds that $u \circ \psi^{-1} \in W^{l,p}_{loc}(\psi(V))$ if and only if $u \in W^{l,p}(V)$.
        \item[$(ii)$] For $l \leq k$ and $1 \leq p \leq q$, there holds that $\psi_*u \in W^{l,p}_{loc}(T_s\psi(V))$ if and only if $u \in W^{l,p}(T_sV)$.
    \end{enumerate}
\end{corollary}

\section{The conformal Green function}
\label{The Conformal Green's function}
Before presenting the main result of this section, let us introduce some definitions and subtleties related to the conformal Green function 
for rough metrics. Recall that if $g \in C^\infty(T_2M)$, then $\mathscr{L}_g$ acts by duality as a map $\mathscr{D}'(M)\to \mathscr{D}'(M)$ and a Green function $G_p \in \mathscr D'(M)$ of $\Lg$ centered at $p \in M$ is a solution to
\begin{equation*}
    \Lg G_p = \delta_p \,,
\end{equation*}
where $\delta_p \in \mathscr D'(M)$ is defined by
\begin{equation}
    \label{eq: delta_p def smooth}
    \delta_p(\phi) = \phi(p) \quad \text{for all } \: \phi \in \mathscr D(M) \,.
\end{equation}
In contrast, if $g$ is not smooth, then $\mathscr{L}_g$ cannot be made to act naturally in this way on all of $\mathscr{D}'(M)$. However, using multiplication properties of \cite[Theorem 2.2]{avalos2024sobolev} one gets the following.
\begin{lemma}
    \label{lemma: conf Lap for Lt}
    Let $M$ be a smooth, closed manifold of dimension $n \geq 3$ and consider a $W^{2,q}$-Riemannian metric $g$ with $q > \tfrac{n}{2}$. Then, the operator $\Lg : C^\infty(M) \to L^q(M)$ extends by duality to a bounded map
    \begin{equation*}
        \Lg : L^t(M) \to W^{-2,t}(M)
    \end{equation*}
    provided that $t \geq q'$. Moreover, 
    \begin{align}
        \langle \mathscr{L}_gu,\phi \rangle_{(M,g)}=\langle u,\mathscr{L}_g\phi \rangle_{L^2(M,d\mu_g)}
    \end{align}
    holds for any $u \in L^t(M)$ and $\phi \in W^{2,t'}(M)$.
\end{lemma}
\begin{proof}
    The first part of the statement is a special case of \cite[Lemma 2.2]{avalos2024sobolev}. For the self-adjointness, consider sequences of functions $\{u_j\}_{j=1}^{\infty}\subset C^{\infty}(M)$ and metrics $\{g_j\}_{j=1}^\infty \subset C^\infty(T_2M)$ converging $u_j \to u$ in $L^t(M,d\mu_g)$ and $g_j \to g$ in $W^{2,q}(T_2M)$ as $j \to \infty$ and notice that
    \begin{align*}
        \langle \mathscr{L}_gu,\phi \rangle_{(M,g)} &= \lim_{j\to\infty}\langle \mathscr{L}_{g_j}u_j,\phi \rangle_{L^2(M,d\mu_g)} = \lim_{j\to\infty}\langle \mathscr{L}_{g_j}u_j,\phi \rangle_{L^2(M,d\mu_{g_j})} 
        \\
        &= \lim_{j\to\infty}\langle u_j,\mathscr{L}_{g_j}\phi \rangle_{L^2(M,d\mu_{g_j})} = \langle u,\mathscr{L}_g\phi \rangle_{L^2(M,d\mu_g)} \,,
    \end{align*}
    where the last identity is a direct consequence of Hölder's inequality.
\end{proof}
Another observation is that \eqref{eq: delta_p def smooth} makes sense as long as $\phi$ is continuous and consequently, in view of \cref{SobolevDualIso}, we can define the distribution $\delta_p \in W^{-2,t}(M)$ by
\begin{equation}
    \label{eq: delta_p def rough}
    \langle\delta_p, \phi\rangle_{(M,g)} = \phi(p) \quad \text{for all } \: \phi \in W^{2,t'}(M)
\end{equation}
provided that $t' > \tfrac{n}{2}$, or equivalently, $t < \tfrac{n}{n-2}$. Since $q > \tfrac{n}{2}$ is equivalent to $q' < \tfrac{n}{n-2}$, we are lead to
\begin{definition}
    \label{GreenFunctionDefn}
    Let $M$ be a smooth, closed manifold of dimension $n \geq 3$ and consider a $W^{2,q}$-Riemannian metric $g$ with $q > \frac{n}{2}$. For $q' \leq t < \tfrac{n}{n-2}$ we say that $G_p\in L^t(M)$ is a Green function for $\Lg$ with a pole at $p \in M$ if it satisfies
    \begin{equation}
        \langle\Lg G_p, \phi\rangle_{(M,g)} = \phi(p)
    \end{equation}
    for all $\phi \in W^{2,t'}(M)$.
\end{definition}
An important direct consequence of \cref{GreenFunctionDefn} and \cref{lemma: conf Lap for Lt}, which mimics the known results for smooth metrics, is the following conformal transformation property for $G_p$.
\begin{proposition}
    \label{prop: conformal Green transformation}
    Let $M$ be a smooth, closed manifold of dimension $n \geq 3$ and let $g$ be a $W^{2,q}$-Riemannian metric with $q > \tfrac{n}{2}$. Consider a positive function $u \in W^{2,q}(M)$ and the conformal metric $\tilde g \doteq u^{\frac{4}{n-2}}g$. If $G_p$ is a Green function for $\Lg$ centered at $p \in M$, then
    \begin{equation}
        \label{eq: Green transformation rule}
        \widetilde G_p \doteq u^{-1}(p)u^{-1}G_p
    \end{equation}
    is a Green function for $\mathscr{L}_{\bar g}$ centered at $p$.
\end{proposition}
\begin{proof}
    First, due the hypotheses on $u$ and $q$, we know that $u^{-1}$ is a positive continuous function on $M$, which implies that $\widetilde{G}_{p}\in L^t(M)$. Then, for arbitrary $\phi \in W^{2,t'}(M)$ \cref{lemma: conf Lap for Lt} gives
    \begin{align*}
        \bigl\langle \mathscr L_{\tilde{g}} \widetilde{G}_{p}, \phi \bigr\rangle_{(M, \tilde{g})} &= \int_{M} \widetilde{G}_{p} \, \mathscr L_{\widetilde{g}} \,\phi \,\, d\mu_{\tilde{g}} = u(p)^{-1} \int_M G_{p} \Lg(u \phi) \,\, d\mu_g \\
        &= u(p)^{-1} \bigl\langle \Lg G_{p}, u \phi \bigr\rangle_{(M,g)} 
        = \phi(p) \,,
    \end{align*}
    where we have used that $u \phi\in W^{2,t'}(M)$ due to $t'\leq q$. 
\end{proof}
We can now tackle the main result of the section.
\begin{replettertheorem}{theorem C}
    \label{thm: greens existence}
    \emph{
    Let $M$ be a smooth, closed manifold of dimension $n \geq 3$ and consider a $W^{2,q}$-Riemannian metric $g$ with $q > \tfrac{n}{2}$ and $\lambda(M, g) > 0$. Then, for each point $p \in M$, the conformal Laplacian $\Lg$ admits a unique, positive Green function $G_{p} \in W_{loc}^{2, q}\bigl(M\setminus\{p\}\bigr)$ centered at $p$. Moreover, there exists a conformal harmonic chart\footnote{This \textit{conformal harmonic chart} refers to harmonic coordinates of a conformal metric $\tilde{g} \in [\,g\,]_{W^{2,q}}$ chosen so that $\R_{\tilde{g}}$ is continuous and positive. The existence of such metric is ensured by \cref{cor: Yamabe classification}.} $(U,x^i)$ around $p$ such that
    \begin{equation}
        \label{CGF_Expansion}
         G_{p}(x) = \frac{B}{|x|^{n-2}}  + h(x) 
    \end{equation}
    for some $B>0$ and a function $h \in W^{2,r}(U)$ satisfying
   \begin{equation}
        \label{h_decay}
        h(x) = A + \BO\bigl(|x|^{2-\frac{n}{r}}\bigr) \quad \text{as} \quad |x| \to 0
    \end{equation}
    for each $1 \leq r < \tfrac{nq}{q(n-2) + n}$ and some $A \in \nR$. If $q > n$, there additionally holds
    \begin{equation}
        \label{h_decay_qn}
        h(x) = A + \BO_1\bigl(|x|^{2-\frac{n}{r}}\bigr) \quad \text{as} \quad |x| \to 0
    \end{equation}
    for each $1 \leq  r < \tfrac{nq}{q(n-2) + n}$ and harmonic coordinates can be substituted by normal coordinates.}
\end{replettertheorem}
\begin{proof}
    First, we claim that due to \cref{prop: conformal Green transformation}, it suffices to show the theorem for some element in the conformal class $[\,g\,]_{W^{2,q}}$. Then, existence, uniqueness, regularity and positivity of any other conformal metric follow directly from the transformation rule \eqref{eq: Green transformation rule} of the Green function under conformal changes. For the expansion near the pole, observe that \cref{cor: reg change harmonic to smooth} implies that if $u \in W^{2,q}(M)$, then also $u(x) \in W^{2,q}(\psi(V))$ in a $g$-harmonic coordinate chart $(V,\psi)$ centered at $p$ with coordinates $\{x^i\}_{i=1}^n$. In particular, $u(x)$ is Hölder continuous and therefore 
    \begin{equation*}
        u(x) = u(0) + O\bigl(|x|^{2-\frac{n}{q}}\bigr) \quad \text{as} \quad |x| \to 0 \,.
    \end{equation*}
    Consequently, using \eqref{eq: Green transformation rule}, \eqref{CGF_Expansion} and \eqref{h_decay} we obtain that the Green function $\widetilde{G}_p$ associated to the conformal metric $\tilde g \doteq u^{\frac{4}{n-2}}$ satisfies
    \begin{align*}
        \widetilde{G}_{0}(x) & = \frac{1}{u(0)} \frac{1}{u(0) + \BO\bigl(|x|^{2- \frac{n}{q}}\bigr)} \left( \frac{B}{|x|^{n-2}} + A + \BO\bigl(|x|^{2-\frac{n}{r}}\bigr)\right) 
        \\
        & = \frac{1}{u(0)^2}  \, \left( \frac{B}{|x|^{n-2}} + A + \BO\bigl(|x|^{2-\frac{n}{r}}\bigr) \right) + \BO\bigl(|x|^{2-\frac{n}{q}}\bigr) + \BO\bigl(|x|^{4-\frac{n}{q} -n}\bigr) + \BO\bigl(|x|^{4-\frac{n}{q} - \frac{n}{r}}\bigr) 
        \\
        & = \frac{1}{u(0)^2}  \, \left(\frac{B}{|x|^{n-2}} + A + \BO\bigl(|x|^{2-\frac{n}{r}}\bigr) \right)
    \end{align*}
    as $|x| \to 0$, where the last identity holds since for all $n \geq 3$,  $q > \tfrac{n}{2}$ and $1 \leq r < \tfrac{nq}{q(n-2) + n}$ we have
    \begin{equation*}
        2 - \frac{n}{r} < 4-\frac{n}{q} - n \leq 4- \frac{n}{q} - \frac{n}{r} \,, \qquad 2 - \frac{n}{r} < 2 - \frac{n}{q} \,.
    \end{equation*}
    If $q > n$, notice that $u$ is furthermore $C^{1,2-\frac{n}{q}}$-regular over the pole in harmonic coordinates and consequently it admits the expansion
    \begin{equation*}
        u(x) = u(0) + \BO_1\bigl(|x|^{2-\frac{n}{q}}\bigr) \quad \text{as} \quad |x| \to 0 \,.
    \end{equation*}
    The same expansion holds in $g$-normal coordinates due to item $(ii)$ of \cref{NormalCoordsC1}. In any case, we have that
    \begin{align*}
        \widetilde{G}_{p}(x) & = \frac{1}{u(0)^2}  \, \left(\frac{B}{|x|^{n-2}} + A + \BO_1\bigl(|x|^{2-\frac{n}{r}}\bigr) \right)
    \end{align*}
    and the claim is proven.

    In light of \cref{cor: Yamabe classification}, the above claim allows us to assume that $g$ has positive scalar curvature $\R_g \in W^{2,q}(M)$.

    \textit{Step 1: Existence.} Consider a harmonic coordinate chart $(V,\psi)$ centered at $p$ with coordinates $\{x^i\}_{i=1}^n$, which exists by \cref{prop: Existence and Regularity of harmonic coordinates} and satisfy
    \begin{equation}
        \label{eq: harmonic properties 2}
        \psi(p) = 0 \,, \quad g_{ij}(0) = \delta_{ij} \,, \quad g^{ij}\Gamma^k_{ij} \equiv 0 \,.
    \end{equation}
    Let $U \subset\subset W \subset\subset V$ be smooth neighborhoods of $p$ and consider a cut-off function $\eta\in C^{\infty}_0(W)$ such that $\eta\equiv 1$ in $U$. We define
    \begin{align}
        \label{GreenFunct0}
        u(x) \doteq 
        \begin{dcases}
            b_n\eta(x) |x|^{2-n} &\text{in} \; V
            \\
            0  &\text{in} \; M \setminus V
        \end{dcases}
        ,
    \end{align}
    with $b_n \doteq \frac{1}{4(n-1)|\nS^{n-1}|}$. Since the function $|x|^{2-n}$ is in $L^t$ for all $t < \tfrac{n}{n-2}$ in bounded regions around the origin, it is not hard to check that \cref{cor: reg change harmonic to smooth} then implies that $u \in L^t(M)$ for all $t < \tfrac{n}{n-2}$. As $q' < \tfrac{n}{n-2}$ due to $q > \tfrac{n}{2}$, the function $\Lg u$ is well-defined in regard of \cref{lemma: conf Lap for Lt}. In order to compute $\Lg u$, consider some $\phi \in W^{2,t'}(M)$ and suppose first that $\supp\phi \subset\subset V$. Using \cref{lemma: conf Lap for Lt},
    \begin{align*}
        \langle &\Lg u, \phi\rangle_{(M,g)} = \int_{V}u\Lg\phi \, d\mu_g = \int_{\psi(V)} u(x) (\Lg\phi)(x) \sqrt{\det g(x)} \, dx
        \\
        &
        = -a_n\int_{\psi(V)} u(x) \partial_i\bigl(\sqrt{\det g(x)}g^{ij}(x)\partial_j\phi(x)\bigr) dx + \int_{\psi(V)} u(x) \R_g(x)\phi(x) \sqrt{\det g(x)} \, dx
        \\
        &
        = \Bigl\langle u(x), -a_n\sqrt{\det g(x)}\,\Delta_{g(x)}\phi(x) + \R_g(x)\phi(x) \sqrt{\det g(x)}\Bigr\rangle_{\psi(V)}
        \\
        &
        = \Bigl\langle\mathscr L_{g(x)} u(x), \sqrt{\det g(x)} \phi(x)\Bigr\rangle_{\psi(V)} \,,
    \end{align*}
    where $\langle\cdot,\cdot\rangle_\Omega :W^{-k,p'}(\Omega) \times W^{k,p}_0(\Omega) \to \nR$ for $\Omega\subset\nR^n$ stands for the standard duality pairing. It should be mentioned that in the second equality one would, by definition, write the integral in a smooth chart of $M$. Since harmonic coordinates are $W^{3,q}$-compatible with the smooth differentiable structure of $M$, however, the change to harmonic coordinates is justified. For general $\phi \in W^{2,t'}(M)$, we may consider a smooth partition of unity $\{\chi_\alpha\}_{\alpha=1}^N$ such that $\chi_1 \equiv 1$ on $W$ and $\supp\chi_1 \subset\subset V$. Then, since $\supp u \cap \supp\chi_\alpha$ is empty for all $\alpha \neq 1$,
    \begin{align}
        \label{eq: eq: Lg dual local action}
        \langle\Lg u, \phi\rangle_{(M,g)} &= \sum_{\alpha=1}^N \langle\Lg u, \chi_\alpha\phi\rangle_{(M,g)} = \langle\Lg u, \chi_1\phi\rangle_{(M,g)} + \sum_{\alpha=2}^N \int_M u \Lg(\chi_\alpha\phi) d\mu_g \nonumber
        \\
        &= \Bigl\langle\mathscr L_{g(x)} u(x), \sqrt{\det g(x)} \,\chi_1(x)\phi(x)\Bigr\rangle_{\psi(V)} \,.
    \end{align}
    Now, using the properties \eqref{eq: harmonic properties 2} of harmonic coordinates (and simplifying notation), we formally compute
    \begin{align*}
        \Delta_g\bigl(\eta|x|^{2-n}\bigr) &= g^{ij}\,\partial_i\partial_j\bigl(\eta|x|^{2-n}\bigr) 
        \\
        & = \eta \, g^{ij} \partial_i \partial_j |x|^{2-n} +  g^{ij} \partial_i \partial_j\eta \, |x|^{2-n} + 2 g^{ij} \partial_i \eta \, \partial_j |x|^{2-n}
        \\
        & = \eta \Delta |x|^{2-n} + \eta \, \bigl(g^{ij} - \delta^{ij}\bigr)  \partial_i \partial_j |x|^{2-n} +  g^{ij} \Big( \partial_i \partial_j\eta \, |x|^{2-n} + 2 \partial_i \eta \, \partial_j |x|^{2-n} \Big) \,.
    \end{align*}
    Since the last two terms involve derivatives of $\eta$, which are supported away from $\psi(U)$, we have that
    \begin{equation*}
        f_1 \doteq g^{ij} \Big( \partial_i \partial_j\eta \, |x|^{2-n} + 2 \partial_i \eta \, \partial_j |x|^{2-n} \Big) \in W^{2,q}(\psi(V)) \subset C^{0, 2 - \frac{n}{q}}(\psi(V)) \,.
    \end{equation*}
    Here we are using that $g_{ij}(x) \in W^{2,q}(\psi(V))$ by \cref{cor: reg change harmonic to smooth}. On the other hand, a simple computation gives
    \begin{align*}
        \partial_{i}\partial_j|x|^{2-n} = (2-n)\left(\frac{\delta_{ij}}{|x|^n} - n \frac{x^ix^j}{|x|^{n+2}} \right) = \BO(|x|^{-n}) \quad \text{as} \quad |x| \to 0 \,,
    \end{align*}
    while the Hölder continuity of the metric components, together with \eqref{eq: harmonic properties 2}, implies
    \begin{equation*}
        g^{ij}-\delta^{ij} = \BO\bigl(|x|^{2-\frac{n}{q}}\bigr) \quad \text{as} \quad |x| \to 0 \,.
    \end{equation*}
    It follows that
    \begin{equation*}
        f_2 \doteq \eta \, (g^{ij} - \delta^{ij})  \partial_i \partial_j |x|^{2-n} = \BO(|x|^{-\gamma}) \quad \text{as} \quad |x| \to 0
    \end{equation*}
    for $\gamma \doteq n + \tfrac{n}{q} - 2$. Finally, we use that $\R_g \in W^{2,q}(M)$ in our good conformal gauge
    
    \noindent and, by \cref{cor: reg change harmonic to smooth}, $\R_g$ is continuous in harmonic coordinates too, implying that
    \begin{equation*}
        f_3 \doteq \R_g \eta \, |x|^{2-n} = \BO\bigl(|x|^{2-n}\bigr) \quad \text{as} \quad |x| \to 0 \,.
    \end{equation*}
    Combining the above, we obtain
    \begin{equation*}
        \begin{split}
            \Lg\bigl(\eta \, |x|^{2-n}\bigr) &= -a_n \Delta_g \bigl(\eta \, |x|^{2-n}\bigr) + \R_g \eta \, |x|^{2-n}
            \\
            & = -a_n \eta\Delta\bigl(|x|^{2-n}\bigr) - a_n \, f_1(x) - a_n \, f_2(x) + f_3(x) 
            \\
            &\doteq -a_n \eta\Delta\bigl(|x|^{2-n}\bigr) + f(x) \,,
        \end{split}
    \end{equation*}
    with $f(x) = \BO(|x|^{-\gamma})$ as $|x| \to 0$.\footnote{If one works with normal coordinates in the $q > n$ case, an extra $g^{ij}\Gamma^k_{ij}\partial_k\bigl(\eta |x|^{2-n}\bigr)$ term shows up in the definition of $f$. However, since $\Gamma^k_{ij} \in W^{1,q}(V) \subset C^{0,1-\frac{n}{q}}$ with $\Gamma^k_{ij}(0) = 0$, it follows that $\eta g^{ij}\Gamma^k_{ij}\partial_k|x|^{2-n} = \BO\bigl(|x|^{1-\frac{n}{q}}\bigr)\BO\bigl(|x|^{1-n}\bigr) = \BO\bigl(|x|^{-\gamma}\bigr)$ and the qualitative behavior of $f$ is unaffected.} Observe that $|x|^{-\gamma} \in L^r_{loc}(\mathbb{R}^n)$ as long as $n - \gamma r > 0$, or equivalently, 
    \begin{equation*}
        r < \frac{n}{\gamma} = \frac{nq}{q(n-2) + n} \,,
    \end{equation*}
    so there holds that $f(x)  \in L^r(\psi(V))$ for all $r < \tfrac{nq}{q(n-2) + n}$. We remark that $\frac{nq}{q(n-2) + n} > 1$ is guaranteed by $q > \tfrac{n}{2}$. From \cref{cor: reg change harmonic to smooth} we also get that $f \in L^r(V)$, which we extend by zero to the whole manifold and still call it $f$. In light of \eqref{eq: eq: Lg dual local action}, we establish that for any $\phi \in W^{2,t'}(M)$ with $t < \tfrac{n}{n-2}$ there holds
    \begin{align*}
        \langle \Lg u, \phi\rangle_{(M,g)} &= \Bigl\langle-a_n \eta\Delta\bigl(|x|^{2-n}\bigr) + f(x), \sqrt{\det g(x)}\,\chi_1(x) \phi(x)\Bigr\rangle_{\psi(V)}
        \\
        &= \frac{-1}{(n-2)|\nS^{n-1}|}\int_{\psi(V)}|x|^{2-n} \Delta\Bigl(\eta(x) \chi_1(x) \phi(x) \sqrt{\det g(x)}\Bigr) dx 
        \\
        &\hspace{3cm} + b_n\int_{\psi(V)} f(x)\chi_1(x)\phi(x) d\mu_g(x)
        \\
        &= \phi(p) + b_n\int_M f \phi \,d\mu_g \,.
    \end{align*}
    Since $f \in L^r(M)$ for all $r < \tfrac{nq}{q(n-2) + n} \leq q$, we may apply \cref{lemma: positive yamabe isomorphism _ final} to obtain a unique solution $h \in W^{2,r}(M)$ to the equation $\Lg h = -b_n f$. We claim that
    \begin{equation}
        \label{eq: global Gp}
        G_{p} \doteq u + h 
    \end{equation}
    is a Green function of $\Lg$ with pole $p$. If $r$ can be chosen greater or equal than $\tfrac{n}{2}$, then $h \in L^t(M)$ for all $t<\infty$ and consequently $G_p \in L^t(M)$ for all $t < \tfrac{n}{n-2}$. If $r<\frac{n}{2}$ one has $W^{2,r}(M)\hookrightarrow L^s(M)$ for all $s$ such that 
    \begin{align*}
        \frac{1}{s} = \frac{1}{r}-\frac{2}{n}>\frac{n-2}{n}+\frac{1}{q} - \frac{2}{n}=1 + \frac{1}{q}-\frac{4}{n}.
    \end{align*}  
    That is, $h \in L^s(M)$ for some $q'\leq s$ if and only if
    \begin{align*}
        \frac{1}{q'}=1-\frac{1}{q}>1 + \frac{1}{q}-\frac{4}{n}\Longleftrightarrow \frac{2}{n}>\frac{1}{q},
    \end{align*}
    which is satisfied by hypothesis. We conclude that $G_p\in L^t(M)$ for some $q'\leq t< \frac{n}{n-2}$. Moreover, for all $\phi \in W^{2,t'}(M)$ there holds
    \begin{align*}
        \langle \Lg G_p, \phi\rangle_{(M,g)} &= \langle \Lg u, \phi\rangle_{(M,g)} + \langle \Lg h, \phi\rangle_{(M,g)}
        \\
        &= \phi(p) + b_n \int_M f \phi \, d\mu_g - b_n\langle f, \phi\rangle_{(M,g)} = \phi(p) 
    \end{align*}
    as desired.
    
    \medskip
    \textit{Step 2: Uniqueness.} Suppose now there exists another Green function $\widetilde G_p \in L^t(M)$, $t \geq q'$, and let $H = G_p - \widetilde G_p$. Then, by definition,
    \begin{align*}
        \bigl\langle\Lg H, \phi\bigr\rangle_{(M,g)} = \bigl\langle\Lg G_p, \phi\bigr\rangle_{(M,g)} - \bigl\langle\Lg \widetilde G_p, \phi\bigr\rangle_{(M,g)} = 0
    \end{align*}
    holds for all $\phi \in W^{2,t'}(M)$ and \cref{SobolevDualIso} implies that $\Lg H = 0$. Since $H \in L^{q'}(M)$ in particular, \cref{thm: conf Laplace elliptic regularity} implies that $H \in W^{2,q}(M)$, but in view of the hypothesis $\lambda(M,g) > 0$ \cref{lemma: positive yamabe isomorphism _ final} implies that $H = 0$. Hence $G_p = \widetilde G_p$.

    \medskip
    \textit{Step 3: Regularity.} Let us now study the regularity of $G_{p}$ away from the singular point. By the definition \eqref{GreenFunct0}, we know that $u$ is smooth away from $p$ in harmonic coordinates, so \cref{cor: reg change harmonic to smooth} implies that $u \in W^{3,q}(M \setminus B_{\varepsilon}(p))$  for any $\varepsilon > 0$. Fix some small $\varepsilon > 0$ and consider a smooth cut-off function $\xi$ supported in $M\setminus B_{\varepsilon}(p)$ and $\xi \equiv 1$ on $M\setminus B_{2\varepsilon}(p)$. Then,  $\xi G_{p} \in W^{2,r}(M)$ due to the regularity of $h$ and it satisfies
    \begin{align*}
        \Lg \bigl(\xi \, G_{p} \bigr) &= -a_n \Delta_g \bigl(\xi \, G_{p}\bigr) + \R_g \xi \, G_{\p} 
        \\
        & = - \xi \Lg G_{p} - 2a_n \,g(\nabla \xi, \nabla G_{p}) - a_n \,\Delta_g \xi \,G_{p} + \R_g \xi \, G_{p} 
        \\
        & =  -2a_n \,g(\nabla \xi, \nabla G_{p}) - a_n \, \Delta_g\xi \,G_{p} + \R_g \xi \, G_{p} 
    \end{align*}
    From our analysis above, we know that $W^{2, r}(M) \hookrightarrow L^{q'}(M)$ and thus $\xi G_{\p} \in L^{q'}(M)$. In addition, since $r < \tfrac{nq}{q(n-2)+n} < n$, we may use the embedding
    \begin{equation*}
        \xi G_{\p} \in W^{2, r}(M) \hookrightarrow W^{1,r_1}(M) \qquad \text{with} \qquad \frac{1}{r_1} \doteq \frac{1}{r} - \frac{1}{n} \,,
    \end{equation*}
    to get that  $g(\nabla \xi, \nabla G_{p}) \in L^{r_1}(M)$ and consequently $\Lg\bigl(\xi\,G_{p}\bigr) \in L^{r_1}(M)$ (recall that $\R_g \in W^{2,q}(M)$). If $r_1 \geq q$ \cref{thm: conf Laplace elliptic regularity} immediately implies that $\xi G_{p} \in W^{2,q}(M)$. If not, it allows us to conclude that  $\xi G_{p} \in W^{2,r_1}(M)$. If $r_1 \geq \tfrac{qn}{q+n}$, then $W^{2,r_1}(M) \hookrightarrow W^{1,q}(M)$ and we are done by the same reasoning. If $r_1 \leq \tfrac{qn}{q+n} < q$, we have that
    \begin{equation*}
        \xi G_{p} \in W^{2,r_1}(M) \hookrightarrow W^{1,r_2}(M) \qquad \text{with} \qquad \frac{1}{r_2} \doteq \frac{1}{r_1} - \frac{1}{n}
    \end{equation*}
    and we can repeat the reasoning. In fact, we can iterate the argument considering for each integer $k \geq 1$ the Sobolev embeddings
    \begin{equation*}
        W^{1, r_k}(M) \hookrightarrow L^{r_{k+1}}(M) \qquad \text{with} \qquad \frac{1}{r_{k+1}} \doteq \frac{1}{r_k} - \frac{1}{n}
    \end{equation*} 
    and using \cref{thm: conf Laplace elliptic regularity} to conclude that $\xi G_{p} \in W^{2, r_{k+1}}(M)$. Since $\tfrac{1}{r_{k}} - \tfrac{1}{r_{k+1}} = \tfrac{1}{n}$, it follows that $r_k$ reaches $q$ in finitely many steps, showing that $\xi G_{p} \in W^{2,q}(M)$. As $\varepsilon > 0$ was arbitrary, we establish that $G_{p} \in W^{2,q}_{loc}\bigl(M\setminus\{p\}\bigr)$.

    \medskip
    \textit{Step 4: Positivity and expansion.} We claim that the positivity of $G_{p}$ follows from the expansion \eqref{h_decay} of $h$. In fact, in view of \eqref{eq: global Gp}, if $h$ blows up strictly slower than $|x|^{2-n}$ in the harmonic coordinates $\{ x^i\}_{i=1}^n$, then there is some $\varepsilon > 0$ small enough such that $G_{p} > 1$ on $\overline{B_\varepsilon(p) }\setminus \{p\}$. Suppose by contradiction that there is some $\q \in M \setminus \overline{B_\varepsilon(p) }$ such that 
    \begin{equation*}
        \inf_{M \setminus B_\varepsilon(p)} G_{p} = G_{p}(\q) \doteq \m \leq 0 \,.
    \end{equation*}
    Then, on any coordinate chart $(W, y^i)$ around $\q$ with $W \cap B_\varepsilon(p) = \emptyset$ there holds
    \begin{equation*}
        \sqrt{\det g} \, \Lg G_{p} = -a_n\partial_{i}\Big(\sqrt{\det g}\, g^{ij}\partial_{j}G_{p}\Big) + \sqrt{\det g}\R_g G_{p} = 0
    \end{equation*}
    with $\R_g > 0$ by hypothesis, so the strong maximum principle \cite[Theorem 8.19]{GT} implies that $G_{p} \equiv \m \leq 0$ on $W$. This shows that $G_{p}^{-1}(\m)$ is open in $M\backslash\overline{B_\varepsilon(p)}$. On the other hand, it is also closed due to the continuity of $G_{p}$ and therefore equal to $M \setminus\overline{B_\varepsilon(p) }$. However, this contradicts $G_{p}|_{\partial B_\varepsilon(p)} > 1$ and it follows that $G_{p}$ is positive provided that \eqref{h_decay} holds.

    Let us concentrate on the expansion \eqref{h_decay} of $h$ in harmonic coordinates. From \emph{Step 1} and \emph{Step 3} we know that $h \in W^{2,q}_{loc}\bigl(M\setminus\{\p\}\bigr) \cap W^{2,r}(M)$ satisfies
    \begin{equation*}
        \Lg h = -b_nf
    \end{equation*}
    on $M$ and in light of \cref{cor: reg change harmonic to smooth} and \eqref{eq: harmonic properties 2}, $h(x) \in W^{2,q}_{loc}\bigl(\psi(V)\setminus\{0\}\bigr) \cap W^{2,r}_{loc}(\psi(V))$ satisfies
    \begin{equation}
        \label{eq: PDE h locally}
        -a_n g^{ij}(x)\partial_{i}\partial_{j} h(x) + \R_g(x) h(x) = -b_nf(x)
    \end{equation}
    in the harmonic chart $(V,\psi)$. The principal part coefficients $g^{ij}(x)$ as well as the scalar curvature $\R_g(x)$ are in $W^{2,q}_{loc}(\psi(V))$ by the same reasoning as in \emph{Step 1}. Moreover, we also know that $f  = \BO\bigl(|x|^{-\gamma}\bigr)$ with $\gamma = n + \frac{n}{q} - 2$, so the hypotheses of \cref{prop: blow-up behaviour} below are full-filled. Notice that if we prove that the expansion \eqref{h_decay} hold for a sufficiently large $r < \tfrac{nq}{n + (n-2)q}$, then it holds, in particular, for any smaller $r$.
    
    If $\tfrac{n}{2} < \tfrac{nq}{n + (n-2)q}$, then $h(x) \in C^{0,2-\frac{n}{r}}(\psi(U))$ for all $ \tfrac{n}{2} < r < \tfrac{nq}{n + (n-2)q}$ and thus it reads
    \begin{equation*}
        h(x) = h(0) + \BO\bigl(|x|^{2-\frac{n}{r}}\bigr) \quad \text{as} \quad |x| \to 0 \,.
    \end{equation*}
    In particular, $h$ is bounded in $U$ and the positivity of $G_{p}$ follows. Additionally, since
    \begin{equation*}
        \gamma \leq \frac{n}{r}  \quad \iff \quad r \leq \frac{nq}{n+(n-2)q} \,,
    \end{equation*}
    item $(ii)$ in \cref{prop: blow-up behaviour} with $\beta = 2 - \tfrac{n}{r}$ implies the first order decay
    \begin{equation*}
        h(x) = h(0) + \BO_1\bigl(|x|^{2-\frac{n}{r}}\bigr) \quad \text{as} \quad |x| \to 0
    \end{equation*}
    whenever $q > n$. 

    If $\tfrac{n}{2} \geq \tfrac{nq}{n+(n-2)q} > r$, the function $h$ need not be bounded in $U$ and we need to work with part $(i)$ of \cref{prop: blow-up behaviour}.
    We can consider the optimal Sobolev embedding $W^{2,r}(U) \hookrightarrow L^{\frac{rn}{n-2r}}(U)$ and notice that $r$ can be chosen so that  $\tfrac{2q}{2q - 1} < \tfrac{rn}{n-2r}$ as in the hypothesis of \cref{prop: blow-up behaviour}. In fact, combining this with the upper bound on $r$, it is equivalent to require $1-\tfrac{1}{2q} > \tfrac{1}{r} - \tfrac{2}{n} > 1 + \tfrac{1}{q} - \tfrac{4}{n}$ and we can check that $\tfrac{1}{2q} < \tfrac{4}{n} - \tfrac{1}{q}$ is
    ensured by $q > \tfrac{n}{2}$.
     The second condition of item $(i)$ in \cref{prop: blow-up behaviour} follows directly from
    \begin{align*}
        \frac{n}{r} - 2 > n + \frac{n}{q} - 4 = \gamma - 2 \,.
    \end{align*}
    At this point, the claim is implied by part $(i)$ of \cref{prop: blow-up behaviour}.
\end{proof} 

\subsection{Blow-up analysis}
\label{Blow-up analysis}
In the remainder of the section, we show a \emph{quantitative removable singularity} result for linear elliptic equations with rough coefficients, which is key in the last step of the proof of the previous \cref{theorem C}. The statement is rather general and it applies to a large family of linear elliptic operators: given a smooth, 
bounded domain $\Omega \subset \nR^n$ around the origin, we consider a linear operator
\begin{equation}
    \label{eq: Linear Elliptic Operator}
    \LL \doteq a^{ij}(x) \partial_{i} \partial_{j}  + b^i (x) \partial_{i} + c(x) \,,
\end{equation}
where the coefficients $a^{ij} \in W^{2,q}(\Omega)$ and $b^i, c \in L^{\infty}(\Omega)$ with $q > \tfrac{n}{2}$ satisfy
\begin{align}
    \label{coeff conditions 1}
    \inf_{\Omega} a^{ij}\xi_i \xi_j \geq \lambda |\xi|^2 & \qquad \text{for all} \; \xi \in \nR^n \,, 
    \\
    \label{coeff conditions 2}
    \|a^{ij}\|_{C^{0,2 - \frac{n}{q}}(\Omega)} \,, 
    \quad
    \sum_{k = 1}^n \|\partial_{k} a^{ij}\|_{L^{2q}(\Omega)} & \,, \quad
    \sup_{\Omega}|b^i|\,, \quad \sup_{\Omega} |c| \; \leq \; \Lambda 
\end{align}
for some positive constants $\lambda$ and $\Lambda$. Notice that the above norms are well defined due to the Sobolev embeddings $W^{2, q}(\Omega) \hookrightarrow C^{0, 2 - \frac{n}{q}}(\Omega)$ and $W^{1, q}(\Omega) \hookrightarrow L^{2q}(\Omega)$; the latter is clear if $q \geq n$, while it is ensured for $q < n$  by the inequality $2q < \tfrac{nq}{n-q}$ for $q > \tfrac{n}{2}$.

\begin{theorem}
    \label{prop: blow-up behaviour}
    Let $\Omega \subset \nR^n$ be a smooth, bounded domain around the origin and let $\LL$ be an operator of the form of \eqref{eq: Linear Elliptic Operator} with coefficients satisfying \eqref{coeff conditions 1} and \eqref{coeff conditions 2}.  Let $u \in W^{2,q}_{loc}\bigl(\Omega \setminus \{0\}\bigr)$ be a solution of $\LL u = f$ in $\Omega$ with $f \in L^q(\Omega)$ and suppose that $f = \BO(|x|^{-\gamma})$ as $|x| \to 0$ for some $\gamma \geq 0$. 
    \begin{enumerate}
        \item[$(i)$] If $u \in  L^p(\Omega)$ for
        $p > \tfrac{2q}{2q-1}$
        and $\gamma \leq 2 + \tfrac{n}{p}$, then
        \begin{equation*}
            u = \BO\bigl(|x|^{-\frac{n}{p}}\bigr) \quad \text{as} \quad |x| \to 0 \,.
        \end{equation*}
        Moreover, if $q > n$, then
        \begin{equation*}
            u = \BO_1\bigl(|x|^{-\frac{n}{p}}\bigr) \quad \text{as} \quad |x| \to 0 \,.
        \end{equation*}
        \item[$(ii)$] If $u \in C^{0,\beta}(\Omega)$ for $0 < \beta < 1$, $\gamma \leq 2 - \beta$ and $q > n$, then
        \begin{equation*}
            u = u(0) + \BO_1\bigl(|x|^\beta\bigr) \quad \text{as} \quad |x| \to 0 \,.
        \end{equation*}
    \end{enumerate}
\end{theorem}
In order to prove \cref{prop: blow-up behaviour}, we first need the following two technical lemmas. In the first one, the dependences of the constant appearing in the elliptic estimate is carefully tracked down.
\begin{lemma}
    \label{lemma: Avalos 3.2 adaptation}
    Let $\Omega \subset \nR^n$ be a smooth, bounded domain around the origin and $\LL$ be an operator of the form of \eqref{eq: Linear Elliptic Operator} with coefficients satisfying \eqref{coeff conditions 1} and \eqref{coeff conditions 2}. For every smooth bounded domains $\Omega'' \subset \subset \Omega' \subset \subset \Omega$ and $1 < p < \infty$ satisfying $\tfrac{1}{q} - \tfrac{1}{n} \leq \tfrac{1}{p} < \tfrac{1}{q'} + \tfrac{1}{2q}$,\footnote{
    The upper bound $\tfrac{1}{p} < \tfrac{1}{q'} + \tfrac{1}{2q}$ is strictly below the usual 
    bound $\tfrac{1}{q'} + \tfrac{1}{n}$. 
    It is equivalent to $\tfrac{1}{2q} < \tfrac{1}{p'}$ and allows ensuring that the constant $C$ will depend, among others, on $\| \partial_k a^{ij}\|_{L^{2q}}$. This bound is not restrictive for our application to the Green function expansion in \cref{theorem C}.
    Dealing with 
    $\tfrac{1}{p'} \leq \tfrac{1}{2q}$ is possible, but 
    requires control over a different norm and hence to distinguish the two cases. } there exists a constant $C = C(n, p, q, \lambda, \Lambda, \Omega', \Omega'') > 0$ such that
    \[\|u\|_{W^{1, p}(\Omega'')} \leq C \left(\| \LL u \|_{W^{-1, p}(\Omega')} + \| u \|_{L^p(\Omega')}\right)
    \]
    holds for all $u \in W^{1, p}(\Omega)$.
\end{lemma}

\begin{proof} 
  The proof consists of two steps:

    \textit{Step 1.}
    We first assume that $u$ is compactly supported in some small ball $B_r(x_0) \subset \subset \Omega'$ centered at $x_0 \in \Omega''$. Since the estimate holds for constant coefficients operators (see e.g. \cite[Lemma 2.19]{MaxHolTso2}), we freeze the coefficients at $x_0$ to obtain
    \begin{align*}
        \|u\|_{W^{1, p}(\Omega')} & \leq C_{0} \Big(\|a^{ij}(x_0)\partial_i\partial_j u\|_{W^{-1, p}(\Omega')} + \|u\|_{L^p(\Omega')}\Big) 
        \\
        & \leq C_{0} \, \Big(\|\LL u\|_{W^{-1, p}(\Omega')} +  \|u\|_{L^p(\Omega')}  
        \\
        & + \|(a^{ij} - a^{ij}(x_0)) \,\partial_i \partial_j u\|_{W^{-1, p}(\Omega')} +\|b^k \,\partial_k u\|_{W^{-1, p}(\Omega'))} + \|c\, u\|_{W^{-1, p}(\Omega'))} \Big) \,.
    \end{align*}
    The constant $C_0$ depends only on the principal symbol of the frozen coefficient operator, which is controlled by the ellipticity constant $\lambda$ and $\Lambda$; in particular,  it is independent of $r$. Being $u$ supported in a small ball will allow us to estimate the error terms as follows: consider a test function $\phi \in W_0^{1, p'}(\Omega')$ and observe that 
    \begin{align*}
        (a^{ij} - a^{ij}(x_0)) \, \phi \in W^{2, q} (\Omega) \otimes W^{1, p'}_0 (\Omega') \hookrightarrow  W^{1, p'}_0 (\Omega')
    \end{align*}
    as $\tfrac{1}{p} \leq \tfrac{1}{q'} + \tfrac{1}{n}$. Hence, we can estimate
    \begin{align}
        \label{eq: term1}
        \big|\langle (a^{ij} - a^{ij}(x_0))   \,\partial_i \partial_j u, \,\phi \rangle \big| & = \left|\langle \partial_i \partial_j u, \, (a^{ij} - a^{ij}(x_0)) \, \phi \rangle \right| 
        \nonumber
        \\
        & = \left|\bigl\langle \partial_j u, \, \partial_i \bigl(( a^{ij} - a^{ij}(x_0)) \, \phi\bigr) \bigr\rangle \right| 
        \nonumber
        \\
        &\leq \int_{\Omega'} |\partial_j u \, \partial_i a^{ij} \, \phi| \, +  \int_{\Omega'} | a^{ij} - a^{ij}(x_0)| \, |\partial_j u \, \partial_i\phi| 
        \nonumber
        \\
        & = \int_{B_r(x_0)} |\partial_j u \, \partial_i a^{ij} \, \phi| \, +  \int_{B_r(x_0)} | a^{ij} - a^{ij}(x_0)| \, |\partial_j u \, \partial_i\phi| \,.
    \end{align}
    To get a bound on the first term, fix some $s$ such that
    \begin{equation}
        \label{eq: s range}
        \max\left\{\frac{1}{p'} - \frac{1}{n}, \; 0\right\} < \frac{1}{s} < \frac{1}{p'} - \frac{1}{2q}
    \end{equation}
    and observe that $W^{1,p'}_{loc}(\Omega') \hookrightarrow L^s(\Omega)$ always holds. Notice that such $s$ can always be chosen due to $q > \tfrac{n}{2}$ and $\tfrac{1}{p} < \tfrac{1}{q'} + \tfrac{1}{2q} = 1 - \tfrac{1}{2q}$. Then, defining $\tfrac{1}{l} \doteq \tfrac{1}{p'} - \tfrac{1}{2q} - \tfrac{1}{s}$, we notice that $1 < l < \infty$ and we may estimate
    \begin{align*}
        \nonumber
        \int_{B_{r}(x_0)} |\partial_j u \, \partial_i a^{ij} \,\phi| &\leq \|\partial_j u\|_{L^p(B_{r}(x_0))} \, \|\partial_i a^{ij} \, \phi \|_{L^{p'}(B_{r}(x_0))} 
        \\ 
        \nonumber
        & \leq \|\partial_j u\|_{L^p(\Omega')} \,|B_r|^{\frac{1}{l}}\, \|\partial_i a^{ij}\|_{L^{2q}(B_{r}(x_0))} \, \|\phi\|_{L^{s}(\Omega')} 
        \\
        &\leq C(n, p, q, \Omega') \, \Lambda \, r^{\frac{n}{l}}\, \|u\|_{W^{1,p}(\Omega')} \, \|\phi\|_{W^{1,p'}(\Omega')}  \,.
     \end{align*}
     On the other hand, 
     the second term of \eqref{eq: term1}  can be estimated by 
     using the Hölder continuity of $a^{ij}$ as 
     \begin{align*}
         \int_{B_{r}(x_0)} | a^{ij} - a^{ij}(x_0)| \,|\partial_j u \, \partial_i\phi| &\leq \sup_{x \in B_r\setminus\{x_0\}}\frac{\bigl|a^{ij}(x) - a^{ij}(x_0)\bigr|}{|x-x_0|^{2 - \frac{n}{q}}} \, r^{2-\frac{n}{q}} \int_{B_{r}(x_0)} |\partial_j u \, \partial_i\phi|
         \\
         & \leq \Lambda  \, r^{2-\frac{n}{q}} \,\|\partial_j u\|_{L^p(B_{r}(x_0))} \,\|\partial_i \phi\|_{L^{p'}(B_{r}(x_0))} 
         \\
         & \leq \Lambda  \, r^{2-\frac{n}{q}} \,\|u\|_{W^{1, p}(\Omega')} \,\|\phi\|_{W^{1, p'}(\Omega')} \,.
     \end{align*}
     Similarly, one can estimate the lower-order error terms: with the same $s$ as before, satisfying \eqref{eq: s range}, we get
     \begin{align*}
         \left|\langle b^i \partial_i u, \, \phi \rangle \right| & = \int_{B_{r}(x_0)} |b^i \, \partial_i u \, \phi| \leq \sup_{\Omega} |b^i| \int_{B_{r}(x_0)} |\partial_i u \, \phi| 
         \\
         &\leq \Lambda \,|B_r|^{\frac{1}{p'}-\frac{1}{s}} \,\|\partial u\|_{L^p(B_r(x_0))} \|\phi\|_{L^s(B_r(x_0))} 
         \\
         & \leq C(n, p, q, \Lambda,\Omega') \,  r^{\frac{n}{p'} - \frac{n}{s}} \, \|u\|_{W^{1,p}(\Omega')} \, \| \phi \|_{W^{1, p'}(\Omega')} 
     \end{align*}
     as well as
     \begin{align*}
         \left|\langle c \, u, \, \phi \rangle \right| & = \int_{B_{r}(x_0)} |c \, u \, \phi| \leq \sup_{\Omega}|c| \int_{B_r} |u \, \phi|
         \\
         & \leq \Lambda \, |B_r|^{\frac{1}{p'}-\frac{1}{s}} \, \|u\|_{L^{p}(B_{r}(x_0))} \, \| \phi \|_{L^{s}(B_{r}(x_0))} 
         \\
         & \leq C(n, p, q, \Lambda, \Omega') \, r^{\frac{n}{p'} - \frac{n}{s}} \, \|u\|_{W^{1, p}(\Omega')} \, \| \phi \|_{W^{1, p'}(\Omega')} \,.
     \end{align*}
     Combining the estimates on all the error terms and assuming without loss of generality that $r < 1$, we obtain that
     \begin{equation*}
         \|u\|_{W^{1,p}(\Omega')} \leq C_0 \Big(\|\LL u\|_{W^{-1,p}(\Omega')} + \|u\|_{L^p(\Omega')}\Big) + C(n, p, q, \Lambda, \Omega') \, r^\sigma \, \|u\|_{W^{1,p}(\Omega')} \,,
     \end{equation*}
     where
     \begin{equation*}
         \sigma \doteq \min \left\{2-\frac{n}{q}, \frac{n}{l}, \frac{n}{p'} - \frac{n}{s}\right\} > 0 \,.
     \end{equation*}
     Thereby, taking $r > 0$ small enough depending on $n$, $p$, $q$, $\Lambda$ and $\Omega'$, the error term can be absorbed in the left-hand-side, yielding the estimate
    \begin{equation}\label{eq: localized_estimate}
        \|u\|_{W^{1,p}(\Omega')} \leq C\Big(\|\LL u\|_{W^{-1,p}(\Omega')} + \|u\|_{L^p(\Omega')}\Big) \,.
    \end{equation}
    for a constant $C=C(n,p,q,\lambda,\Lambda,\Omega')>0$.

    \medskip
    \textit{Step 2.}  To get the global estimate, given $x\in \Omega''$, we first notice that \eqref{eq: localized_estimate} holds for all $u\in W_0^{1,p}(B_{r_x}(x))$ and some $r_x>0$, which we may take small enough so that $B_{r_x}(x)\subset\subset \Omega'$. Due to the compactness of $\overline{\Omega''}$, there exists a finite amount of such points $\{x_k\}_{k=1}^{K}$ such that $\Omega'' \subset \cup_{k=1}^K B^k$, where $B^k \doteq B_{r_k}(x_k)\subset\subset \Omega'$ and now \eqref{eq: localized_estimate} holds on each $B^k$ for all $u\in W^{1,p}_0(B^k)$ with a constant $C_k=C_k(n,p,q,\lambda,\Lambda,\Omega')$. Let us consider a smooth partition of unity $\{\eta_k\}_{k = 1}^K$ subordinate to $\{B^k\}_{k = 1}^K$ and estimate
    \begin{align*}
            \|u\|_{W^{1,p}(\Omega'')} &\leq \sum_{k=1}^K \|\eta_k u\|_{W^{1,p}(\Omega')}
            \\
            &\leq C \, \sum_{k=1}^K \Big(\|\eta_k \, \LL u\|_{W^{-1,p}(\Omega')} + \|u\|_{L^p(\Omega')}\Big)
            \\
            &\hspace{-1.5cm} + C \, \sum_{k=1}^K \Big(\bigl\|\bigl(a^{ij}\partial_i\partial_j \eta_k + b^i\partial_i\eta_k\bigr)u\bigr\|_{W^{-1,p}(\Omega')} + \bigl\|a^{ij} \bigl(\partial_i\eta_k\partial_j u + \partial_j\eta_k\partial_i u\bigr)\bigr\|_{W^{-1,p}(\Omega')} \Big),
    \end{align*}
    for a constant $C=C(n,p,q,\lambda,\Lambda,\Omega',\Omega'')$ which can be taken to be $C=\max\{C_k, \: k=1,\cdots,N\}$. It is easy to see that for any test function $\phi \in W_0^{1,p'}(\Omega')$, there holds
    \begin{equation*}
        \left|\langle  a^{ij} \partial_i \partial_j \eta_k \, u + b^i \partial_i \eta_k \, u ,\, \phi \rangle \right| \leq C(\eta_k, \Omega') \, \Lambda \, \| u \|_{L^p(\Omega')} \, \| \phi \|_{W^{1, p'}(\Omega')}
    \end{equation*}
    while the last term can be estimated as
    \begin{align*}
        \left|\langle  a^{ij} \partial_i \eta_k \, \partial_ju  ,\, \phi \rangle \right| & =  \left|\langle u,\, \partial_j(a^{ij} \partial_i \eta_k \,\phi) \rangle \right| 
        \\
        &\leq C(\eta_k, \Omega') \, \int_{\Omega'} |u \,( \partial_ja^{ij} \,\phi + a^{ij}(\partial_j\phi + \phi))| 
        \\
        &\leq C(\eta_k, \Omega') \, \left( \|u\|_{L^p(\Omega')}\,\|\partial_ja^{ij} \,\phi\|_{L^{p'}(\Omega')} + \Lambda  \, \| u\|_{L^p(\Omega')} \,\|\phi\|_{
        W^{1,p'}(\Omega')}\right) \\
        &\leq C'(\eta_k, \Omega') \, \|u\|_{L^p(\Omega')}\, \left( \|\partial_ja^{ij}\|_{L^{2q}(\Omega')} \, \|\phi\|_{L^{s}(\Omega')} + \Lambda \, \|\phi\|_{W^{1, p'}(\Omega')}\right)  
        \\
        & \leq C''(\eta_k, \Omega') \, \Lambda \, \|u\|_{L^p(\Omega')}\,\|\phi\|_{W^{1,p'}(\Omega')}\,,
    \end{align*}
    where, following a similar strategy to the one above, we have chosen $\tfrac{1}{s} \doteq \tfrac{1}{p'} - \tfrac{1}{2q} > \tfrac{1}{p'} - \tfrac{1}{n}$ to ensure the last Sobolev inequality and is well-defined due to $\tfrac{1}{p} < \tfrac{1}{q'} + \tfrac{1}{2q}$ equivalent to $\tfrac{1}{2q} < \tfrac{1}{p'}$. The assertion then follows by combining the three estimates above.
\end{proof}
The above lemma allows us to prove 
a refined version of \cite[Theorem 9.11]{GT}:
\begin{lemma}
    \label{lemma: GT 9.11 adaptation}
    Let $\Omega\subset \mathbb{R}^n$ be a smooth, bounded domain around the origin and let $\LL$ be an operator of the form of \eqref{eq: Linear Elliptic Operator} with coefficients satisfying \eqref{coeff conditions 1}-\eqref{coeff conditions 2}. Let $u \in W^{2,q}(\Omega)$ be a solution of $\LL u = f$ in $\Omega$ with $f \in L^q(\Omega)$. Then, for every smooth, bounded $\Omega'' \subset \subset \Omega$ and every 
    $p > \tfrac{2q}{2q-1}$, there exists a positive constant $C = C\bigl(n, q, p, \lambda, \Lambda, \Omega'', \Omega\bigr)$ such that
    \begin{equation*}
        \|u\|_{W^{2, q}(\Omega'')} \leq C \Big(\|u\|_{L^{p}(\Omega)} + \| f\|_{L^q(\Omega)}\Big) \,.
    \end{equation*}
\end{lemma}
\begin{proof}
    Fix $d \coloneqq \text{dist}(\Omega'', \partial\Omega)$ and for each $k = 0, 1, 2, \ldots$ define
    \begin{equation*}
        \Omega_k \coloneqq \Big\{x \in \Omega \; : \; \operatorname{dist}(x, \partial\Omega) > \tfrac{d}{2+k+k_0}\Big\} \,,
    \end{equation*}
    so that $\Omega'' \subset \subset \Omega_0 \subset \subset \Omega_1 \subset \subset \dots \subset \subset \Omega$ and $k_0>0$ is chosen large enough so that $\partial\Omega_k$ are smooth for all $k$. From \cite[Theorem 9.11]{GT} there exists a positive constant $C_0 = C_0\bigl(n, q, \lambda, \Lambda, \Omega'', \Omega\bigr)$ such that
    \begin{equation}
        \label{eq: GT 9.11 estimate}
        \|u\|_{W^{2, q}(\Omega'')} \leq C_0 \Big(\|u\|_{L^{q}(\Omega_0)} + \| f\|_{L^q(\Omega)}\Big) \,.
    \end{equation}
    If $p \geq q$ the statement follows by H\"older inequality, so let us consider $p < q$. Notice that the lower bound on $p$ corresponds to $\tfrac{1}{p} < \tfrac{1}{q'} + \tfrac{1}{2q}$.
    Define then the number
    \begin{equation*}
        \delta \doteq \min\left\{\frac{1}{n}, \; \frac{1}{q'} + \frac{1}{2q} - \frac{1}{p}\right\} > 0\,,
    \end{equation*}
    and set  
    \begin{equation*}
        p_0 = q \,, \qquad \frac{1}{p_{k}} \doteq \frac{1}{p_{k-1}} + \delta = \frac{1}{q} + k\delta
    \end{equation*}
    for integers $k \geq 1$. Let $N \in \nN$ be the smallest integer such that $p_N \leq p$. Observe that all $p_0 > p_1 > \ldots > p_{N}$ satisfy $W^{1,p_{k+1}}(\Omega_k) \hookrightarrow L^{p_k}(\Omega_k)$ and $\tfrac{1}{q} \leq \tfrac{1}{p_k} < \tfrac{1}{q'} + \tfrac{1}{2q} $; the strict inequality is clear for $k \leq N-1$ due to $\tfrac{1}{p_{N-1}} < \tfrac{1}{p} < \tfrac{1}{q'}+\tfrac{1}{2q}$ by hypothesis, but it also holds true for $p_N$, due to the step size $\delta \leq \tfrac{1}{q'} + \tfrac{1}{2q} - \tfrac{1}{p}$. Therefore, we can apply \cref{lemma: Avalos 3.2 adaptation} to get
    \begin{align*}
        \|u\|_{L^{p_k}(\Omega_k)} &\leq C(n, q, p, \Omega_k) \, \|u\|_{W^{1, p_{k+1}}(\Omega_k)}
        \\
        &\leq C(n, q, p, \Omega_k, \lambda, \Lambda, \Omega_{k+1}) \Big(\|u\|_{L^{p_{k+1}}(\Omega_{k+1})} + \|f\|_{W^{-1, p_{k+1}}(\Omega_{k+1})}\Big) 
        \\
        &\leq C(n, q, p,\Omega_k, \lambda, \Lambda,\Omega_{k+1}) \Big(\|u\|_{L^{p_{k+1}}(\Omega_{k+1})} + \|f\|_{L^q(\Omega_{k+1}})\Big)
        \\
         &\leq C(n, q, p, \Omega_k, \lambda, \Lambda, \Omega_{k+1}) \Big(\|u\|_{L^{p_{k+1}}(\Omega_{k+1})} + \|f\|_{L^q(\Omega})\Big)
    \end{align*}
     where the third inequality is given by the fact that 
     \begin{equation*}
         L^q(\Omega_{k+1}) \hookrightarrow W^{-1, p_{k+1}}(\Omega_{k+1}) \quad \iff \,\,\, W^{1, p_{k+1}'}_0(\Omega_{k+1}) \hookrightarrow L^{q'}(\Omega_{k+1}) \,,
     \end{equation*}
     which is guaranteed by $\tfrac{1}{p_{k+1}} \geq \tfrac{1}{q} - \tfrac{1}{n}$. By concatenating the estimates and noticing that all the domains $\Omega_k$ are determined by $\Omega$, $\Omega''$, $n$, $q$ and $p$, we establish
     \begin{align*}
         \|u\|_{L^q(\Omega_0)} &\leq C(n, q, p, \lambda, \Lambda, \Omega_0, \ldots \Omega_N) \Big(\|u\|_{L^{p_{N}}(\Omega)} + \| f\|_{L^q(\Omega)}\Big) 
         \\
         &\leq C(n, q, p, \lambda, \Lambda,\Omega'', \Omega) \Big(\|u\|_{L^{p}(\Omega)} + \| f\|_{L^q(\Omega)}\Big) \,.
     \end{align*}
     The assertion follows by combining it with \eqref{eq: GT 9.11 estimate}.
\end{proof}
Finally, we are ready to prove \cref{prop: blow-up behaviour}. The proof is a scaling argument inspired by \cite[Proposition 9.1]{TaylorToolsForPDEs}.
\begin{proof}[Proof of \cref{prop: blow-up behaviour}]
    Since by hypothesis $f = O\bigl(|x|^{-\gamma}\bigr)$ as $|x| \to 0$, there exists a fixed $\rho_0\in (0,1)$ and a constant $C_{\rho_0}>0$ such that if $\rho < \rho_0$
    \begin{align}
        \label{RadChoice.1}
        \sup_{x\in B_{\rho}(0)}|f(x)||x|^{\gamma}\leq C_{\rho_0}.
    \end{align}
    Having such $\rho_0$ fixed, consider $\rho>0$ such that $0 < 4\rho\leq \rho_0 < 1$ and $B_{4\rho}(0) \subset\subset \Omega$. Let us then consider some fixed $x_0 \in \partial B_{2\rho}(0)$.  Define $u_{\rho}(x) \doteq u(\rho(x-x_0) + x_0)$ and observe that it satisfies
    \begin{equation*}
        \LL_{\rho} u_\rho = a_{\rho}^{ij}\partial_i\partial_j u_{\rho} + b^i_{\rho}\partial_iu + c_\rho u_{\rho} = \rho^2 f_\rho \qquad \text{in} \quad B_1(x_0) \,,
    \end{equation*}
    where
    \begin{align*}
        &a_{\rho}^{ij}(x) = a^{ij}(\rho(x-x_0) + x_0) \,, &&b^i_{\rho}(x) = \rho \, b^i(\rho(x-x_0) + x_0) \,, 
        \\
        &c_\rho(x) = \rho^2 c(\rho(x-x_0) + x_0) \,, \quad &&f_{\rho}(x) = f(\rho(x-x_0) + x_0) \,.
    \end{align*}
    Since $u \in W^{2,q}\bigl(B_{\rho}(x_0)\bigr)$, then $u_{\rho} \in W^{2,q}\bigl(B_1(x_0)\bigr)$. Similarly, if $u \in L^p\bigl(B_{\rho}(x_0)\bigr)$ for given $p > \tfrac{2q}{2q-1}$ in the case $(i)$, then $u_{\rho} \in L^p \bigl(B_1(x_0)\bigr)$. The same applies in the case $(ii)$ for any such $p$. Thus, we may use \cref{lemma: GT 9.11 adaptation} with $\Omega'' = B_{3/4}(x_0)$ and $\Omega = B_1(x_0)$ to obtain
    \begin{equation}
        \label{eq: elliptic estimate in Prop}
        \|u_{\rho}\|_{W^{2,q}(B_{3/4}(x_0))} \leq C_0 \Big(\|u_{\rho}\|_{L^{p}(B_1(x_0))} + \rho^2\|f_{\rho}\|_{L^{q}(B_1(x_0))} \Big)
    \end{equation}
    with a constant $C_0$ depending on $n$, $q$, $p$ and the constants $\lambda$ and $\Lambda$ from \cref{lemma: GT 9.11 adaptation} corresponding to the operator $\LL_{\rho}$. Observe that by
    \begin{equation*}
        \inf_{B_1(x_0)} a_{\rho}^{ij} = \inf_{B_{\rho}(x_0)} a^{ij} \geq \inf_\Omega a^{ij} \,,
    \end{equation*}
    the ellipticity constant $\lambda$ can be chosen independent of $\rho$. Similarly, using that $\rho < 1$ and $q > \tfrac{n}{2}$, one can check that
    {
    \allowdisplaybreaks
    \begin{align*}
        \|a^{ij}_\rho\|_{C^{0,2 - \frac{n}{q}}(B_1(x_0))} &= \sup_{B_{\rho}(x_0)} |a^{ij}| + \rho^{2 - \frac{n}{q}}\sup_{\underset{x\neq y}{x,y \in B_{\rho}(x_0)}}\frac{\bigl|a^{ij}(x) - a^{ij}(y)\bigr|}{|x-y|^{2 - \frac{n}{q}}} \leq \|a^{ij}\|_{C^{0,2-\frac{n}{q}}(\Omega)} \,,
        \\
        \|\partial_k a_{\rho}^{ij}\|_{L^{2q}(B_1(x_0))}
        &= \rho^{1 - \frac{n}{2q}} \, \| \partial_k a^{ij}\|_{L^{2q}(B_\rho(x_0))} \leq \| \partial_k a^{ij}\|_{L^{2q}(\Omega)} \,,
        \\
        \sup_{B_1(x_0)}|b^i_{\rho}| &= \rho \sup_{B_{\rho}(x_0)} |b^i| \leq \sup_{\Omega} |b^i| \,,
        \\
        \sup_{B_1(x_0)}|c_{\rho}| &= \rho^2 \sup_{B_{\rho}(x_0)} |c| \leq \sup_{\Omega} |c| \,.
    \end{align*}
    }
    It follows that also $\Lambda$, and consequently $C_0$, can be chosen independently of $\rho$. The second term in \eqref{eq: elliptic estimate in Prop} can be estimated as follows. By our choice of $\rho\leq \rho_0$, from \eqref{RadChoice.1} we know that
    \begin{align*}
        |f(x)|\leq C_{\rho_0}|x|^{-\gamma}
    \end{align*}
    for all $x\in B_{\rho}(x_0)$. Moreover, since $B_{\rho}(x_0)\subset B_{3\rho}(0)\backslash B_{\rho}(0)$, then $|x|\geq \rho$ for all $x\in B_{\rho}(x_0)$ and thus 
    \begin{align*}
        \sup_{x\in B_{\rho}(x_0)}|f(x)|\leq C_{\rho_0}\rho^{-\gamma} \,.
    \end{align*}
    Therefore, we may estimate
    \begin{equation}
        \label{eq: first f term estimate}
        \rho^2 \|f_\rho\|_{L^q(B_1(x_0))} = \rho^{2-\frac{n}{q}}\|f\|_{L^q(B_\rho(x_0))} \leq C_{\rho_0} \rho^{2-\frac{n}{q}-\gamma} |B_{\rho}(x_0)|^{\frac{1}{q}} \leq C_1 \rho^{2-\gamma}
    \end{equation} 
    with $C_1$ depending only on $n$, $q$ and $\rho_0$.
    
    In the case (\textit{i}), we estimate the first term of \eqref{eq: elliptic estimate in Prop} by
    \begin{equation*}
        \|u_{\rho}\|_{L^p(B_{1}(x_0))} = \rho^{-\frac{n}{p}}\|u\|_{L^{p}(B_{\rho}(x_0))} \leq \rho^{-\frac{n}{p}}\|u\|_{L^{p}(\Omega)}
    \end{equation*}
    and observe that if $\gamma \leq 2 + \tfrac{n}{p}$ the first term dominates over the second for small enough $\rho$. Thus, we establish
    \begin{equation*}
        \|u_{\rho}\|_{W^{2,q}(B_{3/4}(x_0))} \leq C_2 \, \rho^{-\frac{n}{p}}
    \end{equation*}
    for some $C_2$ depending on $n$, $q$, $p$, $\lambda$, $\Lambda$, $\rho_0$ and $\|u\|_{L^p(\Omega)}$. Using the Sobolev embedding $W^{2,q}(B_{3\rho/4}(x_0)) \hookrightarrow C^{0,2-\frac{n}{q}}(B_{3\rho/4}(x_0))$, we obtain
    \begin{align*}
        \sup_{B_{3\rho/4}(x_0)} |u| = \sup_{B_{3/4}(x_0)} |u_{\rho}| \leq \|u_{\rho}\|_{C^{0,2-\frac{n}{q}}(B_{3/4}(x_0))} \leq C(n, q) \|u_{\rho}\|_{W^{2,q}(B_{3/4}(x_0))} \leq C_3 \, \rho^{-\frac{n}{p}} \,.
    \end{align*}
    Clearly, this implies that $|u(x_0)|\leq  C_3 \rho^{-\frac{n}{p}}=2^{\frac{n}{p}}C_3|x_0|^{-\frac{n}{p}}$ and since $x_0 \in \partial B_{2\rho}(0)$ was arbitrary, it follows that $u(x) = \BO\bigl(|x|^{-\frac{n}{p}}\bigr)$. In addition if $q > n$, then we can use the $W^{2,q}(B_{3/4}(x_0)) \hookrightarrow C^{1,1-\frac{n}{q}}(B_{3/4}(x_0))$ embedding to additionally obtain
    \begin{align*}
        \rho \, \sup_{B_{3\rho/4}(x_0)} |\partial u| &=\sup_{B_{3/4}(x_0)} |\partial u_{\rho}|
        \\
        &\leq \|u_{\rho}\|_{C^{1,1-\frac{n}{q}}(B_{3/4}(x_0))} \leq C(n, q) \|u_{\rho}\|_{W^{2,q}(B_{3/4}(x_0))} \leq C_3 \, \rho^{-\frac{n}{p}} \,,
    \end{align*}
    which shows that $\partial u = \BO\bigl(|x|^{-\frac{n}{p}-1}\bigr)$ as desired.

    In the case (\textit{ii}), we can use the Hölder continuity of $u$ through the origin to expand
    \begin{equation*}
        u = u(0) + \BO\bigl(|x|^\beta\bigr) \quad \text{as} \quad |x| \to 0 \,.
    \end{equation*}
    Notice that without loss of generality, we can assume $u(0) = 0$. In fact, considering the function $v \coloneqq u - u(0)$, it solves $\LL v = f - c\, u(0)$ and the right hand side is then still in $L^q(\Omega)$ and $\BO(|x|^{-\gamma})$. We can estimate the first term of \eqref{eq: elliptic estimate in Prop} as
    \begin{align*}
        \|u_\rho\|_{L^p(B_1(x_0))} &= \rho^{-\frac{n}{p}} \left(\int_{B_\rho(x_0)} |u(x)|^p dx\right)^{1/p}
        \\
        &\leq \sup_{\underset{x \neq y}{x,y \in \Omega}}\frac{\bigl|u(x) - u(y)\bigr|}{|x-y|^{\beta}} \, 2^\beta \rho^{-\frac{n}{p}+\beta}|B_\rho|^{\frac{1}{p}} 
        \,\leq C_4 \, \rho^\beta 
        \,.
    \end{align*}
    Combined with \eqref{eq: first f term estimate} and the hypothesis $\beta \leq 2 - \gamma$, we obtain
    \begin{equation*}
        \|u_\rho\|_{W^{2,q}(B_{3/4}(x_0))} \leq C_5 \, \rho^\beta \,, 
    \end{equation*}
    where $C_5$ depends only on $n$, $q$, $p$, $\lambda$, $\Lambda$, $\rho_0$. As before, if $q > n$ we use the $W^{2,q}(B_{3/4}(x_0)) \hookrightarrow C^{1,1-\frac{n}{q}}(B_{3/4}(x_0))$ embedding to obtain
    \begin{align*}
        \rho \, \sup_{B_{3\rho/4}(x_0)} |\partial u| &=\sup_{B_{3/4}(x_0)} |\partial u_{\rho}| \leq \|u_{\rho}\|_{C^{1,1-\frac{n}{q}}(B_{3/4}(x_0))} 
        \\
        &\leq C(n, q) \|u_{\rho}\|_{W^{2,q}(B_{3/4}(x_0))} \leq C_4 \, \rho^{\beta},
    \end{align*}
    which grants that $u(x)-u(0)=O_1(|x|^{\beta})$.
\end{proof}

In view of the \emph{Step 3: Positivity and expansion} part of the proof of \cref{theorem C},
we recall that whenever $\tfrac{n}{2} < \tfrac{nq}{n+(n-2)q}$ the fundamental solution $|x|^{2-n}$ is the only singular part of the conformal Green's function and the error $h(x)$ has a H\"older expansion. This is only possible whenever $n = 3$ and $q > 3$; in particular, this also gives control on the first derivative of the error. In conclusion, we have established the following:
\begin{corollary}
    \label{cor: unnormalized Green 3dim}
    Let $M$ be a smooth, closed 3-manifold and consider a $W^{2,q}$-Riemannian metric $g$ on $M$ with $q > 3$ and $\lambda(M, g) > 0$. Then, for each point $p \in M$ there exists a positive function $\G \in W^{2,q}_{loc}\bigl(M \setminus \{\p\}\bigr)$ satisfying $\Lg\G = 0$ on $M \setminus \{p\}$ such that
    \begin{equation*}
        \G(x) = \frac{1}{|x|} + \h(x) \,, \qquad \h(x) = A + \BO_1\bigl(|x|^{2-\frac{3}{r}}\bigr) \quad \text{as} \quad |x| \to 0
    \end{equation*}
    in a normal coordinate chart $(U,x^i)$ around $p$ with $\h \in W^{2,r}(U)$ for all $r < \tfrac{3q}{3+q}$.
\end{corollary}
Naturally, the function in \cref{cor: unnormalized Green 3dim} is given by $\G = b_nG_p$. Such normalization will make computations simpler later in \cref{Decompactification via Conformal Green's function}.

\section{The Yamabe Problem}
\label{The Yamabe Problem}
In this section we prove \cref{theorem A} following the classic approach: first, we extend the Aubin--Trudinger--Yamabe Theorem to our setting of rough metrics and then we appeal to Schoen's idea of using the conformal Green function to 
get an AE manifold and construct an appropriate test function with the help the Positive Mass Theorem of general relativity.

\subsection{The Aubin--Trudinger--Yamabe Theorem}\label{The Aubin--Trudinger--Yamabe Theorem}
For smooth Riemannian metrics, the Aubin--Trudinger--Yamabe Theorem asserts that the Yamabe problem can be solved provided that the Yamabe invariant is below that of the round sphere. In the following, we show the analogue result for $W^{k,q}$-Riemannian metrics with $k \geq 2$ and $q > \tfrac{n}{2}$ in \cref{thm: Aubin-Trudinger-Yamabe} below, which in particular establishes \cref{theorem B}. As explained in \cref{Introduction}, we remark that the same result for $W^{1,p}$-Riemannian manifolds with $q > n$ has been proven by H. Zhang \cite{ZhangW1pAubin}.

The strategy we follow is the same as in the smooth case: we perturb the variational problem in favour of better compactness properties, find a minimizer via direct methods and finally extract a limit solution of the original problem. The main difficulty one encounters in the low-regularity setting is that the coefficients of the variational PDE are rough, and consequently standard elliptic regularity does not apply. On the other hand, the following two well-known results from the analysis of the classic Yamabe problem for smooth metrics do extend directly to the case of metrics of $W^{2,q}$-regularity.
\begin{lemma}
    \label{lemma: continuity yamabe invariant}
    Let $M$ be a smooth, closed manifold of dimension $n \geq 3$ and consider a $W^{2,q}$-Riemannian metric $g$ on $M$ with $q > \tfrac{n}{2}$ and $\vol_g(M) = 1$. Then, the map $s \mapsto |\lambda_s(M,g)|$ non-increasing for $s \in [2,2^*]$. Moreover, if $\lambda(M,g) \geq 0$, it is continuous from the left. 
\end{lemma}
\begin{proof}
    The proof is identical to the one known for smooth metrics and can be consulted in \cite[Lemma 4.3]{Lee-Parker}.
\end{proof}
\begin{proposition}
    \label{prop: uniform Lr bound}
    Let $M$ be a smooth, closed manifold of dimension $n \geq 3$ and consider a $W^{2,q}$-Riemannian metric $g$ on $M$ with $q > \tfrac{n}{2}$ satisfying $\vol_g(M) = 1$, $\R_g\in C^{0}(M)$ and $\lambda(M,g) < \lambda(\nS^n)$. Let $u_s \in W^{2,q}(M)$ be positive minimizers of $Q_g^s$ normalized by $\|u_s\|_{L^s(M,d\mu_g)} = 1$.\footnote{The existence of such minimisers is granted by Proposition \ref{prop: subcritical existence} and their regularity by Corollary \ref{prop: subcritical regularity}.} Then, there are constants $s_0 < 2^*$, $\delta=\delta(s_0) > 0$ and $C$ such that $\|u_s\|_{L^{2^*(1+\delta)}(M)} \leq C$ for all $s \in [s_0,2^*)$.\footnote{It follows from the proof as given in \cite[Proposition 4.4]{Lee-Parker} that $\delta=\delta(s_0)$ only depends on $\lambda_{s_0}(M,g)$.}
\end{proposition}
\begin{proof}
    Exploiting the fact that $\R_g\in C^0(M)$, the proof follows exactly in the same way as in \cite[Proposition 4.4]{Lee-Parker}, working with smooth approximations when integrating by parts. Aubin's Sobolev inequality needed in the proof can easily be seen to hold for $C^{0,\alpha}$-metrics, for instance, by using coordinates such that $g_{ij}(x) = \delta_{ij} + O\bigl(|x|^\alpha\bigr)$ rather than normal coordinates in the proof of \cite[Theorem 9]{AubinIsoperim}.
\end{proof}
\begin{remark}
    \label{RemarkCompasionZhang.1}
    Let us highlight that the above proposition is also needed for even lower regularity cases --see \cite[Proposition 4.4]{ZhangW1pAubin}--. We again emphasize that the conformal invariance of the problem and \cref{cor: Yamabe classification}, which allow us to consider the case $\R_g\in C^{0}(M)$, essentially reduces the above proposition to the one known for smooth metrics.
\end{remark}
Using the above results, 
we can prove the Aubin--Trudinger--Yamabe theorem:
\begin{theorem}
    \label{thm: Aubin-Trudinger-Yamabe}
    Let $M$ be a smooth, closed manifold of dimension $n \geq 3$ and consider a $W^{2,q}$-Riemannian metric $g$ on $M$ with $q > \tfrac{n}{2}$. If $\lambda(M,g) < \lambda(\nS^n)$, there exists a Riemannian metric in the $W^{2,q}$ conformal class of $g$ of constant scalar curvature.
\end{theorem}
\begin{proof}
    Since $\lambda(M, g)$ is a conformal invariant, by \cref{cor: Yamabe classification} we can assume without loss of generality that $\R_g$ is continuous. Moreover, we may rescale the metric such that $\vol_g(M) = 1$. By \cref{prop: uniform Lr bound} we know that there are fixed numbers $s_0< 2^*$ and $\delta=\delta(s_0) > 0$, such that $\|u_s\|_{L^r(M)} \leq C_0$ for all $s\in (s_0,2^{*})$, where $r = 2^*(1 + \delta)$ and $u_s \in W^{2,q}(M)$ are positive minimizers of $Q_g^s$ normalized by $\|u_s\|_{L^s(M,d\mu_g)} = 1$. We may, without loss of generality, assume that $C_0 \geq 1$. Fix some $\tfrac{n}{2} < t < \min\{q, \tfrac{r}{2^*-2}\}$, which is always possible because
    \begin{equation*}
        \frac{r}{2^* - 2} = \frac{2^*(1 + \delta)}{2^* - 2} = \frac{n}{2}(1 + \delta)
    \end{equation*}
    and $\frac{n}{2} < q$ by hypothesis.
   
    \medskip
    \textit{Step 1.} We first aim to obtain an estimate on $\|u_s\|_{W^{2,t}(M)}$ uniform in $s\in (s_0,2^{*})$, since the Sobolev inequality
    \begin{equation*}
        \|u_s\|_{C^{0,2-\frac{n}{t}}(M)} \leq C(M,n,t) \, \|u_s\|_{W^{2, t}(M)}
    \end{equation*}
    would imply a uniform Hölder bound ($t$ depends on $n$, $q$ and $\delta$, but not on $s$). By Lemma \ref{lemma: Lg is Fredholm} we know that
    \begin{equation}
        \label{eq: 2t elliptic estimate}
        \|u_s\|_{W^{2,t}(M)} \leq C(M,g,n,t,q)\bigl(\|u_s^{s-1}\|_{L^t(M)} + \|u_s\|_{L^t(M)}\bigr)
    \end{equation}
    holds for all $s \in (s_0, 2^*)$. If $t \leq r$ we use that $\mathrm{vol}_g(M)=1$ to directly estimate the second term with the $L^r(M)$ norm via Hölder. If $t > r$ instead, we have the sequence
    \begin{equation*}
        W^{2,t}(M) \hookrightarrow L^t(M) \hookrightarrow L^r(M)
    \end{equation*}
    with the first embedding being compact, and hence we can interpolate
    \begin{equation}
        \label{eq: Lt interpolation}
        \|u_s\|_{L^t(M)} \leq \varepsilon \|u_s\|_{W^{2,t}(M)} + C_\varepsilon(M,n,t,r) \,\|u_s\|_{L^r(M)} \,.
    \end{equation}
    Thus, we see that regardless of $t$ being above or below $r$, given any $\varepsilon>0$ there is a constant $C_\varepsilon(M,n,t,r)>0$ (independent of $s$) such that we can estimate the second term in \eqref{eq: 2t elliptic estimate} by \eqref{eq: Lt interpolation} for all $s \in (s_0, 2^*)$. On the other hand, we can manipulate the first term like
    \begin{align*}
        \|u_s^{s-1}\|_{L^t(M)} &= \|u_s^{s-2} u_s\|_{L^t(M)} 
        \\
        &\leq \|u_s^{s-2}\|_{L^{\frac{r}{s-2}}(M)} \|u_s\|_{L^p(M)} = \|u_s\|^{s-2}_{L^r(M)} \|u_s\|_{L^p(M)} \leq C_0^{2^* - 2}\|u_s\|_{L^p(M)} \,,
    \end{align*}
    where $p$ is the Hölder conjugate 
    \begin{equation*}
        \frac{1}{p} \doteq \frac{1}{t} - \frac{s - 2}{r}
    \end{equation*}
    and it is well-defined due to
    \begin{equation*}
        0 < \frac{1}{t} - \frac{2^* - 2}{r} < \frac{1}{t} - \frac{s - 2}{r} < \frac{1}{t} < 1 \,.
    \end{equation*}
    Now, if $p < r$, we simply use the Hölder inequality together with $\vol_g(M) = 1$ to estimate $\|u_s\|_{L^p(M)} \leq \|u_s\|_{L^r(M)}$. If $p \geq r$, define $l \doteq \tfrac{tr}{r - t(2^*-2)}$, which is well-defined, independent of $s$ and by
    \begin{equation*}
        \frac{1}{l} = \frac{1}{t} - \frac{2^* - 2}{r} < \frac{1}{t} - \frac{s - 2}{r} = \frac{1}{p}
    \end{equation*}
    it satisfies that $r \leq p < l$. Again, the compact-continuous embeddings
    \begin{equation*}
        W^{2,t}(M) \hookrightarrow L^{l}(M) \hookrightarrow L^r(M)
    \end{equation*}
    yield that for every $\eta > 0$ we have that
    \begin{equation*}
        \|u_s\|_{L^{l}(M)} \leq \eta \|u_s\|_{W^{2,t}(M)} + C_\eta(M,n,t,r) \|u_s\|_{L^r(M)} \,,
    \end{equation*}
    which combined with the Hölder inequality $\|u_s\|_{L^p(M)} \leq \|u_s\|_{L^l(M)}$ allows us to estimate the first term in \eqref{eq: 2t elliptic estimate} by
    \begin{equation*}
        \|u_s^{s-1}\|_{L^t(M)} \leq C_0^{2^*-2}\bigl(\eta \|u_s\|_{W^{2,t}(M)} + C_\eta(M,n,t,r) \|u_s\|_{L^r(M)}\bigr) \,.
    \end{equation*}
    Putting it all together and noticing that $r = r(n, \delta)$ and $t = t(n, q, \delta)$ we obtain
    \begin{align*}
        \|u_s\|_{W^{2,t}(M)} &\leq C(M,g,n,q,\delta) C_0^{2^*-2}\Big(\eta\|u_s\|_{W^{2,t}(M)} + C_\eta(M,n,q, \delta) \|u_s\|_{L^r(M)}\Big) 
        \\
        &\qquad + C(M,g,n,q,\delta) \Big(\varepsilon \|u_s\|_{W^{2,t}(M)} + C_\varepsilon(M,n,q,\delta) \,\|u_s\|_{L^r(M)}\Big)
        \\
        &\leq C(M,g,n,q,\delta) \bigl(C_0^{2^*-2}\eta + \varepsilon\bigr) \|u_s\|_{W^{2,t}(M)} + C(M,g,n,q,\delta,C_0,\varepsilon,\eta) \,.
    \end{align*}
    Therefore, we see that choosing $\eta = \varepsilon = C(M,g,n,q,\delta)^{-1}(2C_0^{2^*-2} + 2)^{-1}$ we absorb the first term in the right-hand side above into the left-hand side and establish
    \begin{equation*}
        \|u_s\|_{W^{2,t}(M)} \leq 2C(M,g,n,q,\delta,C_0)
    \end{equation*}
    for all $s \in (s_0,2^*)$, as desired.

    \medskip
    \textit{Step 2.} By the previous step, there holds $\|u_s\|_{C^{0,2-\frac{n}{t}}(M)} \leq C_1(M,g,n,q,\delta,C_0)$ for all $s\in (s_0,2^{*})$, which we may assume to be $C_1(M,n,q,\delta,C_0) \geq 1$. By Arzelà--Ascoli, taking a sequence $\{u_{s_j}\}_{j=1}^{\infty}$ with $s_0<s_j<2^{*}$ such that $s_j\to 2^{*}$, we can then extract a sub-sequence, which we still denote by $\{u_{s_j}\}_{j=1}^{\infty}$, converging uniformly to $u \in C^0(M)$. Since $W^{2,t}(M) \hookrightarrow W^{1,2}(M)$, then $u_s$ are also uniformly bounded in $W^{1,2}(M)$. Consequently, up to a further subsequence to which we restrict and still denote by $\{u_{s_j}\}_{j=1}^{\infty}$, $u_{s_j} \rightharpoonup u$ in $W^{1,2}(M)$. Moreover, by \cref{prop: subcritical regularity}, each such $u_{s_j}\in W^{2,q}(M)$ satisfies $\mathscr{L}_gu_{s_j}=\lambda_{s_j}u_{s_j}^{s_j-1}$ and by \cref{lemma: Lg is Fredholm},
    \begin{align*}
        \Vert u_{s_j}\Vert_{W^{2,q}(M)} &\leq C(M,g,q)\bigl(\Vert \lambda_{s_j}u_{s_j}^{s_j-1}\Vert_{L^q(M)} + \Vert u_{s_j}\Vert_{L^q(M)}\bigr)
        \\
        &\leq C(M,g,q)\bigl(|\lambda_{s_j}|\Vert u_{s_j}\Vert_{C^0(M)}^{s_j-1} + \Vert u_{s_j}\Vert_{C^0(M)}\bigr)
        \\
        &\leq C(M,g,q)\Bigl(|\lambda_{s_j}| C^{s_j-1}_1(M,g,n,q,\delta,C_0) + C_1(M,g,n,q,\delta,C_0)\Bigr) \,,
    \end{align*}
    where above we have appealed to the uniform Hölder bound and the volume normalization $\vol_g(M)=1$. Since $C_1\geq 1$, we can also estimate $C_1 \leq C^{s_j-1}_1 \leq C^{2^{*}-1}_1$, and in combination with \cref{lemma: continuity yamabe invariant}, which grants $|\lambda_{s_j}|\leq |\lambda_2|$, we obtain
    \begin{align*}
        \Vert u_{s_j}\Vert_{W^{2,q}(M)}&\leq C(M,g,n,q,\delta,C_0)|\lambda_{2}| \,.
    \end{align*}
    Since $W^{2,q}(M)\hookrightarrow W^{1,2}(M)$ is compact due to $q>\frac{n}{2}$, this implies that there is a further subsequence, to which we restrict and still denote by $\{u_{s_j}\}_{j=1}^{\infty}$, such that $u_{s_j} \to \hat{u}$ in $W^{1,2}(M)$. Clearly, by uniqueness of the weak $W^{1,2}$-limit, $\hat{u}=u\in C^{0}(M)$. 
    Moreover, since $\Vert u_{s_j}\Vert_{L^{s_j}(M,d\mu_g)}=1$ for all $j$, then 
    \begin{align*}
        \Vert u\Vert_{L^{2^{*}}(M,d\mu_g)}=\lim_{j\to\infty}\Vert u_{s_j}\Vert_{L^{2^{*}}(M,d\mu_g)}\geq \lim_{j\to\infty}\Vert u_{s_j}\Vert_{L^{s_j}(M,d\mu_g)}=1,
    \end{align*}
    by $L^{2^{*}}(M)\hookrightarrow L^{s_j}(M)$ and $\vol_g(M)=1$. Thus, we deduce that $u\not\equiv 0$ and $u\geq 0$, the latter inequality holding since $u$ is the uniform limit of positive functions $u_{s_j}>0$. Now, to extract a limit equation for $u\in W^{1,2}(M)\cap C^0(M)$, we start with the associated sequence of $\{u_{s_j}\}_{j=1}^{\infty}\subset W^{1,2}(M)$ which satisfies
    \begin{align}
        \label{AubinSeqMinim}
        \Lg u_{s_j} = \lambda_{s_j}u_{s_{j}}^{s_j-1}.
    \end{align}
    Since $\{\lambda_{s_j}\}_{j=1}^n$ is a bounded sequence, it converges possibly up to a subsequence to some $\bar{\lambda}\in \mathbb{R}$. Restricting to the corresponding subsequence of functions $\{u_{s_j}\}_{j=1}^\infty$, the continuity of $\Lg : W^{1,2}(M) \to W^{-1,2}(M)$ implies that $\Lg u_{s_j} \to \Lg u$ in $W^{-1,2}(M)$. Besides, the $C^{0}(M)$ convergence of $u_{s_j} \to u$ implies point-wise convergence $u^{s_j-1}_{s_j} \to u^{2^{*}-1}$, and since $\|u_{s_j}\|_{C^0(M)} \leq C_1$, by dominated convergence there holds
    \begin{align}
        \label{DCTConv}
        \int_{M}u^{s_j-1}_{s_{j}}v \, d\mu_g \to  \int_{M}u^{2^{*}-1}v \, d\mu_g
    \end{align}
    for all $v \in W^{1,2}(M)$. Thus, $u^{s_j-1}_{s_j} \to u^{2^{*}-1}$ in $W^{-1,2}(M)$ as well. We deduce that the limit function satisfies
    \begin{align}
        \label{AubinSeqMinim.2}
        \Lg u = \bar{\lambda} u^{2^{*}-1},
    \end{align}
    as an equation in $W^{-1,2}(M)$. On the other hand, we have that
    \begin{equation*}
        \int_{M}u^{s_j}_{s_{j}}d\mu_g\to \int_{M}u^{2^{*}} d\mu_g \,,
    \end{equation*}
    as $u_{s_j}^{s_j} \to u^{2^*}$ point-wise, and consequently $\|u\|_{L^{2^*}(M,d\mu_g)} = 1$. Testing \eqref{AubinSeqMinim.2} with $u$ itself implies that $\bar{\lambda}=Q_g(u)$. When $\lambda(M,g) \geq 0$, we know that $s \mapsto \lambda_s$ is continuous from the left by \cref{lemma: continuity yamabe invariant} and thus $\bar{\lambda}=\lim_{j\to \infty}\lambda_{s_j}=\lambda(M,g)$. On the other hand, when  $\lambda(M,g)< 0$, again by \cref{lemma: continuity yamabe invariant} we know that $-\lambda_s(M,g)$ is non-increasing in $s$. Thus, one has $\lambda_{s_j}(M,g)\leq \lambda(M,g)$ for all $s_j<2^{*}$, which implies that $Q_g(u)=\bar{\lambda}\leq \lambda(M,g)$, so it must be that $\bar{\lambda}=\lambda(M,g)$. We have therefore established that, regardless of the sign of $\lambda(M,g)$, the function  $u \in C^0(M)\cap W^{1,2}(M)$ is a $L^{2^*}(M,d\mu_g)$-normalized, non-negative weak solution to the Yamabe PDE
    \begin{align*}
        \Lg u=\lambda(M,g)u^{2^{*}-1}.
    \end{align*}
    Since $u^{2^*-1} \in L^p(M)$ for any $p \geq 1$, \cref{thm: conf Laplace elliptic regularity} implies that $u\in W^{2,q}(M)$. Therefore, the metric $u^{\frac{4}{n-2}}g\in W^{2,q}(M)$ has constant scalar curvature and the proof concludes.
\end{proof}
We notice that the above \cref{thm: Aubin-Trudinger-Yamabe} settles the Yamabe problem for the case of non-positive Yamabe invariant. In particular, in light of the regularity theory of \cref{thm: conf Laplace elliptic regularity}, a straightforward bootstrap argument yields, as an immediate corollary
\begin{replettertheorem}{theorem B}
\label{Coro-Yamabenegative}
  \emph{  Let $M$ be a smooth, closed manifold of dimension $n \geq 3$ and consider a $W^{k,q}$-Riemannian metric $g$ with $k \geq 2$ and $q > \tfrac{n}{2}$. Suppose that $\lambda(M,g) \leq 0$. Then, there exists a $W^{k,q}$-Riemannian metric conformal to $g$ of constant scalar curvature.}
\end{replettertheorem}

\subsection{Decompactification via conformal Green function}
\label{Decompactification via Conformal Green's function}
Let $M$ be a smooth, closed 3-manifold and $g$ be a $W^{2,q}$-Riemannian metric on $M$ with $q > 3$ and $\lambda(M, g) > 0$. Fix an arbitrary point $p \in M$ and consider the \textit{unnormalized Green function} $\G$ given in \cref{cor: unnormalized Green 3dim}. We define the non-compact Riemannian manifold
\begin{equation}
    \label{eq: decompactified metric}
    \hat{M} \doteq M\setminus \{p\} \,, \quad \hat g \doteq \G^4 g \,.
\end{equation}
In view of \cref{cor: unnormalized Green 3dim}, there exists a normal coordinate chart $(U,x^i)$ at $p$ such that
\begin{equation}
    \label{expansion_with_v}
    \G^4(x) = \frac{1}{|x|^4}\bigl(1 + |x|\h(x)\bigr)^4 \doteq \frac{1}{|x|^4}\bigl(1 + v(x)\bigr)
\end{equation}
for some function $\h \in W^{2,r}(U, d\mu_g)$, for any $r < \tfrac{3q}{q+3} < 3$. Since $|x| \in W^{2,r}(U, d\mu_g)$ and $W^{2,r}$ is an algebra under multiplication for $r\in (\frac{3}{2},\tfrac{3q}{q+3})$, it follows that $v \in W^{2,r}(U, d\mu_g)$. It is easy to see that $v(0) = 0$. On the other hand, by \cref{NormalCoordsC1}, in normal coordinates $g$ can be written as
\begin{equation*}
    g_{ij}(x) = \delta_{ij} + k_{ij}(x)
\end{equation*}
for some symmetric $k_{ij} \in W^{2,q}(U, d\mu_g)$ satisfying $k_{ij}(0) = 0$ and $\partial k_{ij}(0) =0$. Combining these expressions, we get that
\begin{equation}
    \label{def: k}
    \hat g_{ij}(x) = \frac{1}{|x|^4}\bigl(1 + v(x)\bigr)\bigl(\delta_{ij} + k_{ij}(x)\bigr) \,.
\end{equation}
Using the Kelvin transformed coordinates or \emph{asymptotic coordinates}
\begin{equation}
    \label{eq: asyCoord_z}
    z^i \doteq \frac{x^i}{|x|^2}
\end{equation}
on $\hat U \doteq U \setminus \{\p\} \subset \hat M$ we can write 
\begin{align}
    \label{eq: definition hat k}
        \hat g_{ab}(z) &= \frac{\partial x^i}{\partial z^a} \frac{\partial x^j}{\partial z^b} \, \hat g_{ij}(x) = \delta_{ab} + v(x)\delta_{ab} \nonumber
        \\
        &\quad + \left(k_{ab}(x) - \frac{2 \left( k_{aj}(x) z^jz^b + k_{bj}(x) z^j z^a \right)}{|z|^2} + 4k_{ij}(x) \frac{z^iz^jz^az^b}{|z|^4}\right)\bigl(1 + v(x)\bigr)
        \\
        &\doteq \delta_{ab} + \hat k_{ab}(z) \,. \nonumber
\end{align}

\begin{proposition}
    \label{prop: regularity_g_hat}
    Let $M$ be a smooth, closed 3-manifold and $g$ a $W^{2,q}$-Riemannian metric  with $q > 3$ and $\lambda(M, g) > 0$. Let $(\hat M, \hat g)$ be defined as in \eqref{eq: decompactified metric}. Then, 
    \begin{itemize}
        \item[$(i)$] $\hat g \in W^{2,q}_{loc}(\hat M)$ has zero scalar curvature,
        \item[$(ii)$] $(\hat M, \hat g)$ is $AS_\beta(1) \cap W^{2,r}_{-\tau}$-asymptotically Euclidean with respect to the structure at infinity given by the inverted coordinates $\{z^i\}_{i=1}^3$ introduced in \eqref{eq: asyCoord_z} for all $\tfrac{3}{2} < r < \frac{3q}{3+q}$ and $\tau = 2 - \tfrac{3}{r}$, where $\beta \doteq 2-\tfrac{3}{r}.$
    \end{itemize}
\end{proposition}
\begin{proof}
    $(i)$ From \cref{cor: unnormalized Green 3dim} we know that $\G \in W^{2,q}_{loc}(\hat M)$, so the regularity statement for $\hat g$ follows since $W^{2,q}_{loc}(\hat M)$ is an algebra under multiplication. Moreover, $\Lg\G = 0$ in $\hat M$, so $\R_{\hat g} \equiv 0$ by \eqref{eq: ConformalScalar}.

    $(ii)$ Notice that the function $v$ defined in \eqref{expansion_with_v}, due to the expansion of $\h$ in \cref{cor: unnormalized Green 3dim}, takes the form
    \begin{equation}
        \label{eq: v_expansion}
        v(x) = 4A |x|+ \BO_1(|x|^{1 + \beta}) \qquad \text{as} \quad |x| \to 0 \,.
    \end{equation}
    In particular, $\partial v = \BO(1)$. Thus, combining \eqref{eq: definition hat k} with
    \begin{equation*}
        k_{ij}(z) = \BO_1(|z|^{-1-\alpha}), \quad v(z) = \frac{4A}{|z|} + \BO_1\left(|z|^{-1-\beta}\right)
    \end{equation*}
    we get the expansion
    \begin{equation}
        \label{AS(1)condition}
        \hat g_{ij}(z) = \left(1 + \frac{4A}{|z|}\right)\delta_{ij} + \BO_1\bigl(|z|^{-1-\beta}\bigr) \qquad \text{as} \quad |z| \to \infty \,,
    \end{equation}
    which establishes the $AS_\beta(1)$ condition in these coordinates.

    To further show that $(\hat M, \hat g)$ is $W^{2,r}_{-\tau}$-AE, one can follow \cite[Lemma 5.2]{MaxDil} closely, where the singular part of the Green function $|x|^{2-n}$ is used as a conformal factor, instead of $\G$. Below we outline the main steps for this procedure.
    
    From the discussion above, we must show that the functions $\hat k_{ij}$ defined by \eqref{eq: definition hat k} are in $W^{2,r}_{-\tau}(\nR^3 \setminus \overline{B}, d\mu_{\hat g})$. To that purpose, using that $k_{ij} \in C^1(U)$ and suppressing the indices, we compute
    \begin{align}
        \label{eq: firstDerivative_k_hat}
        \begin{split}
            \frac{\partial \hat k}{\partial z} &= \frac{\partial x}{\partial z} \cdot \frac{\partial \hat k}{\partial x} 
            = \BO_1(|x|^2)\BO_1(\partial_x v)\Big(1 + \BO_1(k)\Big) + \Big(\BO_1(\partial_x k)\BO_1(|x|^2) 
            \\
            & \qquad \qquad+ \BO_1(k)\BO_1(|x|)\Big)\bigl(1 + v\bigr)\,,
        \end{split}
    \end{align}
    where we have used \eqref{eq: definition hat k} and $\tfrac{\partial x}{\partial z} = \BO_1(|x|^2)$. Similarly, the second derivatives read
    \begin{align}
        \label{eq: second_derivative_k_hat} 
        \begin{split}
            \frac{\partial^2\hat k}{\partial z^2} &= \frac{\partial}{\partial z}\left(\frac{\partial\hat k}{\partial z}\right) = \frac{\partial x}{\partial z} \cdot \frac{\partial }{\partial x}\left(\frac{\partial\hat k}{\partial z}\right)
            \\
            &= \BO_1(|x|^4)\Big(\BO(\partial_x^2v)\bigl(1 + \BO_1(k)\bigr) + \BO(\partial_x v)\BO(\partial_x k) + \BO(\partial_x^2k)\bigl(1+v\bigr)\Big)
            \\
            &\qquad + \BO(|x|^3)\Big(\BO_1(\partial_x v)\bigl(1 + \BO_1(k)\bigr) + \BO_1(\partial_x k)\bigl(1+v\bigr)\Big)
            \\ 
            &\qquad + \BO_1(|x|^2)\BO_1(k)\bigl(1+v\bigr) \,.
         \end{split}   
    \end{align}
    Since the integral estimate of the first derivative is a simpler version of the one for the second derivative, we only carry out the details of the latter.
    Using that
    \begin{equation*}
        d\mu_{\hat g}(z) = (1+v)^{3/2}|z|^{6} d\mu_g(x)
    \end{equation*}
    and setting $-3 + \tau r + 2r = 4r - 6$, we can use \eqref{eq: second_derivative_k_hat} to obtain 
    \begingroup
\allowdisplaybreaks
    \begin{align*}
        \int_{\nR^3 \setminus B} &\biggl|\frac{\partial^2\hat k}{\partial z^2}\biggr|^r  |z|^{4r-6} \,d\mu_{\hat g}(z) = \\
        & = 
        \int_{B} \Big|\BO(\partial^2v)\bigl(1 + \BO_1(k)\bigr) + \BO(\partial v)\BO(\partial k) + \BO(\partial^2k)\bigl(1+v\bigr)\Big|^r (1+v)^{3/2} d\mu_g(x)
        \\
        &\quad + \int_B |x|^{-r}\Big|\BO_1(\partial v)\bigl(1 + \BO_1(k)\bigr) + \BO_1(\partial k)\bigl(1+v\bigr)\Big|^r (1+v)^{3/2} d\mu_g(x)
        \\
        &\quad + \int_B |x|^{-2r}\big|\BO_1(k)\bigl(1+v\bigr)\big|^r (1+v)^{3/2} d\mu_g(x)
    \end{align*}
    \endgroup
    Recalling that $v \in W^{2,r}(U, d\mu_g) \subset C^{0,\beta}(U)$ and  $k \in W^{2,q}(U, d\mu_g) \subset C^{1,\alpha}(U)$, as well
    
    \noindent as the Sobolev multiplication
    \begin{equation*}
        W^{1,r}(U) \otimes W^{1,q}(U) \hookrightarrow W^{1,r}(U) \,,
    \end{equation*}
    it is not hard to see that the first integral is finite. For the second integral, we essentially have to estimate 
    \begin{equation*}
        \int_B \frac{|\partial v|^r}{|x|^{r}} d\mu_g \qquad \text{and} \qquad \int_B \frac{|\partial k|^r}{|x|^{r}} d\mu_g \,. 
    \end{equation*}
    This is precisely the estimate (5.9) in \cite[Lemma 5.2]{MaxDil}, which we can use since $\tfrac{3}{2} < r < 3$ and $v(0) = k_{ij}(0) = 0$. Our third integral can be estimated using the same argument via the estimate (5.8) in \cite[Lemma 5.2]{MaxDil}. For details on these estimates, see also \cite[Lemma B.1 and Proposition B.1]{avalos2024sobolev}. This concludes the proof.
\end{proof}
\cref{prop: regularity_g_hat} allows us to transfer the Yamabe problem from the compact manifold $(M, g)$ to the AE manifold $(\hat M, \hat g)$. Following \cite{MaxDil} we define the Yamabe invariant of a $W^{2,r}_{-\tau}$-AE manifold $(\hat M, \hat g)$, with $r > \tfrac{3}{2}$ and $\tau > 0$ to be
\begin{equation}
    \label{YamabeInvAE}
    \lambda(\hat M, \hat g) \doteq \inf\left\{Q_{\hat g}(u) \; : \; u \in W^{1,2}_{-1/2}(\hat M) \;, \; u \not\equiv 0 \right\} \,.
\end{equation}
By the results within \cite[Section 3]{MaxDil}, this is well-defined and in fact $W^{1,2}_{-1/2}(\hat M)$ is the largest sensible function space for this variational problem. In parallel to \cite[Proposition 5.3]{MaxDil}, one can easily check the following:
\begin{proposition}
    \label{prop: compac-noncompact yamabe correspondence}
    Let $M$ be a smooth, closed $3$-manifold and $g$ a $W^{2,q}$-Riemannian metric on $M$ with $q > 3$ and $\lambda(M,g)$. Then,
    \begin{equation*}
        \lambda(\hat M, \hat g) = \lambda(M, g) \,,
    \end{equation*}
    where $(\hat M, \hat g)$ is the asymptotically Euclidean Riemannian manifold defined in \eqref{eq: decompactified metric}.
\end{proposition}

\subsection{The Positive Mass Theorem}
\label{The Positive Mass Theorem}
Just as in the case of smooth metrics, we intend to use the PMT in order to get a sign for the constant $A$ in the expansion of \cref{cor: unnormalized Green 3dim}. Standard versions of the PMT, such as that of \cite{Bartnik86}, apply to $W^{2,p}_{-\tau}$-AE manifolds with $p > n$ and $\tau \geq \tfrac{n-2}{2}$. From Proposition \ref{prop: regularity_g_hat} it can immediately be seen, however, that such hypotheses are generally not satisfied in our case. Nevertheless, we will appeal to a refined version of the PMT due to D. Lee and P. LeFloch, which states as follows:
\begin{theorem}[Theorem 1.1 \cite{LeeLeF}] 
    \label{thm: LeeLF}
    Let $\hat M$ be a
    smooth spin manifold of dimension $n \geq 3$ and consider a $C^0 \,\cap\, W^{1, n}_{-\gamma}$-asymptotically Euclidean Riemannian metric $\hat g$ on $\hat M$ with $\gamma \geq \tfrac{n-2}{2}$. If the distributional scalar curvature $\R_{\hat g}$ of $\hat g$ is non-negative, then its generalized ADM mass $\mathbf{m}(\hat M, \hat g)$ is non-negative. Moreover, $\mathbf{m}(\hat M, \hat g) = 0$ occurs only when $(\hat M, \hat g)$ is isometric to Euclidean space.
\end{theorem}
First, we observe that the decompactified manifold $(\hat M, \hat g)$ defined in \eqref{eq: decompactified metric} satisfies the regularity and decay hypothesis of \cref{thm: LeeLF} due to \cref{prop: regularity_g_hat}. Moreover, in dimension three, the spin assumption is equivalent to orientability. Before recalling the definition of the \textit{generalized ADM mass} and \textit{distributional scalar curvature}, let us remark that the rigidity statement, as written in \cref{thm: LeeLF}, is rather loose. Since the metric itself is rough, one would be interested in the regularity of the associated isometry to Euclidean space, which is not tracked down in the proof of \cite[Theorem 1.1]{LeeLeF}. Such regularity is key in our application to the Yamabe problem for the case where the original metric is conformal to $\nS^3$ and will be provided in \cref{PropMassRigidity.1} and \cref{thm: A=0 case}.

Under the hypotheses of \cref{thm: LeeLF}, the authors in \cite{LeeLeF} introduced the notion of generalized ADM mass 
\begin{align}
    \label{LeeleflochMass.2}
    \mathbf{m}(\hat{M},\hat{g}) \doteq \frac{1}{2(n-1) \omega_{n-1} } \, \inf_{\varepsilon > 0} \,\,\liminf_{\rho \to \infty} \,\frac{1}{\varepsilon} \,\int_{\{ \rho < \mathrm{r} < \rho + \varepsilon\}} \operatorname{V} \cdot \,\overline{\nabla} \mathrm{r}  \; d\mu_{\bar g} \,, 
\end{align}
where the vector field $\operatorname{V}$ is given by
\begin{equation*}
    \operatorname{V}^k \doteq \hat g^{ij} \hat g^{kl} (\overline{\nabla}_j \hat g_{il} - \overline{\nabla}_l \hat g_{ij}) \,,
\end{equation*}
$\bar g$ is a background smooth metric isometric to the Euclidean metric in the asymptotic chart, $\overline{\nabla}$ denotes the Levi--Civita connection of $\bar g$ and the function $\mathrm{r}$ is a smooth positive function on $\hat{M}$ which agrees with the radial coordinate in the asymptotic chart. Naturally, $\mathbf{m}$ agrees with classical notions of ADM mass whenever the metric is regular enough and the scalar curvature is integrable \cite[Corollary 2.5]{LeeLeF}. 

In \cref{prop: regularity_g_hat} we have shown that the decompactified manifold $(\hat M, \hat g)$ constructed in \eqref{eq: decompactified metric} is of class $AS(1)$, which implies that the decay conditions of \cref{thm: LeeLF} are fulfilled and hence $\mathbf m$ is well-defined. Moreover, it is scalar flat and thus the distributional scalar curvature is clearly non-negative. In fact, we have
\begin{proposition}\label{PropPMTforGp}
     Let $M$ be an orientable, smooth, closed 3-manifold and $g$ a $W^{2,q}$ Riemannian metric  with $q > 3$ and $\lambda(M, g) > 0$. Let $(\hat M, \hat g)$ be defined as in \eqref{eq: decompactified metric}. Then,
        \begin{align}
            \label{PropPMTforGp.1}
            0\leq \textbf{m}(\hat{M},\hat{g}) = 2A,
        \end{align}
        where $A$ is the constant appearing in the expansion of \cref{cor: unnormalized Green 3dim}.
\end{proposition}
\begin{proof}
    If $M$ is orientable, so is $\hat M$ and combining this with \cref{prop: regularity_g_hat}, 
    $(\hat M, \hat{g})$ satisfies the hypotheses of Theorem \ref{thm: LeeLF}. Consequently, $\textbf{m}(\hat{M},\hat{g})\geq 0$. Let us now explicitly show that $\textbf{m}(\hat{M},\hat{g}) = 2A$. Given a background metric $\bar g$ such that in the asymptotic chart $\bar g_{ij}(z) = \delta_{ij}$, we get from the Schwarzschildean expansion \eqref{AS(1)condition} that
    \begin{align*}
        \operatorname{V}^k
        & = \frac{4A}{|z|^2}\, \left( - \frac{z^k}{|z|}  + 3 \frac{z^k}{|z|} \right) + \BO(|z|^{-2-\beta}) = \frac{8A}{|z|^2} \,\frac{z^k}{|z|} + \BO(|z|^{-2-\beta}) \,.
    \end{align*}
    Hence, using for example $\mathrm r = |z|$, the generalized ADM mass reads
    \begin{align*}
        \mathbf{m}(\hat M, \hat g) = \frac{1}{4} \, \inf_{\varepsilon > 0} \,\,  \liminf_{\rho \to \infty} \, \frac{1}{\varepsilon} \,\int_{\rho}^{\rho + \varepsilon} \left( 8A + \BO(\mathrm{r} ^{-\beta}) \right) \, d\mathrm{r}
    \end{align*}
    Now, estimating $|\BO(\mathrm{r} ^{-\beta})|\leq C\mathrm{r} ^{-\beta}$ for a given constant $C>0$ independent of $\mathrm{r}>R_0$ for some large enough $R_0>0$, we can compute for any $\varepsilon>0$ small enough
    \begin{align*}
        \frac{1}{\varepsilon}\int_{\rho}^{\rho + \varepsilon} \left( 8A + \BO(\mathrm{r} ^{-\beta}) \right) \, d\mathrm{r}&=8A + \frac{C}{\varepsilon}\left( \rho^{1-\beta} + (1-\beta)\varepsilon \rho^{-\beta}  + \BO(\varepsilon^2) \BO(\rho^{-1 -\beta}) - \rho^{1-\beta} \right)
        \\
        &= 8A +  C((1-\beta)\rho^{-\beta}  + \BO(\varepsilon) \BO(\rho^{-1 -\beta}) ),
    \end{align*}
    which implies that for any such $\varepsilon>0$
    \begin{align*}
        \liminf_{\rho \to \infty} \frac{1}{\varepsilon}\int_{\rho}^{\rho + \varepsilon} \left( 8A + \BO(\mathrm{r} ^{-\beta}) \right) \, d\mathrm{r}=8A \,.
    \end{align*}
    We conclude that $\textbf{m}(\hat{M},\hat{g})=2A$.
\end{proof}
We now study the critical case $2A=\textbf{m}=0$ and examine the regularity of the isometry $(\hat{M},\hat{g})\cong\nR^3$ of \cref{thm: LeeLF}. To do so, one could carefully track down the proof of \cite[Theorem 1.1]{LeeLeF} using the extra properties $(\hat{M},\hat{g})$ inherits. In order to avoid introducing spinorial language and with the applicability to higher dimensions in mind, we instead follow an adaptation of the the classic non-spinorial \cite[Lemma 10.7]{Lee-Parker}.
\begin{proposition}
    \label{PropMassRigidity.1}
     Let $M$ be an orientable, smooth, closed 3-manifold and $g$ a $W^{2,q}$ Riemannian metric  with $q > 3$ and $\lambda(M, g) > 0$. Let $(\hat M, \hat g)$ and $\{z^i\}_{i=1}^3$ be defined as in \eqref{eq: decompactified metric} and \eqref{eq: asyCoord_z}. If $ \mathbf{m}(\hat M, \hat g) = 0$, then there exists a $W^{3,q}_{loc}$-isometry
     \begin{equation*}
         y : (\hat M, \hat g) \to \nR^3
     \end{equation*}
     such that
     \begin{equation}
        \label{eq: hatM R3 isometry expansion}
         y^i = z^i + \BO_1\bigl(|z|^{-\varepsilon}\bigr) \qquad \text{as} \quad |z| \to \infty
     \end{equation}
     for some $\varepsilon \in (0,1)$.
\end{proposition}
\begin{proof}
    The proof consists of two steps:

    \textit{Step 1.} One may closely follow \cite[Lemma 10.7]{Lee-Parker} to show that if $\mathbf{m}(\hat M, \hat g)=0$, then $\Ric_{\hat g} \equiv 0$. In doing so, one only need to justify that the same arguments remain true under our regularity conditions. We outline below the necessary changes.
    
    In \cite[Lemma 10.7]{Lee-Parker} one considers a one-parameter family of AE-metrics $\gamma_t=\hat{g}+th$, $h\in C^{\infty}_0(T_2\hat{M})$ and $t\in (-\epsilon,\epsilon)$ for some $\epsilon>0$, and corresponding conformal deformations $\varphi_t$, $\varphi_t\to 1$ as $|z|\to\infty$, such that $\hat{g}_t=\varphi_t^{4}\gamma_t$ satisfies $R_{g_t}\equiv 0$ for all $t\in (-\epsilon,\epsilon)$. By Proposition \ref{prop: regularity_g_hat}, $\gamma_t$ are all $W^{2,r}_{-\tau}$-AE and thus the existence of such $\varphi_t$ follows from \cite[Theorem 5.1 and Lemma 4.3]{MaxDil} noticing that  $\lambda(\hat{M},\gamma_t)>0$ for all $t \in (-\varepsilon,\varepsilon)$ if $\epsilon$ is small. These solutions are easily seen to be differentiable with respect to $t$ around $t=0$. Moreover, $\psi_t\doteq 1-\varphi_t\in W^{2,r}_{-\tau}(\hat{M})$ satisfies
    \begin{align}\label{ConfDefRig.2}
        -8\Delta_{\gamma_t} \psi_t + \R_{\gamma_t}\psi_t = -\R_{\gamma_t},
    \end{align}
    where $\R_{\gamma_t}$ is by construction compactly supported. \cref{thm: 12 local elliptic regularity} grants that $\psi_t\in W^{2,q}_{loc}(\hat{M})$, and then, by \cref{BartniksProp1.6}, $\psi_t\in W^{2,q}_{-\tau}(\hat{M})$. Before proceeding with the proof as in \cite[Lemma 10.7]{Lee-Parker}, we require slightly better asymptotic control for $\psi_t$. This can be achieved by a harmonic expansion since
    \begin{align*}
        \Delta_{\gamma_t}\psi_t(z)=0 \quad \text{ on } \; \nR^3\setminus\overline{B_R(0)} \Longrightarrow \Delta \psi_t(z) = -\bigl(\gamma_t^{ij} - \delta_{ij}\bigr)\partial_i\partial_j \psi_t(z) + \gamma_t^{ij}\Gamma_{ij}^l \partial_l\psi_t
    \end{align*}
    for $R > 0$ large enough. Using that $|\gamma_t^{ij}-\delta_{ij}|\leq A|z|^{-1}$ and $\gamma_t^{ij}\Gamma^l_{ij}\leq C(A)|z|^{-2}$, we see that $\Delta \psi_t(z) \in L^q_{-3-\delta}\bigl(\nR^3\setminus\overline{B_R(0)}\bigr)$ for all $\delta < 1$. One may then proceed exactly as in the proof of \cite[Lemma 5.3, after eq. (5.6)]{avalos2024sobolev} to prove the existence of a constant $C_t$ such that
    \begin{align}
        \label{PMTRigidityConfFactorExpansion.1}
        \psi_t(z) = \frac{C_t}{|z|} + \BO_1(|z|^{-1-\delta}) \quad \text{ as } \; |z| \to \infty \,.
    \end{align}
    Therefore, the metric $g_t=\phi^4_t\gamma_t$ has the asymptotics
    \begin{align}
        \begin{split}
        (g_t)_{ij}(z) &=\left( 1 + \frac{4A_t}{|z|}\right)\delta_{ij} + \BO_1(|z|^{-1-\alpha}),
    \end{split}
    \end{align}
    where $\alpha >0$ and $A_t=A+C_t$.

    We are now in a position to continue with the proof of \cite[Lemma 10.7]{Lee-Parker} where standard variational formulas yield  (see \cite[(8.9)-(8.11)]{Lee-Parker})
    \begin{align}
        \label{PMTRigidityVariationalArgument}
        \frac{d}{dt}\left(\int_{\hat{M}}\R_{g_t}d\mu_{g_t} + E_{ADM}(\hat{g}_t)\right)\Big\vert_{t=0} = \int_{\hat{M}}\Bigl\langle\Ric_{\hat{g}} - \frac{1}{2}\R_{\hat g} \hat g, h\Bigr\rangle_{\hat{g}}d\mu_{\hat{g}}.
    \end{align}
    Since $\R_{g_t} \equiv 0$ for all $t$, the first term in the left-hand-side and the second term in the right-hand-side vanish. Also, because $g_t$ are all $AS(1)$ with $\R_{g_t} \equiv 0$, the generalized ADM mass $\textbf{m}(g_t)$ is well-defined and non-negative for all $t$. Moreover, from \eqref{PropPMTforGp.1} we know that $2A_t = E_{ADM}(g_t) \geq 0$. Thus, if $g_{0} = \hat{g}$ satisfies $A_{0}=A=0$, then $E_{ADM}(g_{0})$ is a local minimum of $E_{ADM}(g_t)$ and consequently
    \begin{align*}
        0 = \frac{d}{dt}E_{ADM}(\hat{g}_t)\Big\vert_{t=0} = \int_{\hat{M}}\bigl\langle\Ric_{\hat{g}}, h\bigr\rangle_{\hat{g}}d\mu_{\hat{g}}
    \end{align*}
    for all $h \in C^\infty_0(T_2\hat M)$, showing that $\Ric_{\hat g} \equiv 0$. Since $(\hat M, \hat g)$ is 3 dimensional, it is flat.\footnote{At this point, one could use the ideas of M. Taylor in \cite[Section 3]{Taylor_ConfFlat} to show it is $W^{3,q}_{loc}$-locally isometric to $\nR^3$.}

    \medskip
    \textit{Step 2.} Next, we aim to find a global orthonormal frame using Bochner's formula (c.f. \cite[Proposition 10.2]{Lee-Parker}). We begin by constructing harmonic coordinates at infinity following ideas of R. Bartnik \cite[Theorem 3.1]{Bartnik86}: extend $z^i$ to $C^{\infty}(\hat{M})$ functions and note asymptotically $\Delta_{\hat{g}}z^i=\hat{g}^{kl}(z)\Gamma^i_{kl}(z)$. If $A=0$, then $\hat{g}^{kl}(z)\Gamma^i_{kl}(z)=O(|z|^{-2-\beta})\in L^{p}_{-2-\varepsilon}(\mathbb{R}^3\backslash\overline{B_R(0)})$ for all $p\leq \infty$ and some $\varepsilon\in (0,1)$. Then, since $\hat g$ is $W^{2,r}_{-\tau}$-AE by \cref{prop: regularity_g_hat}, a simple adaptation of \cite[Proposition 1]{Maxwell0} to the case of AE manifolds without boundary grants the existence of a solution $v^i\in W^{2,r}_{-\varepsilon}(\hat{M})$ to $\Delta_{\hat{g}}v^i=\Delta_{\hat{g}}z^i$. By the local elliptic regularity of \cref{thm: 12 local elliptic regularity}, we also know that $v^i\in W^{2,q}_{loc}(\hat{M})\cap C^0_{-\varepsilon}(\hat{M})$ with $\Delta_{\hat{g}}v^i\in L^q_{-2-\varepsilon}(\hat{M})$ and thus $v^i\in W^{2,q}_{-\varepsilon}(\hat{M})\hookrightarrow C^1_{-\varepsilon}(\hat{M})$ by \cref{BartniksProp1.6}. Thus, the functions $y^i \doteq z^i-v^i \in W^{2,q}_{loc}(\hat M)$ satisfy
    \begin{equation*}
        \Delta_{\hat g}y^i = 0 \quad \text{in} \;\; \hat M \,, \qquad y^i-z^i\in W^{2,q}_{-\varepsilon}(\hat M) \,.
    \end{equation*}
    In particular, $\frac{\partial y^i}{\partial z^j}=\delta^i_j+O(|z|^{-1-\varepsilon})$ and thus $\{y^i\}_{i=1}^3$ define a set of asymptotic coordinates for $\hat{M}$ which are $C^{1,\alpha}$-compatible with $\{z^i\}_{i=1}^3$ and such that
    \begin{align}
        \label{AEHarmoncCoordExpansion}
        \hat{g}_{ij}(y)=\delta_{ij} + \BO(|z|^{-1-\varepsilon}) \qquad \text{as} \quad |z| \to \infty \,.
    \end{align}
    We would like to further show that $\hat g$ is also $AS(1)$ with respect to the $\{y^i\}_{i=1}^3$ coordinates. To do so, let $\phi : \nR^3 \setminus \Omega \to \nR^3\setminus\Omega'$ be the coordinate change $y=\phi(z)$ and fix $\Omega$ large enough so that $\Vert\phi-\Id\Vert_{C^1_{-\varepsilon}(\nR^3\backslash\Omega)}<\frac{1}{2}$. Then, \cite[Corollary 3.2]{AvalosLaurainMarque} implies that $\phi^{-1}-\mathrm{Id}\in W^{2,q}_{-\varepsilon}(\nR^3\backslash\Omega') \hookrightarrow C^1_{-\varepsilon}(\mathbb{R}^3\backslash\Omega')$ too, which grants that \eqref{AEHarmoncCoordExpansion} can actually be upgraded to
    \begin{align}
        \begin{split}
        \hat{g}_{ij}(y) &= \frac{\partial z^a}{\partial y^i} \frac{\partial z^b}{\partial y^b}\hat{g}_{ab}(z) = \delta_{ij} + \BO(|y|^{-1-\varepsilon}) \qquad \text{as} \quad |y| \to \infty \,.
        \end{split}
    \end{align}
    Moreover, since $\phi : \nR^3\setminus\Omega\to \nR^3\setminus\Omega'$ is a $W^{3,q}_{loc}$-diffeomorphism, \cref{lemma: Adams diffeo lemma} implies that $(\hat g_{ab} \circ z)(y) \in W^{2,q}_{loc}(\nR^3\setminus\Omega')$ and the first equality above identities gives $\hat{g}_{ij}(y)\in W^{2,q}_{loc}(\nR^3\setminus\Omega')$ by Sovolev multiplications. Also, the Jacobian matrices $\frac{\partial z^a}{\partial y^i}-\delta^a_i\in W^{1,q}_{-1-\varepsilon}(\nR^3\setminus\Omega')$ by the previous analysis, which proves that
    \begin{align}
        \label{AEHarmoncCoordExpansion.3}
        h_{ij}(y) \doteq \hat{g}_{ij}(y) -\delta_{ij} \in W^{2,q}_{loc}(\nR^3\setminus\Omega')\cap W^{1,q}_{-1-\varepsilon}(\nR^3\setminus\Omega'),
    \end{align}
    where we have used the multiplication properties --see e.g. \cite[Theorem 2.3]{avalos2024sobolev}-- and that $q>3$. Now, writing $\Ric_{\hat{g}} \equiv 0$ in the harmonic coordinates at infinity, we obtain (see e.g. \cite{DeTurck_Kazdan,Sabitov-Shefel})
    \begin{align}
        \label{RicciEquation}
        \Delta_{\hat g}\bigl(h_{ij}(y)\bigr) = \hat{g}^{pq}(y)\partial_{p}\partial
        _{q}h_{ij}(y) = Q_{ij}\bigl(\hat{g}(y),\partial h(y)\bigr).
    \end{align}
    where $Q_{ij}$ are quadratic in $\partial h$.  
    Noticing that $Q_{ij} \in L^{\frac{q}{2}}_{-4-2\varepsilon}(\nR^3\setminus\Omega')$, \cref{BartniksProp1.6} together with \eqref{AEHarmoncCoordExpansion.3} and \eqref{RicciEquation} implies that $h_{ij} \in W^{2,\frac{q}{2}}_{-1-\varepsilon}(\nR^3\setminus\Omega')$. From here, we start an iteration in order to establish that $h_{ij}\in C^1_{-1-\varepsilon}(\mathbb{R}^3\backslash\Omega')$. If $q>6$, then we are done by the Sobolev embedding. Otherwise, we have $\partial h \in W^{1,\frac{q}{2}}_{-2-\varepsilon}(\nR^3\setminus\Omega')$ and we can start a standard iteration in \eqref{RicciEquation} to obtain a $C^1_{-1-\varepsilon}(\mathbb{R}^3\backslash\Omega')$-estimate for $h_{ij}$, implying that:
    \begin{equation}
        \label{eq: hat g AS(1) harmonic asymptotics}
        \hat g_{ij}(y) = \delta_{ij} + \BO_1\bigl(|y|^{-1-\varepsilon}\bigr)  \qquad \text{as} \quad |y| \to \infty \,,
    \end{equation}
    showing that $\hat g$ is also $AS(1)$ of mass zero with respect to the harmonic coordinates at infinity.

    Finally, since $\{y^i\}_{i=1}^3$ are globally defined harmonic functions for $\hat{g}\in W^{2,q}_{loc}(\hat{M})$, then $y^i\in W^{3,q}_{loc}(\hat{M})$ due to \cref{thm: 12 local elliptic regularity} and we can make sense of the usual Bochner formula, which since $\mathrm{Ric}_{\hat{g}}=0$ and $\Delta_{\hat{g}}y^i=0$, reads
    \begin{align*}
        \Delta_{\hat{g}}|dy^i|_{\hat{g}}^2=|\hat{\nabla}^2y^i|^2_{\hat g}
    \end{align*}
    Integrating over a large bounded domain $D_R\subset\subset\hat{M}$ with boundary $\{|z|=R\}$
    \begin{align}\label{BochnerIntegral1}
        \int_{D_R}|\hat{\nabla}^2y^i|^2d\mu_{\hat{g}}= 2 \int_{\{|z|=R\}}\hat{\nabla}^2y^i(\nu_{\hat{g}},\mathrm{grad}_{\hat{g}}y^i)d\omega_{\hat{g}} \,,
    \end{align}
    where $\nu_{\hat g} = \tfrac{z}{|z|}$ is the outer unit normal. Note that in harmonic coordinates at infinity the Hessian of $y^i$ satisfies
    \begin{align*}
        \hat{\nabla}^2y^i(\partial_{y^k},\partial_{y^l}) = -\hat\Gamma^i_{kl}(y) = \BO(|y|^{-2-\varepsilon}) \qquad \text{as} \quad |y| \to \infty
    \end{align*}
    due to \eqref{eq: hat g AS(1) harmonic asymptotics}, and in view of $|y(z)|=|z| + \BO_1(|z|^{-\varepsilon})$ we find that
    \begin{align*}
        \hat{\nabla}^2 y^i(\partial_{y^k},\partial_{y^l}) = \BO(|z|^{-2-\varepsilon}) \qquad \text{as} \quad |z| \to \infty \,.
    \end{align*}
    From here, it is easy to deduce that the right-hand side of \eqref{BochnerIntegral1} is $o(1)$ as$R \to \infty$ and therefore
    \begin{align}
        \int_{M}|\hat{\nabla}^2y^i|^2d\mu_{\hat{g}}=0 \,,
    \end{align}
    showing that $\{dy^i\}_{i=1}^3$ are parallel 1-forms. Moreover, since they become orthonormal as $|z| \to \infty$, they must be orthonormal everywhere, hence forming a global orthonormal coframe. This implies that $y : (\hat M, \hat g) \to \nR^3$ is a local isometry and in light of $y^i = z^i + \BO_1\bigl(|z|^{-\varepsilon}\bigr)$, it is surjective. A standard argument appealing only to the completeness of $(\hat M, \hat g)$ then shows that it must then be a covering map (see the proof of \cite[Chapter 7, Lemma 3.3]{DoCarmo}), but since $\nR^3$ is simply connected, it is a global isometry.
\end{proof}
\begin{theorem}
    \label{thm: A=0 case}
    Let $M$ be an orientable, smooth, closed 3-manifold and $g$ a $W^{2,q}$-Riemannian metric  with $q > 3$ and $\lambda(M, g) > 0$. Let $(\hat M, \hat g)$ be defined as in \eqref{eq: decompactified metric}. If $ \mathbf{m}(\hat M, \hat g) = 0$, then there exists a $W^{3,q}$-conformal diffeomorphism
    \begin{equation*}
        \phi : (M,g) \to \nS^3 \,.
    \end{equation*}
    In particular, the conformal metric $\phi^*g_{\nS^3} \in [\,g\,]_{W^{2,q}}$ has constant scalar curvature.
\end{theorem}
\begin{proof}
    Let $N,S \in \nS^3$ be the north and south poles of the unit 3-sphere and let
    \begin{equation*}
        \sigma_N : \nS^3\setminus\{S\} \to \nR^3 \quad \text{and} \quad \sigma_S : \nS^3\setminus\{N\} \to \nR^3
    \end{equation*}
    be the corresponding stereographic coordinates 
    around them --see e.g. \cite[Section 3]{Lee-Parker}--. Notice that
    \begin{equation}
        \label{eq: stereographic coords inversion}
        (\sigma_N\circ\sigma_S^{-1})(x) = (\sigma_S\circ\sigma_N^{-1})(x) = \frac{x}{|x|^2}
    \end{equation}
    holds for every $x \in \nR^3$. Let $y : (\hat M, \hat g) \to \nR^3$ be the $W^{3,q}_{loc}$-isometry given in \cref{PropMassRigidity.1} and consider the bijective mapping $\phi : M \to \nS^3$ defined by

    \begin{equation}
        \label{eq: hatM to S3 isometry}
        \phi(q) =
        \begin{dcases}
            (\sigma_S^{-1} \circ y \circ \pi)(q) \,, \quad &q \neq p
            \\
            N \,, &q = p
        \end{dcases}
    \end{equation}
    where $\pi : M \to \hat M$ denotes the natural restriction. Note that both $\rho$ and $\pi$ are smooth maps with respect to the maximal differentiable structure inherited by $\hat M$ compatible with the asymptotic coordinates $z^i = \frac{x^i}{|x|^2}$, where $\{x^i\}_{i=1}^3$ are normal coordinates of $(M,g)$ around $p$ as constructed in \cref{subsec: normal coordinates}. Consequently, $\phi \in W^{3,q}_{loc}(M\setminus\{p\};S^3)$. In order to study the regularity of $\phi$ around $p \in M$, we use the normal coordinates $\{x^i\}_{i=1}^3$ as well as stereographic coordinates $\{\sigma_N^i\}_{i=1}^3$ around $N \in \nS^3$. Using that \eqref{eq: stereographic coords inversion} and \eqref{eq: hatM R3 isometry expansion} we get
    \begin{align*}
        \sigma_N(\phi(x)) = (\sigma_N\circ\sigma_S^{-1})(y(x)) = \frac{y(x)}{|y(x)|^2} = \frac{z(x) + \BO_1\bigl(|z(x)|^{-\varepsilon}\bigr)}{|z(x)|^2+\BO_1\bigl(|z(x)|^{1-\varepsilon}\bigr)} = x + \BO_1\bigl(|x|^{1+\varepsilon}\bigr)
    \end{align*}
    for $x \neq 0$, showing that $\phi$ is $C^{1,\alpha}$ around $p$. In order to see that it is in fact $W^{3,q}_{loc}$, following the ideas of \cite[Theorem 2.1]{Taylor_ConfFlat}, consider now a harmonic chart $(U, u^i)$ around $N \in \nS^3$, which in this case is smooth. That is, it belongs to the canonical maximal differentiable structure of $\nS^3$. Then, using that $\phi$ is an isometry, one can check that the maps $u^i\circ\phi : M \supset \phi^{-1}(U) \to \mathbb{R}^3$ are $C^{1,\alpha}$ weak solutions to\footnote{For the associated computations, see \cite[Equations (2.6) to (2.8)]{Taylor_ConfFlat} .}
    \begin{align*}
        \Delta_g(u^i\circ\phi) = 0 \,.
    \end{align*}
    Since $g\in W^{2,q}(M)$, it follows from the same analysis as in the proof of \cref{prop: Existence and Regularity of harmonic coordinates} that $u\circ\phi \in W^{3,q}_{loc}(\phi^{-1}(U))$ and therefore, the same regularity claim holds for any other coordinate system smoothly compatible to $u$. This establishes $\phi\in W^{3,q}(M;S^3)$.

    Finally, recalling that $(\sigma_S)_*g_{\nS^3} = 4u_1^{-4}g_{\nR^3}$ where $u_a \in C^\infty(\nR^n)$ is an Aubin bubble (see \eqref{eq: AubinBubbles} below or \cite[Section 3]{Lee-Parker}), we obtain
    \begin{align*}
        \phi^*g_{\nS^3} &= 4 \, y^*\bigl(u_1^{-4}g_{\nR^3}\bigr) = 4 \, (u_1 \circ y)^{-4} \hat g = 4 \, (u_1 \circ y)^{-4} \G^{-4} g \,,
    \end{align*}
    showing that $\phi$ is indeed a conformal diffeomorphism. Since $\phi\in W^{3,q}(M;S^3)$, tracing both sides by $g \in W^{2,q}(T_2M)$ we see that the conformal factor on the right-hand-side must be in $W^{2,q}(M)$, showing that $\phi^*g_{\nS^3} \in [\,g\,]_{W^{2,q}}$, as desired.
\end{proof}

\subsection{3-manifolds of positive Yamabe invariant}
\label{SectionProofMainThm}

We are now ready to solve the Yamabe problem for closed 3-manifolds by constructing a test function $\psi \in W^{1,2}_{-1/2}(\hat M)$ such that $Q_{\hat g}(\psi) < \lambda(\nS^3)$ and appealing to \cref{thm: Aubin-Trudinger-Yamabe} and \cref{prop: compac-noncompact yamabe correspondence}. For this, we follow the strategy adopted in the case of smooth metrics as presented in \cite[Section 7]{Lee-Parker}. The basic intuition is that, since $\hat{g}$ is close to the Euclidean metric near infinity, one may use so called \textit{Aubin bubbles}
\begin{equation}
    \label{eq: AubinBubbles}
    u_a(z) \doteq \left(\frac{|z|^2 + a^2}{a}\right)^{-\frac{1}{2}} \qquad \text{with} \qquad a > 0 
\end{equation}
to construct the desired test function. These are smooth solutions to
\begin{equation}
    \label{AubinBubble.1}
    -\Delta u_a = 3\,u_a^{2^*-1} \qquad \text{in} \;\; \nR^3 \,,
\end{equation}
which therefore saturate the optimal Sobolev inequality in $\nR^3$ and consequently satisfy --see \cite[Theorem 3.3]{Lee-Parker} or \cite{Talenti}--
\begin{align}
    \label{AubinBubble.2}
    24\|u_a\|^{2^{*}-2}_{L^{2^{*}}(\nR^3)} = 8\frac{\|\nabla u_a\|^2_{L^2(\nR^3)}}{\|u_a\|^{2^{*}}_{L^{2^{*}}(\nR^3)}} = \lambda(\nS^3) \,.
\end{align}
Observe that for large values of the parameter $a$, the functions $u_a$ become approximately constant in a ball $B_R(0)$ with $a>>R$. Thus, we fix some large radius $R > 0$ and let $\hat M_\infty \doteq \{|z| > R\}\subset \hat{M}$. Below, we shall use the notation $\rho(z) = |z|$ for the asymptotic radial coordinate. We define
    \begin{equation}
        \label{eq: test function}
        \psi_a(z) \doteq \begin{cases}
            u_a(z) \,, \quad \rho(z) \geq R
            \\
            u_a(R) \,, \quad \rho(z) \leq R
        \end{cases}
    \end{equation}
with $a >> R$ to be chosen latter. Clearly, $\psi_a \in W^{1,2}_{-1/2}(\hat M)$ and it is thus a candidate test function. Before going to the main proof, let us show the following, based on \cite[Lemma 3.5]{Lee-Parker}.
\begin{lemma}
    \label{PropIntegralEstimateYamabe}
    Let $a>R>0$ and $-1<k<3$ be real numbers. Then, there exists a constant $C=C(k)>0$ such that
    \begin{align}
        \label{IntegralEstimateYamabe}
        \int_{R}^{\infty}\rho^{-k}\frac{a}{(a^2+\rho^2)^2} \rho^{2}d\rho\leq Ca^{-k} 
    \end{align} 
\end{lemma}
\begin{proof}
    First, let us rewrite introduce the change of variables $\sigma\doteq \frac{\rho}{a}$ so that
    \begin{align*}
        \int_{R}^{\infty}\rho^{2-k}\frac{a}{(a^2+\rho^2)^2}d\rho
        &= a^{-k}\left(\int_{\frac{R}{a}}^1\sigma^{2-k}(1+\sigma^2)^{-2}d\sigma + \int_{1}^{\infty}\sigma^{2-k}(1+\sigma^2)^{-2}d\sigma\right) \,.
    \end{align*}
    To estimate the second integral, we notice that $(1+\sigma^2)^{-2} \leq \sigma^{-4}$ for all $\sigma\in [1,\infty)$, so using that $k>-1$, we have
    \begin{align*}
        \int_{1}^{\infty}\sigma^{2-k}(1+\sigma^2)^{-2}d\sigma\leq \int_{1}^{\infty}\sigma^{-2-k}d\sigma = \frac{1}{1+k}.
    \end{align*}
    For the first integral, we now notice that if $\frac{R}{a} \leq \sigma \leq 1$, then $1+\left(\frac{R}{a}\right)^2 \leq 1+\sigma^2\leq 2$, and thus $(1+\sigma^2)^{-2}\leq \left( 1+\left(\frac{R}{a}\right)^2 \right)^{-2}$. Then, using that $k<3$ and $a > R$,
    \begin{align*}
        \int_{\frac{R}{a}}^1\sigma^{2-k}(1+\sigma^2)^{-2}&d\sigma
        &= \left( 1+\left(\frac{R}{a}\right)^2\right)^{-2}\frac{1}{3-k}\left(1 - \left(\frac{R}{a}\right)^{3-k} \right) \leq \frac{1}{3-k}
    \end{align*}
    and the assertion follows.
\end{proof}
We are now ready to establish the main result of the section. We stress that the orientability hypothesis below is only required to apply \cref{PropPMTforGp} and were \cref{thm: LeeLF} to hold true, either without the spin assumption or allowing for multiple ends, one could immediately drop it.
\begin{theorem}
    \label{thm: Yamabe Problem 3d}
    Let $M$ be an orientable, smooth, closed 3-manifold and $g$ a $W^{2,q}$-Riemannian metric  with $q > 3$ and $\lambda(M, g) > 0$. Then, there exists a Riemannian metric in the conformal class $[\,g\,]_{W^{2,q}}$ of positive constant scalar curvature.
\end{theorem}
\begin{proof}
    Let $\bigl(\hat M, \hat g\bigr)$ be the AE manifold constructed in \cref{Decompactification via Conformal Green's function}. In view of \cref{thm: Aubin-Trudinger-Yamabe} and Proposition \ref{prop: compac-noncompact yamabe correspondence}, all we need to do is to verify that the test function $\psi_a \in W^{1,2}_{-1/2}(\hat M)$ defined in \eqref{eq: test function} satisfies $Q_{\hat g}(\psi_a) < \lambda(\nS^3)$ for some $a > R$ large enough. Using that $\R_{\hat g} \equiv 0$ and $\psi_a$ is constant in $\hat M \setminus \hat M_\infty$, we get
    \begin{align*}
        E_{\hat g}(\psi_a) = 8\int_{\hat M} |\nabla\psi_a|^2_{\hat g} \, d\mu_{\hat g} = 8\int_{\hat M_{\infty}} |\nabla u_a|^2_{\hat g} \, d\mu_{\hat g} \,.
    \end{align*}
    Following \cite[Proposition 7.1]{Lee-Parker}, 
    \begin{align}
        \label{YamabeEnergyAubinBubble}
            E_{\hat g}(\psi_a)
            &= 8\lim_{L \to \infty} \int_{\{R < \rho < L\}} \hat g^{\rho\rho} \partial_\rho u_a \partial_\rho u_a \, \sqrt{\det\hat g} \, \rho^{2} d\omega d\rho \nonumber
            \\
            &= -8 \lim_{L \to \infty}\int_{\{R < \rho < L\}} u_a \partial_\rho\bigl(\rho^{2} \partial_\rho u_a\bigr) \, \hat g^{\rho\rho} \sqrt{\det\hat g}\, d\omega d\rho \nonumber
            \\
            &\quad - 8 \lim_{L \to \infty} \int_{\{R < \rho < L\}} u_a \partial_\rho u_a \partial_\rho \bigl(\hat g^{\rho\rho}\sqrt{\det\hat g}\bigr) \, \rho^{2} d\omega d\rho
            \\
            &\quad + 8 \lim_{L \to \infty} \int_{\partial B_R \cup \partial B_L} u_a \partial_\rho u_a \, \hat g^{\rho\rho} \sqrt{\det\hat g}\; \rho^{2} d\omega \,. \nonumber
    \end{align}
    Let us start computing the first integral. Exploiting that $u_a$ is radial, the Euclidean Laplacian reads $\Delta u_a = \rho^{-2}\partial_\rho(\rho^{2}\partial_\rho u_a)$, so we get
    \begin{align*}
        -8\int_{\{R < \rho < L\}} &u_a \partial_\rho\bigl(\rho^{2} \partial_\rho u_a\bigr) \, \hat g^{\rho\rho} \sqrt{\det\hat g}\, d\omega d\rho
        \\
        &= -8\int_{\{R < \rho < L\}} u_a \Delta u_a \, \hat g^{\rho\rho} \sqrt{\det\hat g}\, dz
        \\
        &=24\int_{\{R < \rho < L\}} u_a^{2^{*}} \hat{g}^{\rho\rho} \sqrt{\det\hat g}\, dz
        \\
        &=24\left(\int_{\{R < \rho < L\}} u_a^{2^{*}} dz\right)^{\frac{2^{*}-2}{2^{*}}}  \left(\int_{\{R < \rho < L\}} u_a^{2^{*}} \Bigl(\hat{g}^{\rho\rho} \sqrt{\det\hat g}\Bigr)^{\frac{2^{*}}{2}} dz\right)^{\frac{2}{2^{*}}}
        \\
        &\leq \lambda(\nS^3) \left(\int_{\{R < \rho < L\}} u_a^{2^{*}} \Bigl(\hat{g}^{\rho\rho} \sqrt{\det\hat g}\Bigr)^{\frac{2^{*}}{2}} dz\right)^{\frac{2}{2^{*}}} \,,
    \end{align*}
    where the second equality is due to \eqref{AubinBubble.1} and the fourth inequality due to \eqref{AubinBubble.2}. Since $\hat g \in AS_\beta(1)$ by \cref{prop: regularity_g_hat}, it follows that
    \begin{align}
        \label{DetExpansion}
        \sqrt{\det\hat{g}} = 1 + \frac{6A}{\rho} + \BO_1(\rho^{-1-\beta}) \qquad \text{and} \qquad \hat{g}^{\rho\rho} = 1 - \frac{4A}{\rho} + \BO_1(\rho^{-1-\beta})
    \end{align}
    as $\rho \to \infty$. This implies, in particular, that
    \begin{align*}
        \hat{g}^{\rho\rho} \sqrt{\det\hat{g}} = (\det\hat{g})^{\frac{1}{6}} + \BO_1(\rho^{-1-\beta}) \,,
    \end{align*}
    and noting that in 3 dimensions $2^{*}=6$, we estimate the first term in \eqref{YamabeEnergyAubinBubble} by
    \begin{align}
        \label{YamabeQuotientIntegral1}
        \lambda(\nS^3)\left(\int_{\hat M_\infty} u_a^{2^{*}} d\mu_{\hat{g}} + \int_{\hat M_\infty} u_a^{2^{*}}O(\rho^{-1-\beta}) dz \right)^{\frac{2}{2^{*}}} \,.
    \end{align}
    Now, set $C_1=C_1(R)$ to be a constant such that $|\BO(\rho^{-1-\beta})|\leq C_1\rho^{-1-\beta}$ and thus
    \begin{align*}
        \Bigl| \int_{\hat M_\infty} &u_a^{2^{*}} \BO(\rho^{-1-\beta}) dz\Bigr| \leq C_1\int_{\hat M_\infty} u_a^{2^{*}}\rho^{-1-\beta} dz 
        \\
        &= 4\pi C_1\int_R^\infty \rho^{1-\beta}\frac{a^3}{(a^2+\rho^2)^3}d\rho = 4\pi C_1 \int_R^L\rho^{1-\beta}\frac{a^2}{(a^2+\rho^2)}\frac{a}{(a^2+\rho^2)^2}d\rho 
        \\
        &\leq 4\pi C_1 \int_R^{\infty}\rho^{-(1+\beta)}\frac{a}{(a^2+\rho^2)^2}\rho^2 d\rho
    \end{align*}
    Observe that $-1 < 1+\beta < 3$ due to $\beta\in (0,1)$, so \cref{PropIntegralEstimateYamabe} implies that
    \begin{align*}
         \left| \int_{\hat M_\infty} u_a^{2^{*}} \BO(\rho^{-1-\beta}) dz\right| \leq C(R,\beta) a^{-1-\beta} \,.
    \end{align*}
    We conclude that the first term in \eqref{YamabeEnergyAubinBubble} can be estimated by
    \begin{align}
        \label{eq: I_1 estimate}
        \begin{split}
            I_1 &\leq \lambda(\nS^3)\Bigl(\|\psi_a\|^{2^{*}}_{L^{2^{*}}(\hat{M},d\mu_{\hat{g}})} + C(R,\beta)a^{-1-\beta} \Bigr)^{\frac{2}{2^{*}}}
            \\
            &\leq \lambda(\nS^3)\|\psi_a\|^{2}_{L^{2^{*}}(\hat{M},d\mu_{\hat{g}})} + \BO(a^{-1-\beta})
        \end{split}
    \end{align}
    as $a \to \infty$.

    Let us go back to \eqref{YamabeEnergyAubinBubble} and estimate the second term on the right-hand side. Appealing again to \eqref{DetExpansion}, we have that
    \begin{equation*}
        \partial_\rho\left(\hat g^{\rho\rho} \sqrt{\det\hat g}\right) = -\frac{2A}{\rho^2} + \BO\bigl(\rho^{-2-\beta}\bigr) \qquad \text{as} \; \rho \to \infty \,.
    \end{equation*}
    Then, using the co-area formula, we get
    \begin{align*}
        -\int_{\hat M_\infty} u_a \partial_\rho u_a \partial_\rho \bigl(\hat g^{\rho\rho}\sqrt{\det\hat g}\bigr) \, &\rho^{2} d\omega d\rho = -\int_R^\infty u_a \partial_\rho u_a\int_{\partial B_\rho} \partial_\rho \bigl(\hat g^{\rho\rho}\sqrt{\det\hat g}\bigr) \rho^2 d\rho \, d\omega
        \\
        &= -4\pi\int_R^\infty u_a \partial_\rho u_a \left(-\frac{2A}{\rho^2} + \BO\bigl(\rho^{-2-\beta}\bigr)\right) \rho^2 d\rho
        \\
        &= 4\pi  \int_R^\infty \rho\frac{a}{(\rho^2 + a^2)^2} \bigl(-2A + \BO(\rho^{-\beta})\bigr) d\rho \,.
    \end{align*}
    The first integral can be computed explicitly via the coordinate change $\sigma = \tfrac{\rho}{a}$ as
    \begin{align*}
        -8\pi A\int_R^\infty \frac{a}{(\rho^2 + a^2)^2} \rho \, d\rho = \frac{-8\pi A}{a}\int_{R/a}^\infty \frac{\sigma}{\bigl(1 + \sigma^2)^{2}} d\sigma = \frac{-4\pi a}{a^2 + R^2} \, A \,,
    \end{align*}
    whereas the second one may be estimated by
    \begin{align*}
        4\pi \int_R^\infty \rho\frac{a}{(\rho^2 + a^2)^2} \BO(\rho^{-\beta}) d\rho \leq 4\pi C_1 \int_R^\infty \rho^{-(1+\beta)}\frac{a}{(\rho^2 + a^2)^2} \rho^2 d\rho \leq C(R,\beta) a^{-1-\beta}
    \end{align*}
    again using \cref{PropIntegralEstimateYamabe}. We conclude that the second term in \eqref{YamabeEnergyAubinBubble} can be estimated by
    \begin{align}
        \label{YamabeQuotientIntegral2.3}
        I_2 \leq -C(R)A a^{-1} + \BO\bigl(a^{-1-\beta}\bigr) \qquad \text{as} \quad a \to \infty \,.
    \end{align}

    For the last term in the right-hand-side of \eqref{YamabeEnergyAubinBubble}, observe first that the $\partial B_L$ boundary term vanishes as $L \to \infty$, as $u_a\partial_\rho u_a = \BO(\rho^{-3})$ as $\rho \to \infty$ for each fixed $a$. The interior boundary instead can be estimated, using that $u_a\partial_\rho u_a = \BO(a^{-3})$ as $a \to \infty$ for fixed $\rho$, by
    \begin{align}
        \label{eq: estimate I_3}
        I_3 \leq 8 \sup_{\hat{M}_{\infty}}\left|\hat g^{\rho\rho}\sqrt{\det\hat g}\right| \int_{\partial B_R} u_a \partial_\rho u_a \, \rho^2 d\omega = \BO\bigl(a^{-3}\bigr)
    \end{align}
    as $a \to \infty$. Hence, coming back to \ref{YamabeEnergyAubinBubble} and combining \eqref{eq: I_1 estimate}, \eqref{YamabeQuotientIntegral2.3} and \eqref{eq: estimate I_3}, we establish
    \begin{align}
        \label{YamabeEnergyAubinBubbleFinal}
        E_{\hat g}(\psi_a) \leq  \Lambda \|\psi_a\|_{L^{2^*}(\hat M)} - C(R) A a^{-1} + \BO\bigl(a^{-1-\beta}\bigr) \qquad \text{as} \quad a \to \infty
    \end{align}
    In light of \cref{PropPMTforGp} and  \cref{thm: A=0 case}, we know that $A>0$ unless $(M,g)$ is $W^{3,q}$-conformal to $\nS^3$. In latter case we are done, while if $A>0$, the inequality \eqref{YamabeEnergyAubinBubbleFinal} implies that $Q_{\hat g}(\psi_a) < \Lambda$ for $a$ large enough. In this case, \cref{thm: Aubin-Trudinger-Yamabe} grants the assertion.
\end{proof}
Using the regularity theory developed in \cref{Analytical Results}, we finally demonstrate the main result of the work \cref{theorem A}:
\begin{proof}[Proof of \cref{theorem A}]
    We prove the Theorem by induction in $k$. The case $k = 2$ is proven in \cref{thm: Yamabe Problem 3d}, so assume that the statement holds for some integer $k \geq 2$. In particular this implies the existence of a conformal factor $u \in W^{k,q}(M)$ satisfying
    \begin{equation*}
        -8\Delta_g u = \Lg u - \R_g u = \lambda u^{2^*-1} - \R_g u \,.
    \end{equation*}
    Now, if $g \in W^{k+1,q}(T_2M)$, then $\R_g \in W^{k-1,q}(M)$ and consequently $\Delta_g u \in W^{k-1,q}(M)$. But then \cref{thm: main global elliptic regularity} implies that $u \in W^{k+1,q}(M)$ and the conformal metric $\bar g = u^{2^*-2} g$ is of class $W^{k+1,q}(T_2M)$.
\end{proof}

\begin{appendix}

\vspace{1cm}
\section{Sobolev diffeomorphisms}\label{IFTSobolevAppendix}

In this appendix we extend some properties of $C^k$ diffeomorphisms to diffeomorphisms of Sobolev class. We say that a map $\phi : \Omega \to \Omega'$ between two open sets of $\nR^n$ is a $W^{k,q}$ diffeomorphism if it is a $C^1$ diffeomorphism with $\phi \in W^{k,q}_{loc}(\Omega)$. Due to the inverse function theorem, the inverse of an invertible $C^k$-map is $C^k_{loc}$. The analogue statement for Sobolev is presented below and follows along the lines of \cite[Proposition 4.2]{avalos2024sobolev}, which corresponds to the case $k=3$.
\begin{theorem}\label{LemmaInvFunctThmSobReg}
    Let $\phi : \Omega \to \Omega'$ be a $C^1$-diffeomorphism between bounded domains of $\nR^n$ with smooth boundary. Suppose that $\phi \in W^{k,q}(\Omega,\Omega')$ for $k \geq 3$ and $q > \tfrac{n}{2}$. Then, $\phi^{-1} \in W^{k,q}_{loc}(\Omega',\Omega)$. 
\end{theorem}
\begin{proof}
    Name $\varphi \coloneq \phi^{-1}$, fix cartesian coordinates $(x_1, \ldots, x_n)$ in $\Omega'$ and let $y = \varphi(x)$. By Sobolev embedding, $\phi \in C^{k-2}_{loc}(\Omega)$ and the inverse function theorem implies that $\varphi \in C^{k-2}_{loc}(\Omega')$ with
    \begin{equation}\label{eq: inverse deriv}
        D\varphi = \left[D \phi\circ \varphi\right]^{-1} \,.
    \end{equation}
where $[\:-\:]^{-1}$ denotes the inverse matrix. Consider now $V'_y\subset\subset \Omega'$ a neighbourhood of $y$ and $V_y=\varphi(V'_y)\subset \Omega$, so that $\varphi:V'_y\to V_y$ is a $C^{k-2}$-bounded map. Since $D\phi\in W^{k-1,q}(V_y)\hookrightarrow W^{k-2,p}(V_y)$ for some $p\geq q$ with $p>n$, then $\frac{\partial\phi_i}{\partial y_a}\circ \varphi\in W^{k-2,p}(V'_y)$ by  \cite[Theorem 3.41]{AdamsFournierBook}. Then, since the map $\GL(n,\nR) \to \GL(n,\nR)$ sending a matrix to its inverse is smooth, by \eqref{eq: inverse deriv} it follows that $D\varphi\in W^{k-2,p}(V'_y)$. Since $k\geq 3$, we may take an extra weak derivative of \eqref{eq: inverse deriv}, obtaining
\begin{align*}
\partial_{x^ix^j}\varphi^a(y)= \partial_{x^i}\varphi^l(y)\partial_{x^l}\left(\left[D\phi \right]^{-1}\right)^a_j(\varphi(y)).
\end{align*}
Noticing that $[D\phi]^{-1}:V_y\to  \GL(n,\nR)$ is a $W^{k-1,q}(V_y)$, which follows since $W^{k-1,q}(V_y)$ is an algebra under multiplication for $k-1>n/q$ and an appeal to \cite[Lemma 2.1]{avalos2024sobolev}. Thus, $\partial_{x^l}\left(\left[D\phi \right]^{-1}\right)^a_j\circ \varphi\in W^{k-2,q}(V'_y)$ again by \cite[Theorem 3.41]{AdamsFournierBook} and we find
\begin{align*}
\partial_{x^ix^j}\varphi^a\in W^{k-2,p}_{loc}(V'_y) \otimes W^{k-2,q}_{loc}(V'_y)\hookrightarrow W^{k-2,q}_{loc}(V'_y)
\end{align*}
where the last implication follows from a standard Sobolev multiplication and the conditions $p\geq q$ and $p>n$. This implies that $\varphi\in W^{k,q}_{loc}(\Omega')$.
\end{proof}

We are now concerned with the regularity of functions under pull-back by diffeomorphisms of Sobolev class. From \cite[Theorem 3.41]{AdamsFournierBook}, it is known that a $C^k$ diffeomorphism $\phi : \Omega \to \Omega'$ induces a continuous map between the space of functions $W^{k,p}_{loc}(\Omega)$ into $W^{k,p}_{loc}(\Omega')$ for any $1 \leq p \leq \infty$. In our case, we prove the following:
\begin{lemma}
    \label{lemma: Adams diffeo lemma}
    Fix $k \geq 2$, $q > \tfrac{n}{2}$ and let $\phi : \Omega \to \Omega'$ be a $W^{k+1,q}$-diffeomorphism between two bounded smooth domains in $\nR^n$. If $u \in W^{l,p}(\Omega)$, then $Au\doteq u \circ \phi^{-1} \in W_{loc}^{l,p}(\Omega')$ provided that $l \leq k+1$ and $1\leq p \leq q$. Moreover, \begin{align}
        \label{ChangeOfCoordsMapSob.1}
        \begin{split}
            A:W_{loc}^{l,p}(\Omega)&\to W_{loc}^{l,p}(\Omega')
        \end{split}
    \end{align}    
    is bounded with bounded inverse, meaning that, given any smooth domain $K \subset\subset \Omega$, there exists a constant $C =C(n,p,q,l,K,\phi)$ such that
    \begin{align}
    \label{ChangeOfCoordsMapSob.2}
        \begin{split}
            \Vert Au \Vert_{W^{l,p}(\phi(K))}&\leq C\Vert u \Vert_{W^{l,p}(K)} \,,
            \\
            \Vert A^{-1}v \Vert_{W^{l,p}(K)}&\leq C\Vert v \Vert_{W^{l,p}(K)}
        \end{split}
    \end{align}
    hold for all $u \in W^{l,p}_{loc}(\Omega)$ and $v \in W^{l,p}_{loc}(\Omega')$.
\end{lemma}
\begin{proof}
    We name $K' = \phi(K) \subset\subset \Omega'$ for simplicity. By \cref{LemmaInvFunctThmSobReg} we know that $\phi^{-1} \in W^{k+1,q}_{loc}(\Omega')$ and therefore $\phi^{-1} \in W^{k+1,q}(K')$. The Sobolev embeddings 
    \begin{align*}
        W^{k+1,q}(K) \hookrightarrow C^{k-1}(K) \,, \qquad W^{k+1,q}(K') \hookrightarrow C^{k-1}(K')
    \end{align*}
    imply that $\phi: K \to K'$ is a $C^{k-1}$-diffeomorphism, so it follows from \cite[Theorem 3.41]{AdamsFournierBook} that $Au \in W^{l,p}(K')$ for all $u \in W^{l,p}_{loc}(\Omega)$ and the continuity estimates \eqref{ChangeOfCoordsMapSob.2} for $A$ and $A^{-1}$ hold for each $l \leq k-1$.

    Now, if $l = k$ or $l = k+1$, we need to show that $\phi$ preserves, respectively, one or two more levels of Sobolev regularity. We assume that $k = 2$, as the $k > 2$ case follows from the same computations applied to the $(k-2)$-th derivatives. Let $x = (x^1, \ldots, x^n)$ be a point in $K$ and denote by $y = (\phi^1(x), \ldots, \phi^n(x))$ its image in $K'$. The chain rule gives
    \begin{align}
        \label{eq: chain rule 2}
        \begin{split}
        \frac{\partial^2(u \circ \phi^{-1})}{\partial y^i \partial y^j}(y) = &\frac{\partial^2 u}{\partial x^a \partial x^b}(\phi^{-1}(y)) \frac{\partial (\phi^{-1})^a}{\partial y^i}(y) \frac{\partial (\phi^{-1})^b}{\partial y^j}(y)
        \\
        & + \frac{\partial u}{\partial x^a}(\phi^{-1}(y)) 
        \frac{\partial^2(\phi^{-1})^a}{\partial y^i \partial y^j}(y) \,,
        \end{split}
    \end{align}
    where $(\phi^{-1})^a\doteq x^a\circ\phi^{-1}$. Since $q > \tfrac{n}{2}$, $\phi$ is a $C^1$-diffeomorphism and by \cite[Theorem 3.41]{AdamsFournierBook} one has 
    \begin{equation}
        \label{cond: Lp du}
        \frac{\partial^2 u}{\partial x^a \partial x^b} \circ \phi^{-1} \in L^p(K') \qquad \text{and} \qquad \frac{\partial u}{\partial x^a} \circ \phi^{-1} \in W^{1,p}(K') \,,
    \end{equation}
    with estimates
    \begin{align*}
        \Big\Vert \frac{\partial^2 u}{\partial x^a \partial x^b} \circ \phi^{-1}\Big\Vert_{L^p(K')}&\leq  \Vert \phi\Vert^{\frac{1}{p}}_{C^1(K)}\Big\Vert \frac{\partial^2 u}{\partial x^a \partial x^b} \Big\Vert_{L^p(K)} \,,
        \\
        \Big\Vert \frac{\partial u}{\partial x^a } \circ \phi^{-1}\Big\Vert_{W^{1,p}(K')}&\leq  C\bigl(n,p,K,K',\Vert \phi\Vert_{C^1(K)},\Vert \phi^{-1}\Vert_{C^1(K')}\bigr)\:\Big\Vert \frac{\partial u}{\partial x^a} \Big\Vert_{W^{1,p}(K)} \,.
    \end{align*}
    In view of \eqref{eq: chain rule 2}, since by \cref{LemmaInvFunctThmSobReg} we know that $\phi^{-1}\in W^{3,q}(K')$, we conclude that $u \circ \phi^{-1} \in W^{2,p}(K')$ provided that the Sobolev multiplications
    \begin{align}
        \label{ChangeCoordMult1}
        \begin{split}
        L^p(K') \otimes W^{2,q}(K') \otimes W^{2,q}(K') \hookrightarrow L^p(K') \,, 
        \\
        \quad W^{1,p}(K') \otimes W^{1,q}(K') \hookrightarrow L^p(K')
        \end{split}
    \end{align}
hold. Appealing to \cite[Theorem 2.2]{avalos2024sobolev}, the first of these embeddings is always true for $q > \tfrac{n}{2}$, whereas the second one holds for all $\tfrac{1}{p} > \tfrac{1}{q} - \tfrac{1}{n}$ and hence in particular if $p \leq q$. Therefore we find that $Au = u\circ\phi^{-1}\in W^{2,p}(K')$ with an estimate of the form
    \begin{align}
        \label{ChangeCoordEstimate2}
        \Vert Au\Vert_{W^{2,p}(K')}\leq C\bigl(n,p,K,K',\Vert \phi\Vert_{C^1(K)},\Vert \phi^{-1}\Vert_{C^1(K')}\bigr)\:\Vert u\Vert_{W^{2,p}(K)},
    \end{align}
    and thus for $l = k = 2$, we are done. Let us now proceed with the case $l = k +1 = 3$, for which we compute 
    \begin{align}
        \label{eq: chain rule 3}
        \frac{\partial^3(u \circ \phi^{-1})}{\partial y^i \partial y^j \partial y^k} (y) &= \frac{\partial^3 u}{\partial x^a \partial x^b \partial x^c}(\phi^{-1}(y)) \, \frac{\partial (\phi^{-1})^a}{\partial y^i}(y) \frac{\partial (\phi^{-1})^b}{\partial y^j}(y)  \frac{\partial (\phi^{-1})^c}{\partial y^k}(y) \nonumber
        \\
        &\quad + 3\frac{\partial^2 u}{\partial x^a \partial x^b}(\phi^{-1}(y)) 
        \frac{\partial^2 (\phi^{-1})^a}{\partial y^i \partial y^j}(y) \frac{\partial (\phi^{-1})^b}{\partial y^k}(y) 
        \\
        \nonumber
        &\quad + \frac{\partial u}{\partial x^a}(\phi^{-1}(y)) 
        \frac{\partial^3(\phi^{-1})^a}{\partial y^i \partial y^j \partial y^k}(y) \,.
    \end{align}
    As a consequence of the case $l = k = 2$ case, we now know that $\phi : K \to K'$ preserves $W^{2,p}$-regularity. In combination with the hypothesis $u \in W^{3, p}(K)$, we can upgrade the conditions \eqref{cond: Lp du} to
    \begin{align*}
        \frac{\partial^3 u}{\partial x^a \partial x^b \partial x^b} \circ \phi^{-1} \in L^p(K')\,, \quad \frac{\partial^2 u}{\partial x^a \partial x^b} \circ \phi^{-1} \in W^{1,p}(K')\,, \quad \frac{\partial u}{\partial x^a} \circ \phi^{-1} \in W^{2,p}(K') \,.
    \end{align*}

    \noindent
    Again, we conclude that $Au = u \circ \phi^{-1} \in W^{3,p}(K')$ provided that the following Sobolev 

    \noindent
    multiplications hold
    \begin{align*}
        &L^p(K') \otimes W^{2,q}(K') \otimes W^{2,q}(K')\otimes W^{2,q}(K') \hookrightarrow L^p(K'), 
        \\
        & W^{1,p}(K') \otimes W^{1,q}(K')\otimes W^{2,q}(K') \hookrightarrow L^p(K'), 
        \\
        & W^{2,p}(K') \otimes L^{q}(K') \hookrightarrow L^p(K').
    \end{align*}
    The first one holds by the same arguments as those granting the first embedding in \eqref{ChangeCoordMult1}. The second one also follows by the same arguments as the second one in \eqref{ChangeCoordMult1} since $W^{1,q}(K')\otimes W^{2,q}(K') \hookrightarrow W^{1,q}(K')$. Finally, the last one follows again by \cite[Theorem 2.2]{avalos2024sobolev} as long as $p\leq q$. Moreover, going back to \eqref{eq: chain rule 3}, using \eqref{ChangeCoordEstimate2} and the continuity of the above multiplication properties, we also obtain the estimate
    \begin{align}
        \label{ChangeCoordEstimate3}
        \Vert Au\Vert_{W^{3,p}(K')}\leq C\bigl(n,p,K,K',\Vert \phi\Vert_{C^1(K)},\Vert \phi^{-1}\Vert_{W^{3,q}(K')}\bigr)\Vert u\Vert_{W^{3,p}(K)} \,.
    \end{align}
    Since $K$ (and thus $K'$) was arbitrary, we conclude in any case that $Au \in W^{l,p}_{loc}(\Omega')$. Additionally, we may interchange the roles of $\phi$ and $\phi^{-1}$ to obtain the second estimate in \eqref{ChangeOfCoordsMapSob.2}.
\end{proof}
A similar statement to \cref{lemma: Adams diffeo lemma} holds true for tensor fields. The key difference is that coordinate transformations of tensors involve the Jacobian of the coordinate change, and therefore, one level less of regularity is preserved. We denote by $T_s\Omega$ the bundle of $s$-covariant tensor fields over $\Omega$ and by $W^{l,p}(T_s\Omega)$ the sections of $T_s\Omega$ with coefficients in $W^{l,p}(\Omega,\nR^s)$.
\begin{lemma}\label{prop: tensor_regularity}
    Fix $k \geq 2$, $q > \tfrac{n}{2}$ and let $\phi : \Omega \to \Omega'$ be a $W^{k+1,q}$-diffeomorphism between two bounded smooth domains in $\nR^n$.
    If $u \in W^{l,p}(T_s\Omega)$, then $Au \doteq \phi_*u \in W^{l,p}_{loc}(T_s\Omega')$ provided that $l \leq k$ and $p \leq q$. Moreover, 
    \begin{align}
        A:W_{loc}^{l,p}(T_s\Omega)&\to W_{loc}^{l,p}(T_s\Omega')
    \end{align}    
    is bounded with bounded inverse, meaning that, given any smooth domain $K \subset\subset \Omega$, there exists a constant $C =C(n,p,q,l,K,\phi)$ such that
    \begin{align*}
        \Vert Au \Vert_{W^{l,p}(\phi(K))}&\leq C\Vert u \Vert_{W^{l,p}(K)} \,,
        \\
        \Vert A^{-1}v \Vert_{W^{l,p}(K)}&\leq C\Vert v \Vert_{W^{l,p}(K')}
    \end{align*}
    hold for all $u \in W^{l,p}_{loc}(T_s\Omega)$ and $v \in W^{l,p}_{loc}(T_s\Omega')$.
\end{lemma}
\begin{proof}
    We carry out the proof for $s =2$, but other cases are analogous. For each point $x = (x^1, \ldots, x^n)$ in $K$, the tensor field $u$ may be expressed by
    \begin{align*}
        u_x = u_{ij}(x) \, dx^i \otimes dx^j \qquad \text{with} \qquad u_{ij}(x)=u_x\bigl(\partial_{x^i},\partial_{x^j}\bigr) \,.
    \end{align*}
    Letting $y = (\phi^1(x), \ldots, \phi^n(x)) \in K'$, where $K' \doteq \phi(K)$, by standard transformation rules
    \begin{align*}
    (Au)_y(\partial_{y^a},\partial_{y^b})
        & = \partial_{y^a}(\phi^{-1})^i \partial_{y^b} (\phi^{-1})^j (u_{ij}  \circ \phi^{-1}(y)).
    \end{align*}
    Since $u \in W^{l,p}(T_2\Omega)$ implies that $u_{ij} \in W^{l,p}(K)$, then it follows from \cref{lemma: Adams diffeo lemma} that
    \begin{align*}
        (u_{ij}  \circ \phi^{-1}) \in W^{l,p}(K')
    \end{align*}
    provided that $l \leq k+1$ and $p \leq q$, which holds true by hypothesis. Since we also know that $\phi^{-1}\in W^{k+1,q}(K')$ by \cref{LemmaInvFunctThmSobReg}, then the Sobolev multiplication properties of \cite[Theorem 2.2]{avalos2024sobolev} imply
    \begin{align*}
        W^{k,q}(K') \otimes W^{k, q}(K') \otimes W^{l,p}(K') \hookrightarrow W^{l,p}(K')
    \end{align*}
    under our hypotheses. Thus $(Au)_{ab} \in W^{l,p}(K')$ and therefore $Au \in W^{l,p}(T_sK')$ together with an estimate
    \begin{align*}
        \Vert Au \Vert_{W^{l,p}(K')}&\leq C\Vert u \Vert_{W^{l,p}(K)} \,,
    \end{align*}
    where $C=C(n,p,q,l,K,\phi)$ is provided by the above multiplication constants and the estimates of \cref{lemma: Adams diffeo lemma}. Since $K$ (and thus $K'$) was arbitrary, we conclude that $Au \in W^{l,p}_{loc}(T_s\Omega')$ and by interchanging the role of $\phi$ and $\phi^{-1}$ we obtain the second estimate of the Lemma.
\end{proof}

\vspace{0.5cm}

\section{Weighted estimates revisited}
\label{Weighted estimates revisited}

In this appendix we record a weighted elliptic estimate on AE manifolds very close to that presented in \cite[Theoren A.3]{avalos2024sobolev}. The only difference is that in our case is that we only assume the manifold to be $W^{1,q}_{\tau}$-AE with $q>n$, rather than $W^{2,q}_{\tau}$-AE with $q>\tfrac{n}{2}$, but assume on the other hand a better local regularity of the metric $g\in W^{2,q}_{loc}(T_2M)$ and the a-priori solution $u\in W^{2,p}_{loc}(M)\cap L^{p}_{\delta}(M)$. It should be mentioned that a stronger statement than ours is claimed in \cite[Proposition 1.6]{Bartnik86}, although there are reasons to be skeptical about its validity in the generality stated there. We refer the reader to \cite[Appendix A]{avalos2024sobolev} for a discussion on this matter. 
\begin{theorem}
    \label{BartniksProp1.6}
    Let $(M,g)$ be a $W^{1,q}_{\tau}$-AE manifold such that $g\in W^{2,q}_{loc}(M)$ with $q>n$ and $\tau<0$. If $u\in W^{2,p}_{loc}(M)\cap L^{p}_{\delta}(M)$ and $\Delta_gu\in L^p_{\delta-2}(M)$ for some $1<p\leq q$ and $\delta\in\nR$, then $u\in W^{2,p}_{\delta}(M)$ and 
    \begin{align*}
        \Vert u\Vert_{W^{2,p}_{\delta}(M)}\leq C\left(\Vert \Delta_gu\Vert_{L^p_{\delta-2}(M)} + \Vert u\Vert_{L^p_{\delta}(M)} \right).
    \end{align*}
\end{theorem}
The proof of \cref{BartniksProp1.6} follows almost line by line as in \cite[Theorem A.3]{avalos2024sobolev}, with Sobolev multiplications applied differently in certain estimates.

\end{appendix}

\vspace{1cm}
\printbibliography[]

@ARTICLE{Holst1,
    author = {Holst, M. and Nagy, G. and Tsogtgerel, G.},
    title = {Rough solutions of the Einstein constraint equations on Closed Manifolds without Near CMC Conditions},
    journal = {Comm. Math. Phys.},
    fjournal = {Communications in Mathematical Physics},
    volume = {288},
    year = {2009},
    pages = {547-613},
}

@ARTICLE{Holst-LichCompactWithBoundary,
    author = {Holst, M. and Tsogtgerel, G.},
    title = {The Lichnerowicz Equation on Compact Manifolds with Boundary},
    journal = {Class. Quantum Grav.},
    fjournal = {Classical and Quantum Gravity},
    volume = {30},
    year = {2013},
    pages = {205011},
    issn = {},
    mrclass = {},
    mrnumber = {},
    mrreviewer = {},
    doi = {},
    url = {},
    zbl = {1276.83007},
}

@ARTICLE{MaxwellFarCMC,
    author = {Maxwell, D.},
    title = {A class of solutions of the vacuum Einstein constraint equations with freely specified mean curvature},
    journal = {Math. Res. Lett.},
    fjournal = {Mathematical Research Letters},
    volume = {16},
    year = {2009},
    pages = {627--645},
    issn = {},
    mrclass = {},
    mrnumber = {},
    mrreviewer = {},
    doi = {},
    url = {},
    zbl = {1187.83022},
}

@article{avalos2024sobolev,
 author = {R. Avalos},
 title = {Sobolev regularity of compactified 3-manifolds and the {ADM} {Center} of {Mass}},
 Journal = {Ann. of PDE},
 Volume = {12},
 Number = {35},
 year = {2026},
}

@Article{ZhangW1pAubin,
 Author = {Zhang, H.},
 Title = {The {Yamabe} problem for distributional curvature},
 FJournal = {The Journal of Geometric Analysis},
 Journal = {J. Geom. Anal.},
 Volume = {33},
 Number = {10},
 Pages = {33},
 Note = {},
 Year = {2023},
 Language = {},
 zbMATH = {7723922},
 Zbl = {1520.53025}
}

@Article{Yamabe,
 Author = {Yamabe, H.},
 Title = {On a deformation of {Riemannian} structures on compact manifolds},
 FJournal = {Osaka Mathematical Journal},
 Journal = {Osaka Math. J.},
 Volume = {12},
 Pages = {21--37},
 Year = {1960},
 zbMATH = {3157588},
 Zbl = {0096.37201}
}

@Article{Trudinger,
 Author = {Trudinger, N. S.},
 Title = {Remarks concerning the conformal deformation of {Riemannian} structures on compact manifolds},
 FJournal = {Annali della Scuola Normale Superiore di Pisa. Scienze Fisiche e Matematiche. III. Ser},
 Journal = {Ann. Sc. Norm. Super. Pisa, Sci. Fis. Mat., III. Ser.},
 Volume = {22},
 Pages = {265--274},
 Year = {1968},
 zbMATH = {3256308},
 Zbl = {0159.23801}
}

@Article{Aubin1,
 Author = {Aubin, T.},
 Title = {Equations diff{\'e}rentielles non lin{\'e}aires et probl{\`e}me de {Yamabe} concernant la courbure scalaire},
 FJournal = {Journal de Math{\'e}matiques Pures et Appliqu{\'e}es. Neuvi{\`e}me S{\'e}rie},
 Journal = {J. Math. Pures Appl. (9)},
 Volume = {55},
 Pages = {269--296},
 Year = {1976},
 zbMATH = {3525890},
 Zbl = {0336.53033}
}

@Article{SchoenYamabe,
 Author = {Schoen, R.},
 Title = {Conformal deformation of a {Riemannian} metric to constant scalar curvature},
 FJournal = {Journal of Differential Geometry},
 Journal = {J. Differ. Geom.},
 Volume = {20},
 Pages = {479--495},
 Year = {1984},
 zbMATH = {3921414},
 Zbl = {0576.53028}
}

@Book{TaylorToolsForPDEs,
 Author = {Taylor, M. E.},
 Title = {Tools for {PDE}. {Pseudodifferential} operators, paradifferential operators, and layer potentials},
 FSeries = {Mathematical Surveys and Monographs},
 Series = {Math. Surv. Monogr.},
 Volume = {81},
 ISBN = {},
 Year = {2000},
 Publisher = {Providence, RI: American Mathematical Society (AMS)},
 Language = {},
 zbMATH = {1470598},
 Zbl = {0963.35211}
}

@Article{Lee-Parker,
 Author = {Lee, J. M. and Parker, T. H.},
 Title = {The {Yamabe} problem},
 FJournal = {Bulletin of the American Mathematical Society. New Series},
 Journal = {Bull. Am. Math. Soc., New Ser.},
 Volume = {17},
 Pages = {37--91},
 Year = {1987},
 zbMATH = {4030435},
 Zbl = {0633.53062}
}

@Book{AubinBook,
 Author = {Aubin, T.},
 Title = {Some nonlinear problems in {Riemannian} geometry},
 FSeries = {Springer Monographs in Mathematics},
 Series = {Springer Monogr. Math.},
 ISSN = {1439-7382},
 ISBN = {},
 Year = {1998},
 Publisher = {Berlin: Springer},
 Language = {},
 zbMATH = {1179494},
 Zbl = {0896.53003}
}

@Book{SchoenYauBook,
 Author = {Schoen, R. and Yau, S-T.},
 Title = {Lectures on differential geometry},
 FSeries = {Conference Proceedings and Lecture Notes in Geometry and Topology},
 Series = {Conf. Proc. Lect. Notes Geom. Topol.},
 Volume = {1},
 ISBN = {},
 Year = {1994},
 Publisher = {Cambridge, MA: International Press},
 Language = {},
 zbMATH = {770462},
 Zbl = {0830.53001}
}

@Book{GT,
 Author = {Gilbarg, D. and Trudinger, N. S.},
 Title = {Elliptic partial differential equations of second order},
 Edition = {Reprint of the 1998 ed.},
 FSeries = {Classics in Mathematics},
 Series = {Class. Math.},
 ISSN = {1431-0821},
 ISBN = {},
 Year = {2001},
 Publisher = {Berlin: Springer},
 Language = {},
 zbMATH = {1554166},
 Zbl = {1042.35002}
}

@Article{Maxwell1,
 Author = {Maxwell, D.},
 Title = {Rough solutions of the {Einstein} constraint equations on compact manifolds},
 FJournal = {Journal of Hyperbolic Differential Equations},
 Journal = {J. Hyperbolic Differ. Equ.},
 Volume = {2},
 Number = {2},
 Pages = {521--546},
 Year = {2005},
 zbMATH = {2202330},
 Zbl = {1076.58021}
}

@article{Maxwell0,
 author = {Maxwell, D.},
 title = {Solutions of the {Einstein} constraint equations with apparent horizon boundaries},
 fjournal = {Communications in Mathematical Physics},
 journal = {Commun. Math. Phys.},
 volume = {253},
 number = {3},
 pages = {561--583},
 year = {2005},
 zbMATH = {2156007},
 Zbl = {1065.83011}
}

@Article{MaxDil,
 Author = {Dilts, J. and Maxwell, D.},
 Title = {Yamabe classification and prescribed scalar curvature in the asymptotically {Euclidean} setting},
 FJournal = {Communications in Analysis and Geometry},
 Journal = {Commun. Anal. Geom.},
 Volume = {26},
 Number = {5},
 Pages = {1127--1168},
 Year = {2018},
 zbMATH = {7019862},
 Zbl = {1408.53046}
}

@Article{LeeLeF,
 Author = {Lee, D. A. and LeFloch, P. G.},
 Title = {The positive mass theorem for manifolds with distributional curvature},
 FJournal = {Communications in Mathematical Physics},
 Journal = {Commun. Math. Phys.},
 Volume = {339},
 Number = {1},
 Pages = {99--120},
 Year = {2015},
 zbMATH = {6466593},
 Zbl = {1330.53062}
}

@Article{Bartnik86,
 Author = {Bartnik, R.},
 Title = {The mass of an asymptotically flat manifold},
 FJournal = {Communications on Pure and Applied Mathematics},
 Journal = {Commun. Pure Appl. Math.},
 Volume = {39},
 Pages = {661--693},
 Year = {1986},
 zbMATH = {3964921},
 Zbl = {0598.53045}
}

@Article{DeTurck_Kazdan,
 Author = {DeTurck, D. M. and Kazdan, J. L.},
 Title = {Some regularity theorems in {Riemannian} geometry},
 FJournal = {Annales Scientifiques de l'{\'E}cole Normale Sup{\'e}rieure. Quatri{\`e}me S{\'e}rie},
 Journal = {Ann. Sci. {\'E}c. Norm. Sup{\'e}r. (4)},
 Volume = {14},
 Pages = {249--260},
 Year = {1981},
 zbMATH = {3764670},
 Zbl = {0486.53014}
}

@Article{Sabitov-Shefel,
 Author = {Sabitov, I. Kh. and Shefel', S. Z.},
 Title = {The connections between the order of smoothness of a surface and its metric},
 FJournal = {Siberian Mathematical Journal},
 Journal = {Sib. Math. J.},
 ISSN = {},
 Volume = {17},
 Pages = {687--694},
 Year = {1977},
 Language = {},
 DOI = {},
 zbMATH = {3600723},
 Zbl = {0386.53014}
}

@misc{MaxHolTso2,
      title={A Scaling Approach to Elliptic Theory for Geometrically-Natural Differential Operators with Sobolev-Type Coefficients}, 
      author={M. Holst and D. Maxwell and G. Tsogtgerel},
      year={2023},
      note={Preprint, arXiv.2306.15842},
}

@Article{MazzeoPollackUhlenbeck,
 Author = {Mazzeo, R. and Pollack, D. and Uhlenbeck, K.},
 Title = {Moduli spaces of singular {Yamabe} metrics},
 FJournal = {Journal of the American Mathematical Society},
 Journal = {J. Am. Math. Soc.},
 Volume = {9},
 Number = {2},
 Pages = {303--344},
 Year = {1996},
 zbMATH = {868354},
 Zbl = {0849.58012}
}

@Article{CaffarelliGidasSpruck,
 Author = {Caffarelli, L. A. and Gidas, B. and Spruck, J.},
 Title = {Asymptotic symmetry and local behavior of semilinear elliptic equations with critical {Sobolev} growth},
 FJournal = {Communications on Pure and Applied Mathematics},
 Journal = {Commun. Pure Appl. Math.},
 Volume = {42},
 Number = {3},
 Pages = {271--297},
 Year = {1989},
 zbMATH = {4150823},
 Zbl = {0702.35085}
}

@Article{HanLiLi,
 Author = {Han, Q. and Li, X. and Li, Y.},
 Title = {Asymptotic expansions of solutions of the {Yamabe} equation and the {{\(\sigma_{k}\)}}-{Yamabe} equation near isolated singular points},
 FJournal = {Communications on Pure and Applied Mathematics},
 Journal = {Commun. Pure Appl. Math.},
 Volume = {74},
 Number = {9},
 Pages = {1915--1970},
 Year = {2021},
 zbMATH = {7396342},
 Zbl = {1482.35070}
}

@Article{MarquesSingular,
 Author = {Marques, F. C.},
 Title = {Isolated singularities of solutions to the {Yamabe} equation},
 FJournal = {Calculus of Variations and Partial Differential Equations},
 Journal = {Calc. Var. Partial Differ. Equ.},
 Volume = {32},
 Number = {3},
 Pages = {349--371},
 Year = {2008},
 zbMATH = {5284780},
 Zbl = {1143.35323}
}

@Article{MazzeoPacard,
 Author = {Mazzeo, R. and Pacard, F.},
 Title = {Constant scalar curvature metrics with isolated singularities},
 FJournal = {Duke Mathematical Journal},
 Journal = {Duke Math. J.},
 Volume = {99},
 Number = {3},
 Pages = {353--418},
 Year = {1999},
 zbMATH = {1425247},
 Zbl = {0945.53024}
}

@Article{XiongZhang,
 Author = {Xiong, J. and Zhang, L.},
 Title = {Isolated singularities of solutions to the {Yamabe} equation in dimension 6},
 FJournal = {IMRN. International Mathematics Research Notices},
 Journal = {Int. Math. Res. Not.},
 Volume = {2022},
 Number = {12},
 Pages = {9571--9597},
 Year = {2022},
 zbMATH = {7542596},
 Zbl = {1491.35243}
}

@Article{DruetLaurain,
 Author = {Druet, O. and Laurain, P.},
 Title = {Stability of the {Poho{\v{z}}aev} obstruction in dimension 3},
 FJournal = {Journal of the European Mathematical Society (JEMS)},
 Journal = {J. Eur. Math. Soc. (JEMS)},
 Volume = {12},
 Number = {5},
 Pages = {1117--1149},
 Year = {2010},
 zbMATH = {5779483},
 Zbl = {1210.35105}
}

@Article{TrudingerMeasurableCoef,
 Author = {Trudinger, N. S.},
 Title = {Linear elliptic operators with measurable coefficients},
 FJournal = {Annali della Scuola Normale Superiore di Pisa. Scienze Fisiche e Matematiche. III. Ser},
 Journal = {Ann. Sc. Norm. Super. Pisa, Sci. Fis. Mat., III. Ser.},
 Volume = {27},
 Pages = {265--308},
 Year = {1973},
 zbMATH = {3439318},
 Zbl = {0279.35025}
}

@Article{Taylor_ConfFlat,
 Author = {Taylor, M.},
 Title = {Existence and regularity of isometries},
 FJournal = {Transactions of the American Mathematical Society},
 Journal = {Trans. Am. Math. Soc.},
 ISSN = {},
 Volume = {358},
 Number = {6},
 Pages = {2415--2423},
 Year = {2006},
 zbMATH = {5010028},
 Zbl = {1156.53310}
}

@Book{DruetHebeyRobert,
 Author = {Druet, O. and Hebey, E. and Robert, F.},
 Title = {Blow-up theory for elliptic {PDEs} in {Riemannian} geometry},
 FSeries = {Mathematical Notes (Princeton)},
 Series = {Math. Notes},
 Volume = {45},
 ISBN = {},
 Year = {2004},
 Publisher = {Princeton, NJ: Princeton University Press},
 Language = {},
 zbMATH = {2116063},
 Zbl = {1059.58017}
}

@Book{AdamsFournierBook,
Author = {Adams, R. A. and Fournier, J. J. F.},
 Title = {Sobolev spaces},
 Edition = {2nd ed.},
 Series = {Pure Appl. Math., Academic Press},
 ISSN = {0079-8169},
 Volume = {140},
 ISBN = {},
 Year = {2003},
 Publisher = {New York, NY},
 Language = {},
 zbMATH = {2208228},
 Zbl = {1098.46001}
}

@Book{Kato,
 Author = {Kato, T.},
 Title = {Perturbation theory for linear operators},
 FSeries = {Grundlehren der Mathematischen Wissenschaften},
 Series = {Grundlehren Math. Wiss.},
 ISSN = {0072-7830},
 Volume = {132},
 Year = {1966},
 Publisher = {Springer, Cham},
 Language = {},
 zbMATH = {3239038},
 Zbl = {0148.12601}
}

@Article{HolstBehzadanSobSpaces,
AUTHOR = {Behzadan, A. and Holst, M.},
TITLE = {Sobolev-Slobodeckij Spaces on Compact Manifolds, Revisited},
JOURNAL = {Mathematics},
VOLUME = {10},
YEAR = {2022},
NUMBER = {3},
ARTICLE-NUMBER = {522},
ISSN = {},
}

@Book{Taylor3,
 Author = {Taylor, M. E.},
 Title = {Partial differential equations. {III}: {Nonlinear} equations.},
 Edition = {2nd ed.},
 FSeries = {Applied Mathematical Sciences},
 Series = {Appl. Math. Sci.},
 ISSN = {},
 Volume = {117},
 ISBN = {},
 Year = {2011},
 Publisher = {New York, NY: Springer},
 zbMATH = {5771970},
 Zbl = {1206.35004}
}

@Book{Adams1st,
 Author = {Adams, R. A.},
 Title = {Sobolev spaces},
 FSeries = {Pure and Applied Mathematics},
 Series = {Pure Appl. Math., Academic Press},
 ISSN = {0079-8169},
 Volume = {65},
 Year = {1975},
 Publisher = {New York, NY},
 Language = {},
 zbMATH = {3491650},
 Zbl = {0314.46030}
}

@Article{Calderon1,
 Author = {Calder{\'o}n, A. P.},
 Title = {Intermediate spaces and interpolation, the complex method},
 FJournal = {Studia Mathematica},
 Journal = {Stud. Math.},
 Volume = {24},
 Pages = {113--190},
 Year = {1964},
 Language = {},
 zbMATH = {3324177},
 Zbl = {0204.13703}
}

@Book{Triebel1,
 Author = {Triebel, H.},
 Title = {Interpolation theory, function spaces, differential operators.},
 Edition = {2nd ed.},
 ISBN = {},
 Year = {1995},
 Publisher = {Leipzig: Barth},
 Language = {},
 zbMATH = {762608},
 Zbl = {0830.46028}
}

@Book{Taylor1,
 Author = {Taylor, M. E.},
 Title = {Partial differential equations. {I}: {Basic} theory},
 Edition = {2nd ed.},
 FSeries = {Applied Mathematical Sciences},
 Series = {Appl. Math. Sci.},
 ISSN = {},
 Volume = {115},
 ISBN = {},
 Year = {2011},
 Publisher = {New York, NY: Springer},
 zbMATH = {5771951},
 Zbl = {1206.35002}
}

@Book{HebeyBook,
 Author = {Hebey, E.},
 Title = {Sobolev spaces on {Riemannian} manifolds},
 FSeries = {Lecture Notes in Mathematics},
 Series = {Lect. Notes Math.},
 ISSN = {0075-8434},
 Volume = {1635},
 ISBN = {},
 Year = {1996},
 Publisher = {Berlin: Springer},
 Language = {},
 zbMATH = {945657},
 Zbl = {0866.58068}
}

@book{Morrey,
 author = {Morrey, Charles Bradfield jun.},
 title = {Multiple integrals in the calculus of variations},
 fseries = {Grundlehren der Mathematischen Wissenschaften},
 series = {Grundlehren Math. Wiss.},
 issn = {0072-7830},
 volume = {130},
 year = {1966},
 publisher = {Springer, Cham},
 language = {},
 zbMATH = {3229944},
 Zbl = {0142.38701}
}

@article{Brezis1,
 author = {Brezis, H.},
 title = {On a conjecture of {J}. {Serrin}},
 fjournal = {Atti della Accademia Nazionale dei Lincei. Classe di Scienze Fisiche, Matematiche e Naturali. Serie IX. Rendiconti Lincei. Matematica e Applicazioni},
 journal = {Atti Accad. Naz. Lincei, Cl. Sci. Fis. Mat. Nat., IX. Ser., Rend. Lincei, Mat. Appl.},
 volume = {19},
 number = {4},
 pages = {335--338},
 year = {2008},
 language = {},
 zbMATH = {5787947},
 Zbl = {1197.35079}
}

@article{Serrin,
 author = {Serrin, J.},
 title = {Pathological solutions of elliptic differential equations},
 fjournal = {Annali della Scuola Normale Superiore di Pisa. Scienze Fisiche e Matematiche. III. Ser},
 journal = {Ann. Sc. Norm. Super. Pisa, Sci. Fis. Mat., III. Ser.},
 issn = {},
 volume = {18},
 pages = {385--387},
 year = {1964},
 language = {},
 zbMATH = {3229911},
 Zbl = {0142.37601}
}

@article{Salo-pharm,
 author = {Julin, V. and Liimatainen, T. and Salo, M.},
 title = {{{\(p\)}}-harmonic coordinates for {H{\"o}lder} metrics and applications},
 fjournal = {Communications in Analysis and Geometry},
 journal = {Commun. Anal. Geom.},
 volume = {25},
 number = {2},
 pages = {395--430},
 year = {2017},
 language = {},
 zbMATH = {6823218},
 Zbl = {1380.53042}
}

@article{GromovC0Scalar,
 author = {Gromov, M.},
 title = {Dirac and {Plateau} billiards in domains with corners},
 fjournal = {Central European Journal of Mathematics},
 journal = {Cent. Eur. J. Math.},
 volume = {12},
 number = {8},
 pages = {1109--1156},
 year = {2014},
 zbMATH = {6310170},
 Zbl = {1315.53027}
}

@incollection{GromovLectures,
 author = {Gromov, M.},
 title = {Four lectures on scalar curvature},
 booktitle = {Perspectives in scalar curvature. In 2 volumes},
 isbn = {},
 pages = {1--514},
 year = {2023},
 publisher = {Singapore: World Scientific},
 zbMATH = {7733259},
 Zbl = {1532.53003}
}

@article{GromovConvexPolytopes,
 author = {Gromov, M.},
 title = {Convex polytopes, dihedral angles, mean curvature and scalar curvature},
 fjournal = {Discrete \& Computational Geometry},
 journal = {Discrete Comput. Geom.},
 volume = {72},
 number = {2},
 pages = {849--875},
 year = {2024},
 zbMATH = {7931803}
}

@article{BamlerGromovv,
 author = {Bamler, R. H.},
 title = {A {Ricci} flow proof of a result by {Gromov} on lower bounds for scalar curvature},
 fjournal = {Mathematical Research Letters},
 journal = {Math. Res. Lett.},
 volume = {23},
 number = {2},
 pages = {325--337},
 year = {2016},
 zbMATH = {6609371},
 Zbl = {1373.53043}
}

@article{Burkhardt-GuimADM,
 author = {Burkhardt-Guim, P.},
 title = {{ADM} mass for {{\(C^0\)}} metrics and distortion under {Ricci}-{DeTurck} flow},
 fjournal = {Journal f{\"u}r die Reine und Angewandte Mathematik},
 journal = {J. Reine Angew. Math.},
 volume = {806},
 pages = {187--245},
 year = {2024},
 zbMATH = {7785471},
 Zbl = {1536.53185}
}

@article{ChuLee,
 author = {Chu, J. and Lee, M-C.},
 title = {Ricci-{DeTurck} flow from rough metrics and applications to scalar curvature problems},
 fjournal = {Journal of Functional Analysis},
 journal = {J. Funct. Anal.},
 volume = {289},
 number = {2},
 pages = {41},
 note = {},
 year = {2025},
 zbMATH = {8017519}
}

@article{HuiskenMass1,
     author={Huisken, G.},
     title={An isoperimetric concept for mass and quasilocal mass},
     journal={Oberwolfach Reports, European Mathematical Society (EMS), Z\"urich},
   date={2006},
   volume={3},
   number={1},
   pages={87--88}
}

@article{HuiskenMass3,
     author={Huisken, G.},
     title={An isoperimetric concept for the mass in general relativity},
     journal={Oberwolfach Reports, European Mathematical Society (EMS), Z\"urich},
   date={2006},
   volume={3},
   number={3},
   pages={1898--1899}
}

@article{ChodoshEichmairShiYu,
 author = {Chodosh, O. and Eichmair, M. and Shi, Y. and Yu, H.},
 title = {Isoperimetry, scalar curvature, and mass in asymptotically flat {Riemannian} 3-manifolds},
 fjournal = {Communications on Pure and Applied Mathematics},
 journal = {Commun. Pure Appl. Math.},
 volume = {74},
 number = {4},
 pages = {865--905},
 year = {2021},
 zbMATH = {7363255},
 Zbl = {1470.53036}
}

@article{JauregiLee,
 author = {Jauregui, J. L. and Lee, D. A.},
 title = {Lower semicontinuity of mass under {{\(C^0\)}} convergence and {Huisken}'s isoperimetric mass},
 fjournal = {Journal f{\"u}r die Reine und Angewandte Mathematik},
 journal = {J. Reine Angew. Math.},
 volume = {756},
 pages = {227--257},
 year = {2019},
 zbMATH = {7126239},
 Zbl = {1430.53052}
}

@article{AubinIsoperim,
author = {T. Aubin},
title = {{Problèmes isopérimétriques et espaces de Sobolev}},
volume = {11},
journal = {Journal of Differential Geometry},
number = {4},
publisher = {Lehigh University},
pages = {573 -- 598},
year = {1976},
}

@ARTICLE{McOwen,
    author = {McOwen, R. C.},
    title = {The behavior of the laplacian on weighted sobolev spaces},
    journal = {Communications on Pure and Applied Mathematics},
    volume = {32},
    number = {6},
    pages = {783-795},
    year = {1979},
    mrnumber = {},
    zbl = {0426.35029},
}

@book{DoCarmo,
 author = {do Carmo, M. P.},
 title = {Riemannian geometry. {Translated} from the {Portuguese} by {Francis} {Flaherty}},
 isbn = {},
 year = {1992},
 publisher = {Boston, MA etc.: Birkh{\"a}user},
 language = {},
 zbMATH = {52737},
 Zbl = {0752.53001}
}

@ARTICLE{NirenbergWalker,
    title = {The null spaces of elliptic partial differential operators in $\mathbb{R}^n$},
    journal = {Journal of Mathematical Analysis and Applications},
    volume = {42},
    number = {2},
    pages = {271-301},
    year = {1973},
    issn = {},
    author = {Nirenberg, L. and Walker, H. F.},
    mrnumber = {0320821},
    zbl = {0272.35029},
}

@ARTICLE{Lockhart,
    author = {Lockhart, R. B.},
    title = {Fredholm properties of a class of elliptic operators on non-compact manifolds},
    volume = {48},
    journal = {Duke Mathematical Journal},
    number = {1},
    publisher = {Duke University Press},
    pages = {289--312},
    year = {1981},
    mrnumber = {0610188},
    zbl = {0486.35027},
}

@ARTICLE{CB-C,
    author = {Choquet-Bruhat, Y. and Christodoulou, D.},
    title = {Elliptic Systems in $H_{s,\delta}$ Spaces on Manifolds Which are Euclidean at Infinity},
    journal = {Acta Mathematica},
    fjournal = {Acta Mathematica},
    volume = {146},
    year = {1981},
    pages = {129-150},
    mrnumber = {0594629},
    zbl = {0484.58028},
}

@article{Talenti,
 author = {Talenti, G.},
 title = {Best constant in {Sobolev} inequality},
 fjournal = {Annali di Matematica Pura ed Applicata. Serie Quarta},
 journal = {Ann. Mat. Pura Appl. (4)},
 volume = {110},
 pages = {353--372},
 year = {1976},
 zbMATH = {3549635},
 Zbl = {0353.46018}
}

@article{Madani,
 author = {Madani, Farid},
 title = {The {Yamabe} problem with singularities},
 fjournal = {Bulletin des Sciences Math{\'e}matiques},
 journal = {Bull. Sci. Math.},
 volume = {132},
 number = {7},
 pages = {575--591},
 year = {2008},
}

@misc{FregugliaMalchiodiMalizia,
 author = {Freguglia, Mattia and Malchiodi, Andrea and Malizia, Francesco},
 title = {Min-max theory and {Yamabe} metrics on conical four-manifolds},
 year = {2025},
 howpublished = {Preprint, {arXiv}:2508.02667 (2025)},
 arXiv = {arXiv:2508.02667}
}

@misc{FregugliaMalchiodi,
 author = {Freguglia, Mattia and Malchiodi, Andrea},
 title = {Yamabe metrics on conical manifolds},
 year = {2024},
 howpublished = {Preprint, {arXiv}:2402.05927] (2024)},
 arXiv = {arXiv:2402.05927}
}

@article{AkutagawaBotvinnikCylindrical,
 author = {Akutagawa, K. and Botvinnik, B.},
 title = {Yamabe metrics on cylindrical manifolds},
 fjournal = {Geometric and Functional Analysis. GAFA},
 journal = {Geom. Funct. Anal.},
 volume = {13},
 number = {2},
 pages = {259--333},
 year = {2003},
 zbMATH = {5056109},
 Zbl = {1161.53344}
}

@article{AkutagawaMondello,
 author = {Akutagawa, K. and Mondello, I.},
 title = {Non-existence of {Yamabe} minimizers on singular spheres},
 fjournal = {The Journal of Geometric Analysis},
 journal = {J. Geom. Anal.},
 volume = {32},
 number = {7},
 pages = {20},
 year = {2022},
 zbMATH = {7523363},
 Zbl = {1492.53048}
}

@article{AkutagawaCarronMazzeo,
 author = {Akutagawa, K. and C., Gilles and Mazzeo, R.},
 title = {The {Yamabe} problem on stratified spaces},
 fjournal = {Geometric and Functional Analysis. GAFA},
 journal = {Geom. Funct. Anal.},
 volume = {24},
 number = {4},
 pages = {1039--1079},
 year = {2014},
 zbMATH = {6381488},
 Zbl = {1310.53034}
}

@article{AkutagawaBotvinnikOrbifold,
 author = {Akutagawa, K. and Botvinnik, B.},
 title = {The {Yamabe} invariants of orbifolds and cylindrical manifolds, and {{\(L^2\)}}-harmonic spinors.},
 fjournal = {Journal f{\"u}r die Reine und Angewandte Mathematik},
 journal = {J. Reine Angew. Math.},
 volume = {574},
 pages = {121--146},
 year = {2004},
 zbMATH = {2104131},
 Zbl = {1055.53035}
}

@article{AkutagawaOrbifold,
 author = {Akutagawa, K.},
 title = {Computations of the orbifold {Yamabe} invariant},
 fjournal = {Mathematische Zeitschrift},
 journal = {Math. Z.},
 volume = {271},
 number = {3-4},
 pages = {611--625},
 year = {2012},
 zbMATH = {6065393},
 Zbl = {1248.53034}
}

@article{Mondello2017,
 author = {Mondello, I.},
 title = {The local {Yamabe} constant of {Einstein} stratified spaces},
 fjournal = {Annales de l'Institut Henri Poincar{\'e}. Analyse Non Lin{\'e}aire},
 journal = {Ann. Inst. Henri Poincar{\'e}, Anal. Non Lin{\'e}aire},
 volume = {34},
 number = {1},
 pages = {249--275},
 year = {2017},
 zbMATH = {6672844},
 Zbl = {1355.53027}
}

@article{ViaclovskyOrbifoldYamabe,
 author = {Viaclovsky, J.},
 title = {Monopole metrics and the orbifold {Yamabe} problem},
 fjournal = {Annales de l'Institut Fourier},
 journal = {Ann. Inst. Fourier},
 volume = {60},
 number = {7},
 pages = {2503--2543},
 year = {2010},
 zbMATH = {5954680},
 Zbl = {1227.53060}
}

@article{ViaclovskyJu,
 author = {Ju, T. and Viaclovsky, J.},
 title = {Conformally prescribed scalar curvature on orbifolds},
 fjournal = {Communications in Mathematical Physics},
 journal = {Commun. Math. Phys.},
 volume = {398},
 number = {2},
 pages = {877--923},
 year = {2023},
 zbMATH = {7658239},
 Zbl = {1514.53081}
}

@article{Finn,
 author = {Finn, D. L.},
 title = {On the negative case of the singular {Yamabe} problem},
 fjournal = {The Journal of Geometric Analysis},
 journal = {J. Geom. Anal.},
 volume = {9},
 number = {1},
 pages = {73--92},
 year = {1999},
 zbMATH = {1542086},
 Zbl = {1003.58020}
}

@article{BrendleYamabeFlow1,
 author = {Brendle, S.},
 title = {Convergence of the {Yamabe} flow for arbitrary initial energy},
 fjournal = {Journal of Differential Geometry},
 journal = {J. Differ. Geom.},
 volume = {69},
 number = {2},
 pages = {217--278},
 year = {2005},
 zbMATH = {5004290},
 Zbl = {1085.53028}
}

@article{BrendleYamabeFlow2,
 author = {Brendle, S.},
 title = {Convergence of the {Yamabe} flow in dimension 6 and higher},
 fjournal = {Inventiones Mathematicae},
 journal = {Invent. Math.},
 volume = {170},
 number = {3},
 pages = {541--576},
 year = {2007},
 zbMATH = {5220759},
 Zbl = {1130.53044}
}

@article{Bahri93,
 author = {Bahri, A.},
 title = {Another proof of the {Yamabe} conjecture for locally conformally flat manifolds},
 fjournal = {Nonlinear Analysis. Theory, Methods \& Applications},
 journal = {Nonlinear Anal., Theory Methods Appl.},
 volume = {20},
 number = {10},
 pages = {1261--1278},
 year = {1993},
 zbMATH = {270170},
 Zbl = {0782.53027}
}

@incollection{Brezis-Bahri,
 author = {Bahri, A. and Br{\'e}zis, H.},
 title = {Nonlinear elliptic equations on {Riemannian} manifolds with the {Sobolev} critical exponent},
 booktitle = {Topics in geometry. In memory of Joseph D'Atri},
 pages = {1--100},
 year = {1996},
 publisher = {Boston, MA: Birkh{\"a}user},
 zbMATH = {912646},
 Zbl = {0863.35037}
}

@misc{RoyoThesis,
 author = {Royo Abrego, A.},
 title = {Coupled geometric flows and regularity problems in conformal geometry},
 year = {2026},
 howpublished = {PhD Thesis, University Tübingen}
}

@article{TaylorBochner,
 author = {Taylor, M.},
 title = {B{\^o}cher's theorem with rough coefficients},
 fjournal = {Potential Analysis},
 journal = {Potential Anal.},
 volume = {56},
 number = {1},
 pages = {65--86},
 year = {2022},
}

@article{AvalosLaurainMarque,
 author = {Avalos, R. and Laurain, P. and Marque, N.},
 title = {Rigidity theorems for asymptotically {Euclidean} {{\(Q\)}}-singular spaces},
 fjournal = {Advances in Mathematics},
 journal = {Adv. Math.},
 volume = {490},
 pages = {34},
 year = {2026},
}
\clearpage

\end{document}